\documentclass{amsart}
\usepackage[utf8]{inputenc}

\usepackage{amsmath, amssymb, amsfonts, amsthm}
\usepackage{enumerate}
\usepackage{bbold}
\usepackage{color}
\usepackage{xcolor}
\usepackage{esvect}
\usepackage{mathtools}
\usepackage{graphicx}
\usepackage{hyperref}
\usepackage{longtable}
\usepackage{fullpage}
\usepackage{rotating}
\usepackage{blkarray, bigstrut}
\usepackage{tikz}
\usepackage{tikz-cd}
\usepackage[boxsize = .4em,aligntableaux=center]{ytableau}
\theoremstyle{plain}
\newtheorem{thm}{Theorem}[section]
\newtheorem{lem}[thm]{Lemma}
\newtheorem{prop}[thm]{Proposition}
\newtheorem{cor}[thm]{Corollary}

\newcommand{\att}[2]{\raisebox{-.5\height}{ \includegraphics[scale = #2]{diagrams/#1.pdf}}}
\theoremstyle{definition}
\newtheorem{defn}[thm]{Definition}
\newtheorem{ex}[thm]{Example}
\newtheorem{rmk}[thm]{Remark}

\newtheorem{conj}[thm]{Conjecture}

\newcommand{\Hom}{\operatorname{Hom}}
\newcommand{\End}{\operatorname{End}}

\newtheoremstyle{named}{}{}{\itshape}{}{\bfseries}{.}{.5em}{#1 \thmnote{(#3) }}
\theoremstyle{named}

\title{Classification of 1-super-transitive quantum subgroups in type A}
\author{Cain Edie-Michell and Jacques Katumba}
\date{}

\begin{document}

\begin{abstract}
    We define a notion of super-transitivity for \`etale algebra objects $A \in \mathcal{C}(\mathfrak{sl}_N, k)$. This definition is a direct analogue of the notion of super-transitivity for subfactors, and measures at what depth the first ``new stuff'' appears in the category of $A$-modules internal to $\mathcal{C}(\mathfrak{sl}_N, k)$. Our main theorem gives a classification of all 1-super-transitive \`etale algebra objects in $\mathcal{C}(\mathfrak{sl}_N, k)$ running over all $N,k \in \mathbb{N}$. Our classification captures all known infinite families of non-pointed \`etale algebras in $\mathcal{C}(\mathfrak{sl}_N, k)$, and includes all but 16 of the known non-pointed \`etale algebra objects in these categories. These remaining 16 known examples have super-transitivities between 2 and 4.
\end{abstract}
\maketitle
\section{Introduction}

An important class of tensor categories are the categories of integrable level-$k$ representations of the affine Lie algebras $\mathfrak{g}$ where $\mathfrak{g}$ is a simple Lie algebra, and $k$ is a positive integer level \cite{SisiBakBak}. These categories are typically denoted $\mathcal{C}(\mathfrak{g},k)$. There has been a recent revival in the study of \`etale algebra objects in these categories. This revival is mainly due to work of Gannon \cite{LevelTerry} and Schopieray \cite{LevelAndy}. In particular recent work of Gannon has shown that for a fixed simple Lie algebra $\mathfrak{g}$, there exists an effective bound on the level $k$ where the category $\mathcal{C}(\mathfrak{g}, k)$ can have non-pointed (i.e. exceptional) \`etale algebra objects. This has allowed Gannon to give complete classification results for \`etale algebra objects in the categories $\mathcal{C}(\mathfrak{g}, k)$ for all $\mathfrak{g}$ of rank strictly less than 7. This fully resolves the classification problem for all exceptional simple Lie algebras (aside from the Lie algebras of type $E_7$ and $E_8$, which will fall to additional computational power using Gannon's techniques).

Further progress on the classification of \`etale algebras for the classical Lie algebras will have to improve on Gannons results by obtaining restrictions independent of the rank of the Lie algebra. This paper is an initial attempt at the study of \`etale algebras in a rank agnostic fashion. We restrict our attention to type $A$ for this work, though we expect our techniques will extend to give results in types $B$, $C$, and $D$. Towards this goal we make the following definition of \textit{super-transitivity} for \`etale algebras in $\mathcal{C}(\mathfrak{sl}_N, k)$. This is a direct analogue of the notion of super-transitivity for subfactors (see \cite[Definition 4.1.1]{sup} and \cite[Definition 3.2.1]{sup2}), which itself is a generalisation of the notion of transitivity for group actions.

\begin{defn} 
Let $A\in \mathcal{C}(\mathfrak{sl}_N, k)$ be an \`etale algebra object. We say $A$ is $n$-super-transitive if 
\[   \dim\operatorname{End}_{ \mathcal{C}(\mathfrak{sl}_N, k)}(  \ydiagram{1}^{\otimes n}) =   \dim\operatorname{End}_{ \mathcal{C}(\mathfrak{sl}_N, k)_A}(  \mathcal{F}_A(\ydiagram{1}^{\otimes n}))  \]
and 
\[   \dim\operatorname{End}_{ \mathcal{C}(\mathfrak{sl}_N, k)}(  \ydiagram{1}^{\otimes n+1}) <   \dim\operatorname{End}_{ \mathcal{C}(\mathfrak{sl}_N, k)_A}(  \mathcal{F}_A(\ydiagram{1}^{\otimes n+1})).  \]
Equivalently, $A$ is $n$-super-transitive if \[\dim\Hom_{\mathcal{C}(\mathfrak{sl}_N,k)} (A \to \ydiagram{1}^{\otimes n} \otimes \ydiagram{1}^{\otimes -n})=1\quad \text{and} \quad \dim\Hom_{\mathcal{C}(\mathfrak{sl}_N,k)} (A \to \ydiagram{1}^{\otimes n+1} \otimes \ydiagram{1}^{\otimes -(n+1)})>1.\]
\end{defn}

We should roughly think of super-transitivity as a measure of where the ``new stuff'' appears in the category $\mathcal{C}(\mathfrak{sl}_N, k)_A$ relative to $\mathcal{C}(\mathfrak{sl}_N, k)$, with respect to the free module functor and the vector representation. More precisely if $A$ is $n$-super-transitive, then the fusion graph for $\ydiagram{1}\in \mathcal{C}(\mathfrak{sl}_N, k)$ agrees with the fusion graph for $A\otimes \ydiagram{1}\in \mathcal{C}(\mathfrak{sl}_N, k)_A$ up to depth $n$. Note that it is not clear from the definition that every \`etale algebra in $\mathcal{C}(\mathfrak{sl}_N,k)$ is $n$-super-transitive for some $n\in \mathbb{N}$. It is immediate from the definition that an \`etale algebra has super-transitivity if and only if it has non-trivial support in the 0-graded component of $\mathcal{C}(\mathfrak{sl}_N, k)$. From the theory of algebra objects in graded tensor categories \cite{iwenttodmitrisofficeandaskedhim}, this implies that every non-pointed \`etale algebra object is $n$-super-transitive for some $n\in \mathbb{N}$.

As of the time of writing, the known examples of non-pointed \`etale algebra objects in the categories $\mathcal{C}(\mathfrak{sl}_N,k)$ consists of three infinite families, along with a handful of (as of now) sporadic examples. The majority of these \`etale algebras come from the theory of conformal inclusions. We provide a summary of this correspondence in Subsection~\ref{sub:CI}. We point the reader to \cite{ModulesPt2} for a summary of these \`etale algebras. Using the notation of this paper (which is recalled in Subsection~\ref{sub:CI}) we have the following super-transitivities for these known examples. We exclude level-rank duals \cite{Mirror,lrdual} (as this construction preserves super-transitivity) and extensions of algebras (as this construction decreases super-transitivity in general, and preserves it in these special cases) from this list\footnote{We also exclude the example $A_{\mathfrak{so}_{70}}\in \mathcal{C}(\mathfrak{sl}_8,10)$, as we were unable to find the branching rules for this conformal embedding in the literature. Curiously, the rank and level of this example exactly line up with the singular missing case in our main theorem.}
\[  
\begin{tabular}{c|c}
     $n$&  \`Etale algebras $A$ which are $n$-super-transitive \\\hline \hline
    1 & $A_{\mathfrak{so}_{N^2-1}}\in \mathcal{C}(\mathfrak{sl}_N, N)$\\
     & $A_{\mathfrak{sl}_{N(N+1)/2}}\in \mathcal{C}(\mathfrak{sl}_N, N+2)$  \\\hline
    2& $A_{\mathfrak{so}_{5}}\in \mathcal{C}(\mathfrak{sl}_2, 10)$\\
     &$A_{\mathfrak{e}_{6}}\in \mathcal{C}(\mathfrak{sl}_3, 9)$\\
     &$A_{\mathfrak{so}_{20}}\in \mathcal{C}(\mathfrak{sl}_4, 8)$\\
     &$A_{\mathfrak{sp}_{20}}\in \mathcal{C}(\mathfrak{sl}_6, 6)$\\\hline
      3& $A_{\mathfrak{e}_{7}}\in \mathcal{C}(\mathfrak{sl}_3, 21)$\\
      & $A_{\text{Shellekens}}\in \mathcal{C}(\mathfrak{sl}_7, 7)$\\ \hline
    4  &  $A_{\mathfrak{g}_{2}}\in \mathcal{C}(\mathfrak{sl}_2, 28)$
\end{tabular}
\]
While our list of known \`etale algebra objects in the categories $\mathcal{C}(\mathfrak{sl}_N,k)$ is almost certainly incomplete, it is very tempting to make conjectures based on the above table. The following is a slight refinement of a conjecture made by the first author and Noah Snyder. 
\begin{conj}\label{conj:CN}
    Consider the set $\mathcal{A}$ of all non-pointed \`etale algebra objects in the categories $\mathcal{C}(\mathfrak{sl}_N,k)$ running over both $N$ and $k$. Then \begin{enumerate}[i)]
        \item there exist only finitely many $A\in \mathcal{A}$ with super-transitivity strictly greater than 1, and
        \item if $A\in \mathcal{A}$ is 1-super-transitive, then $A$ is one of the \`etale algebras coming from the conformal embeddings 
        \begin{align*}
            \mathcal{V}(\mathfrak{sl}_{N}, N-2)&\subseteq \mathcal{V}(\mathfrak{sl}_{N(N-1)/2},1)\\
            \mathcal{V}(\mathfrak{sl}_{N}, N)&\subseteq \mathcal{V}(\mathfrak{so}_{N^2-1},1)\\
            \mathcal{V}(\mathfrak{sl}_{N}, N+2)&\subseteq \mathcal{V}(\mathfrak{sl}_{N(N+1)/2},1),
        \end{align*}
        or of a simple current extension (see \cite{simplesyrup}) of one of these embeddings.
    \end{enumerate}
\end{conj}
Note that a constructive proof of this conjecture would resolve the classification of all \`etale algebras in the categories $\mathcal{C}(\mathfrak{sl}_N,k)$. A slightly weaker conjecture (which mimics the similar conjecture made regarding super-transitivity for subfactors) is as follows.
\begin{conj}
    There exists $K \in \mathbb{N}$ such that all non-pointed \`etale algebra objects in the categories $\mathcal{C}(\mathfrak{sl}_N,k)$ have super-transitivity less than or equal to $K$.
\end{conj}
Based on the above table one may even suggest that $K=4$ is the universal bound of super-transitivity.

The main result of this paper is to prove part ii) of Conjecture~\ref{conj:CN}. To state our main theorem as cleanly as possible we recall the bijection from \cite[Section 3.6]{LagrangeUkraine} for a fixed \`etale algebra $A\in \mathcal{C}$ 
\[  \{B \in \mathcal{C}_A^0: B  \text{ is an \`etale algebra} \}  \cong \{B\in \mathcal{C}: A \subseteq B \text{  is an extension of \`etale algebras}\}.          \]
This bijection in the forward direction is given by the forgetful functor. 
\begin{defn}
    We will refer to \`etale algebra extensions of $A$ which correspond to pointed algebra objects in $\mathcal{C}_A^0$ under the above bijection as \textit{simple current extensions}.
\end{defn}
Our main theorem is the following.
\begin{thm}\label{thm:main}
Let $N,k \geq 2$ with $(N,k) \notin \{(8,10), (10,8)\}$ and let $A\in \mathcal{C}(\mathfrak{sl}_N, k)$ be a $1$-super-transitive \`etale algebra object. Then either:
\begin{enumerate}[(i)]
    \item $k = N-2$ and $A$ is a simple current extension of $A_{\mathfrak{sl}_{N(N-1)/2}}$, or 
    \item $k = N$ and $A$ is a simple current extension of $A_{\mathfrak{so}_{N^2-1}}$, or 
    \item $k = N+2$ and $A$ is a simple current extension of $A_{\mathfrak{sl}_{N(N+1)/2}}$.
    \end{enumerate}
\end{thm}
The proof of our main theorem is completed in three cases. These are Theorem~\ref{thm:classN}, Theorem~\ref{thm:classN-2}, and Theorem~\ref{thm:classN+2}.
\begin{rmk}
    We point out that our proof of the above theorem relies on bounds of $N,k\geq 7$. However with explicit classification of \`etale algebras in $\mathcal{C}(\mathfrak{sl}_N,k)$ known for $N \leq 7$ or $k\leq 7$ due to Gannon \cite{LevelTerry}, we can extend our results to the general case of $N,k\geq 2$. These results of Gannon also rule out several special pairs of $(N,k)$ where our techniques fail (namely Lemma~\ref{lem:twistanalysis}). With additional computational power, the techniques of Gannon would work to also resolve the missing cases of $(8,10)$ and $(10,8)$ in our main theorem.
\end{rmk}
We note that the simple current extensions appearing in the above theorem are fully understood and classified. 
\begin{rmk}
 We recall from \cite[Theorem 5.2]{kril} that 
 \[ \mathcal{C}(\mathfrak{sl}_N,k)_{A_\mathfrak{g}}^0 \cong \mathcal{C}(\mathfrak{g}, 1).    \]
  For $\mathcal{C}(\mathfrak{so}_{N^2-1}, 1)$ we have that this category only has the trivial \`etale algebra for $N\not\equiv \pm 1\pmod 8$, and for $N\equiv \pm 1\pmod 8$ these categories have three pointed \`etale algebras, all of which are pointed. Using the explicit branching rules for the forgetful functor $\mathcal{C}(\mathfrak{so}_{N^2-1}, 1) \to \mathcal{C}(\mathfrak{sl}_{N}, N)$ \cite[Remark 4.2.2]{KacMan} we see that the underlying objects of the two non-trivial extensions of $A_{\mathfrak{so}_{N^2-1}}$ are both
   \[      A_{\mathfrak{so}_{N^2-1}} \oplus \left(2^{\frac{N-3}{2}}\right )\cdot V_{ [N-1, N-2, \cdots, 1]   }.            \]
  For the categories $\mathcal{C}(\mathfrak{sl}_{N(N \pm 1)/2}, 1)$ we have that \`etale algebras are classified by divisors $m \mid N(N \pm 1)/2$ such that $m^2 \mid N(N \pm 1)/2$ if $N(N \pm 1)/2$ is odd, and such that $2m^2 \mid N(N \pm 1)/2$ if $N(N \pm 1)/2$ is even. The explicit structure of the object structure of the corresponding \`etale algebra extensions of $A_{\mathfrak{sl}_{N(N\pm 1)/2}}$ can be determined using the branching rules for $\mathcal{C}(\mathfrak{sl}_{N(N\pm 1)/2}, 1) \to \mathcal{C}(\mathfrak{sl}_{N}, N \pm 2)$ from \cite{LevLaughLib}.
\end{rmk}

Let us now outline our proof of Theorem~\ref{thm:main}, and the structure of the paper.

In Section~\ref{sec:pre} we recall the relevant facts about the categories $\mathcal{C}(\mathfrak{sl}_N,k)$, with a particular focus on the ``$GL(N)$" labeling set of the simples which is independent of $N$. We also recall the Kazhdan-Wenzl presentation for the categories $\mathcal{C}(\mathfrak{sl}_N,k)$, which will be a key ingredient in our later proofs. We review the basics of \`etale algebra objects in braided tensor categories, and the construction and properties of the tensor category of module objects internal to the braided tensor category. Finally we discuss the \`etale algebra objects in $\mathcal{C}(\mathfrak{sl}_N,k)$ coming from the theory of conformal inclusions, and recall presentations for the corresponding categories of modules for the three infinite families.

In Section~\ref{sec:ini} we obtain initial results on the structure of 1-super-transitive algebra objects $A$. While we have no hope of fully determining the object structure of $A$ using the arguments of this section, we are able to determine the intersection of $A$ with $\ydiagram{1}^{\otimes 3} \otimes \ydiagram{1}^{\otimes -3}$. The key obstruction here is to use that the simple summands of an \`etale algebra in $\mathcal{C}(\mathfrak{sl}_N,k)$ have trivial twist. This fact gives us strong restrictions on what simple summands may appear in $A$. We then use various adjunction arguments to further bound the multiplicities of the summands which are allowed to occur in $A$. We show that for $\mathcal{C}(\mathfrak{sl}_N,k)$ to have a 1-super-transitive \`etale algebra object $A$, then we must have that $k \in \{N-2, N, N+2\}$. Further, in each of these three cases we determine the fusion graph (up to a finite number of possibilities) for $\mathcal{F}_A(\ydiagram{1}) \in \mathcal{C}(\mathfrak{sl}_N,k)_A$ up to depth three.

In Section~\ref{sec:kN} we focus on the case of $k=N$. Here we obtain a basis for $\operatorname{End}_{\mathcal{C}(\mathfrak{sl}_N,N)_A}(\mathcal{F}_A(\ydiagram{1})^{\otimes 3})$. This basis allows us to compute non-trivial three-strand relations in the category $\mathcal{C}(\mathfrak{sl}_N,N)_A$. A key ingredient in these computations is the \textit{convolution product} on $\operatorname{End}_{\mathcal{C}(\mathfrak{sl}_N,N)_A}(\mathcal{F}_A(\ydiagram{1})^{\otimes 2})$. This is an associative algebra structure on these spaces, whose structure constants we can explicitly pin down via skein theory considerations. Using the known presentation for the category $\mathcal{C}(\mathfrak{sl}_N,N)_{A_{\mathfrak{so}_{N^2-1}}}$, and the three strand relations we have computed, we can explicitly produce a functor from this known presentation into the category $\mathcal{C}(\mathfrak{sl}_N,N)_A$. It then follows by general theory that $A$ is an \`etale algebra extension of $ A_{\mathfrak{so}_{N^2-1}}$.

In Section~\ref{sec:kN2} we repeat our analysis in the remaining cases of $k = N \pm 2$. For each of these two cases we have two possible fusion graphs up to depth 3. For one of these graphs we run a similar game plan as in Section~\ref{sec:kN}. However in this case we obtain contradictions in the three strand relations, showing non-existence. For the remaining graph we can use a simple skein theoretic argument to show that there is an invertible object at depth 2 in the fusion graph. This allows us to produce a functor from the category $\mathcal{C}(\mathfrak{sl}_N,k)_{A_{\mathfrak{sl}_{N(N\pm 1)/2}}}$ into $\mathcal{C}(\mathfrak{sl}_N,N\pm2)_A$. As before, this implies that $A$ is an \`etale algebra extension of $A_{\mathfrak{sl}_{N(N\pm 1)/2}}$.
\subsection*{Acknowledgments}
CE would like to thank Dmitri Nikshych and Dave Penneys for helpful references. JK would like to give a special thanks to their family for supporting their academic journey, with a special thank you to their mother for being their biggest supporter. CE was supported by NSF DMS grant 2400089.
\section{Preliminaries}\label{sec:pre}
We refer the reader to \cite{book} the basics on tensor categories, and to \cite{kril,LagrangeUkraine} for the basics on \`etale algebra objects in braided tensor categories.

\subsection{The Categories \texorpdfstring{$\mathcal{C}(\mathfrak{sl}_N, k)$}{C(sl(N),k)}}\label{sub:sl}

One of the key objects of study in this paper are the modular tensor categories $\mathcal{C}(\mathfrak{sl}_N, k)$. These are the categories of integrable level-$k$ representations of $\widehat{\mathfrak{sl}_N}$ where $k \in \mathbb{N}$ \cite{SisiBakBak}. The tensor product on this category is level preserving fusion \cite{KL1}. Equivalently, one has that $\mathcal{C}(\mathfrak{sl}_N, k)$ is braided equivalent to $\overline{\operatorname{Rep}(U_q(\mathfrak{sl}_N))}$ with $q = e^{2 \pi i \frac{1}{2(N+k)}}$ \cite{MR1384612}.

The simple objects of $\mathcal{C}(\mathfrak{sl}_N, k)$ are typically labeled by certain highest weights, or equivalently certain Young diagrams. More precisely we have
\[  \operatorname{Irr}(\mathcal{C}(\mathfrak{sl}_N, k)) \cong \{ V_\lambda : \lambda \text{ is a Young diagram with } \lambda_1 \leq k \text{ and } \ell(\lambda) <N    \} .    \]
While this labeling set is convenient for computing tensor powers, it is less than ideal for computing duals. We have that $(V_\lambda)^* \cong V_{\lambda^*}$ where 
\[ \lambda^* = (\lambda_1, \lambda_1 - \lambda_{N-1}, \lambda_1 -  \lambda_{N-2} ,\cdots, \lambda_1 - \lambda_2).        \]
For the purpose of this paper we instead work with the ``$GL(N)$-style'' labeling set of pairs of Young diagrams. For a fixed $N$ this relabeling is given by the map
\[   (\lambda, \mu) \mapsto (\lambda_1 + \mu_1 , \lambda_2 + \mu_1 - \mu_{N-1} , \lambda_3 + \mu_1 - \mu_{N-2}, \cdots, \lambda_{N-1} +\mu_1 - \mu_2 ).     \]
Note that this new labeling set is not injective. i.e. $(  [1], [1^{N-1}]) = ( [2], \emptyset)$. However it is injective if we restrict to $(\lambda, \mu)$ with $\ell(\lambda) + \ell(\mu) < N$. The benefit of this relabeling is that duals behave nicely. It is immediate that our relabeling satisfies $(V_{\lambda, \mu})^* \cong V_{\mu, \lambda}$. In the case that $\mu = \emptyset$ we have $V_{\lambda,\emptyset} = V_\lambda$. In this setting we will write $\lambda$ as shorthand for $(\lambda, \emptyset)$.

The twist of a simple object in $\mathcal{C}(\mathfrak{sl}_N, k)$ can be computed in terms of its highest weight and the Killing form. In terms of the standard basis $\{\epsilon_i : 1\leq i < N\}$ for the weight space of $\mathfrak{sl}_N$ we have that the highest weight for $V_\lambda$ is $\sum_{i=1}^{N-1} \lambda_i \epsilon_i$. The Killing form on our basis (normalised so the positive roots have length $\sqrt{2}$) is given by 
\[    \langle \epsilon_i, \epsilon_j\rangle  = \begin{cases}
    \frac{N-1}{N} \quad \text{ if } i=j\\
    -\frac{1}{N} \quad \text{ if } i\neq j.
\end{cases}   \]
We then have that the twist of $V_\lambda$ is given by $q^{  \langle \lambda , \lambda + 2\rho\rangle  }$ where $q = e^{2 \pi i \frac{1}{2(N+k)}}$ and $\rho$ is the distinguished root $\sum_{i=1}^{N-1} (N-i)\epsilon_i$. These twists are tedious, but straight-forward to compute in practice\footnote{It appears that the general formula for the twist of $(\lambda, \mu)$ is $q^{ \frac{(|\lambda| + |\mu|)N^2 + (c_\lambda + c_\mu)N - (|\lambda| - |\mu|)^2}{N} }$ where $c_\lambda:= \sum_i \lambda_i (\lambda_i - 2i + 1)$. We do not use nor prove this formula in this paper.}.
\begin{ex}
    Consider the object labeled by $(\ydiagram{1}, \ydiagram{1})$ in $\mathcal{C}(\mathfrak{sl}_N, k)$. The corresponding weight is
    \[   ( 2,1,1,\cdots,1).              \]
    We then have
    \begin{align*} & \langle ( 2,1,1,\cdots,1) , ( 2N,2N-3,2N-5,\cdots,3) \rangle \\
    &= \frac{N-1}{N}(4N + \sum_{i=2}^{N-1}(2N - (2i - 1))) - \frac{4}{N}\sum_{i=2}^{N-1}(2N - (2i - 1)) - \frac{1}{N}\sum_{j=2}^{N-1}\sum_{i=2, i\neq j}^{N-1} 2N - (2j-1)\\
    &= (N^2 + N -2) - 4(N-2) - (N^2 - 5N + 6)\\
    &= 2N.
    \end{align*}
    Thus $\theta_{ (\ydiagram{1}, \ydiagram{1})} = q^{2N}$. The twist formulas for all other object in this paper are computed in an analogous manner. We do not include these other computations to keep the paper to a reasonable length.
\end{ex}
The fusion rules for $\mathcal{C}(\mathfrak{sl}_N, k)$ are complicated in general (see \cite{MR1328736} and \cite[Corollary 8]{sawin}). When needed we will state relevant fusion rules between simples. These are computed by noting that for Young diagrams $\lambda$ and $\mu$ with $\lambda_1 +\mu_1 < k$, the tensor product decomposition of $V_\lambda \otimes V_\mu$ agrees with the classical case in $\operatorname{Rep}(\mathfrak{sl}_N)$. The largest tensor product we need to compute in this paper will involve Young diagrams with 3 boxes, and so with an assumption of $k\geq 7$ these classical decompositions suffice.

For the purposes of this paper, it will useful to understand $\mathcal{C}(\mathfrak{sl}_N, k)$ as the Cauchy completion of a nice generators and relations category. This presentation for $\mathcal{C}(\mathfrak{sl}_N, k)$ is due to Kazhdan and Wenzl (we also direct the reader to \cite[Section 2.4]{dan} for the graphical interpretation of these results). 
\begin{defn}[\cite{sovietHans}]
    Let $q\in \mathbb{C} - \{0\}$ and $N\in \mathbb{N}_{\geq 2}$. We define $\mathcal{SH}(q,N)$ as the pivotal $\mathbb{C}$-linear monoidal $\dag$-category with objects strings in $\{+,-\}$, and morphisms generated by 
    \[         \att{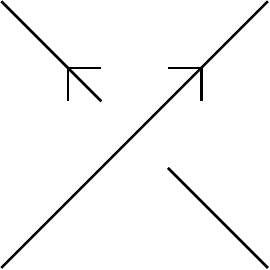}{.25}\in \End_{\mathcal{SH}(q,N)}(+^2)      \quad  \text{and} \quad \att{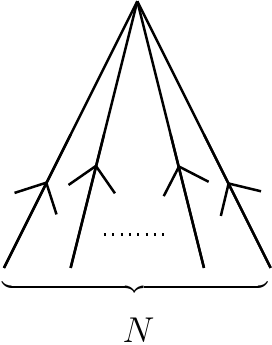}{.3}\in \Hom_{\mathcal{SH}(q,N)}( +^N \to \mathbf{1}).   \]
    The $\dag$-structure is given by the conjugate linear extension of the vertical flip of diagrams, and the relations are given by
    \begin{alignat*}{3}
\att{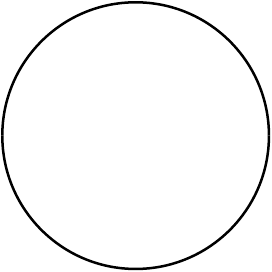}{.15} &= \frac{q^N-q^{-N}}{q-q^{-1}} \qquad&  \att{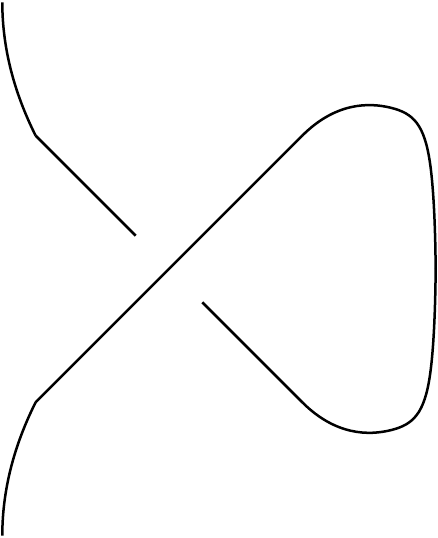}{.15}&= q^N\att{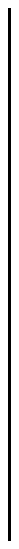}{.15} & \att{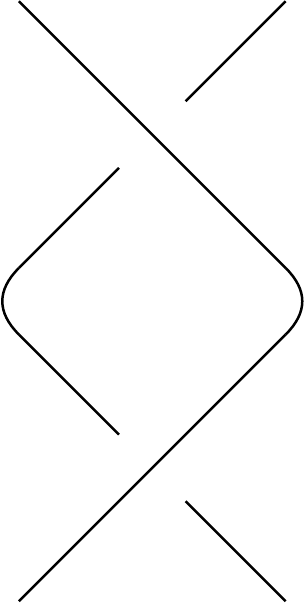}{.15} &= \att{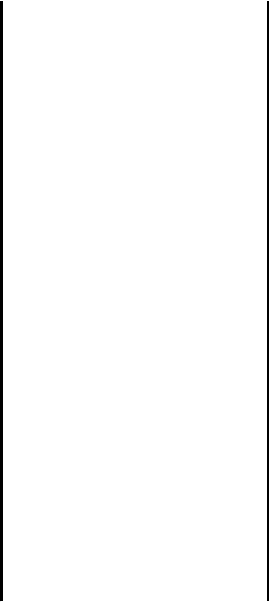}{.15}
\\
\att{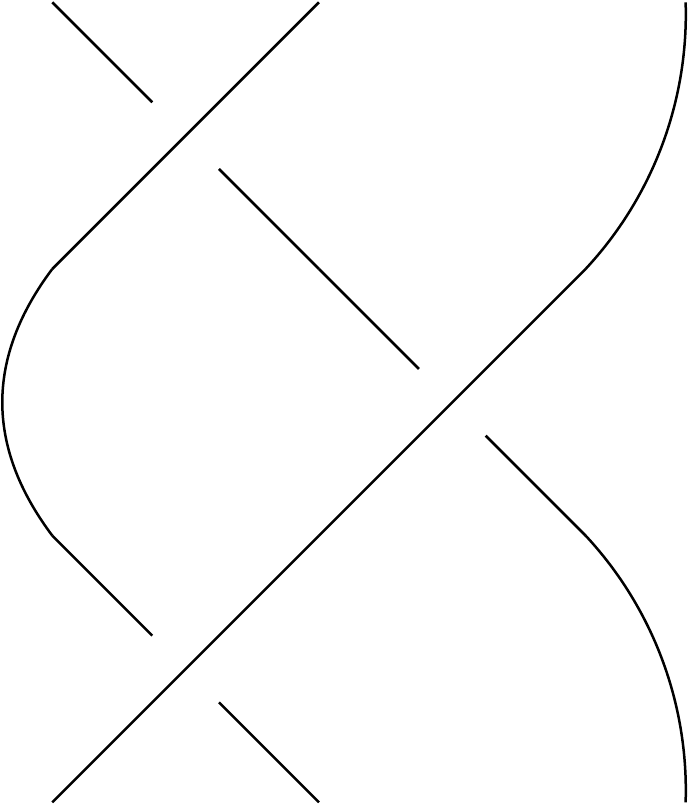}{.10} &= \att{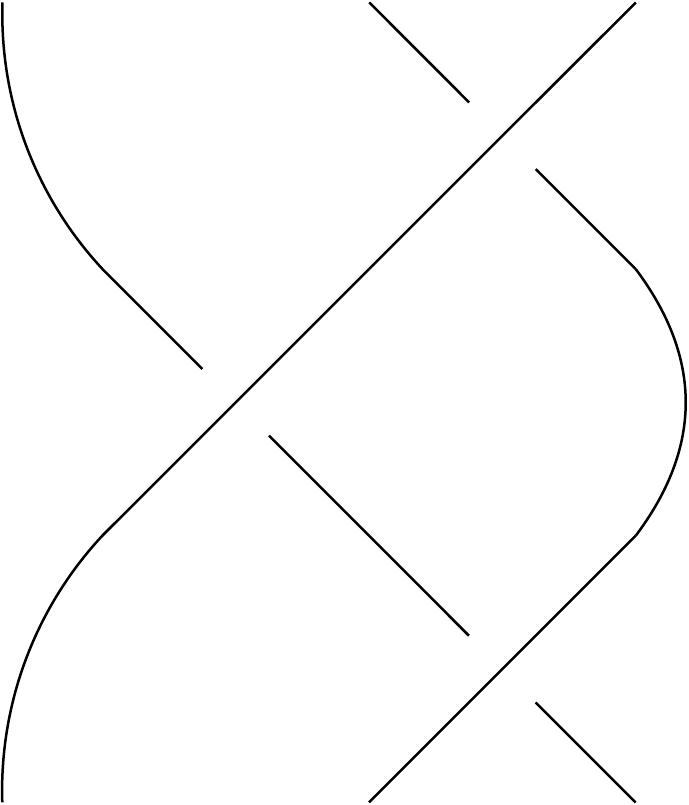}{.10}\qquad & \att{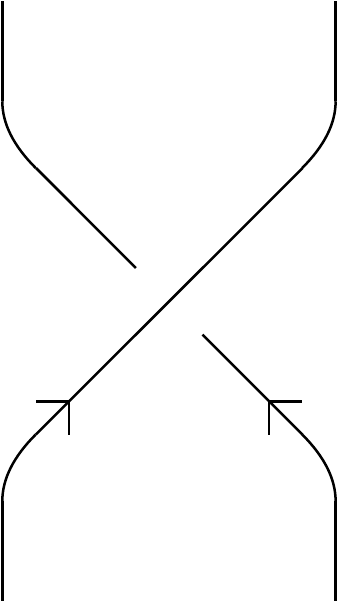}{.15} - \att{H11Inv}{.15} &= (q-q^{-1})\att{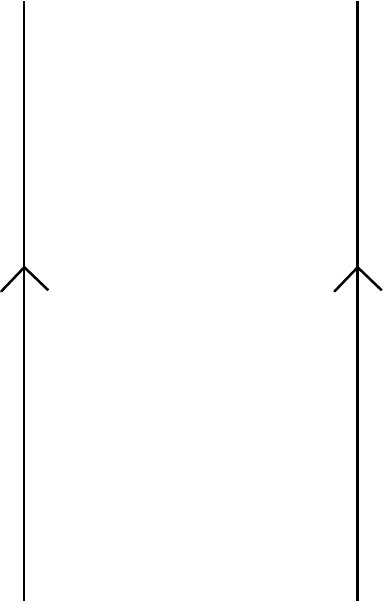}{.15} \\
\att{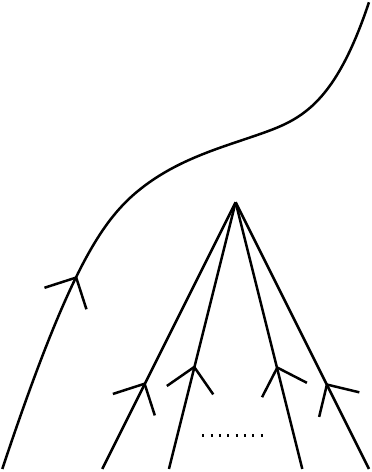}{.25} &= q \att{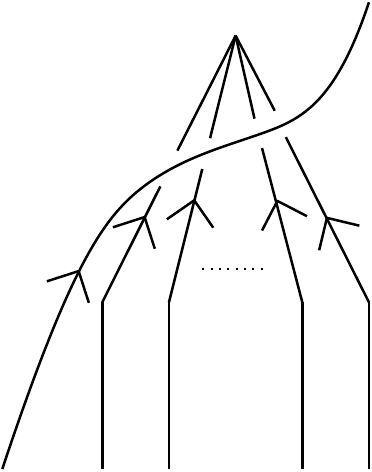}{.25}\qquad &   \att{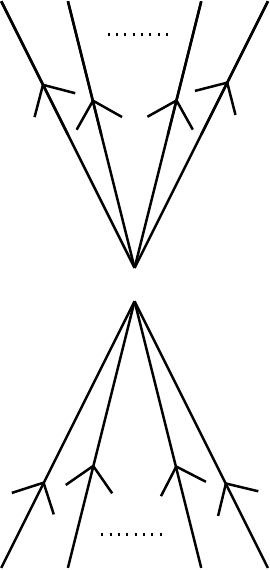}{.25} &=  \att{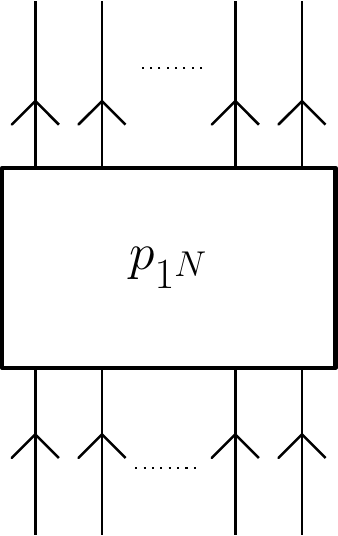}{.25}
    \end{alignat*}
    Here we use the convention that a relation involving unoriented strands holds for all possible orientations. The term $p_{1^N}$ in the last relation denotes the unique minimal central idempotent in the $+^N$-box space onto the $N$-th anti-symmetriser (see \cite[Theorem 6]{Paggo} for an explicit formula in terms of braids).
\end{defn}
We then have the following, which is also due to Kazhdan and Wenzl.
\begin{prop}[\cite{sovietHans}]
    Let $N\in \mathbb{N}_{\geq 2}$ and $k \in \mathbb{N}$. Then 
    \[  \operatorname{Ab}(\overline{ \mathcal{SH}(q,N)   }) \simeq \mathcal{C}(\mathfrak{sl}_N, k)    \]
    where $q = e^{2\pi i \frac{1}{2(N + k)}}$. Under this equivalence, the objects $+$ and $-$ get sent to $(\ydiagram{1}, \emptyset)$ and $( \emptyset,\ydiagram{1})$ respectively.
\end{prop}
Here $\operatorname{Ab}$ denotes the Cauchy completion (see \cite[Section 2.2]{me}), i.e. the additive and idempotent completion, and the overline denotes the semisimplification (see \cite{simp}). We do not require too much knowledge of the details of the above equivalence. However we will require explicit formulas for the projections onto $\ydiagram{2}$ and $\ydiagram{1,1}$ in $\operatorname{End}_{\mathcal{SH}(q,N) }(+^2)$. These are
\[   \att{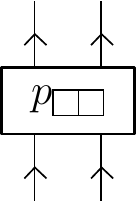}{.5}= \frac{1}{1+q^2}\left( \att{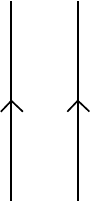}{.5} + q\att{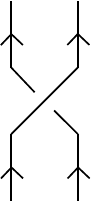}{.5}  \right) \quad \text{and} \quad  \att{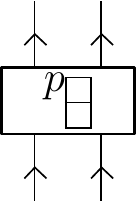}{.5}= \frac{1}{1+q^{-2}}\left( \att{idid}{.5} - q^{-1}\att{YY}{.5}  \right)\]

\subsection{Tensor Categories from \`Etale Algebras}

Given \`etale algebra object $A$ in a braided tensor category $\mathcal{C}$, one can equip the category of left $A$-modules internal to $\mathcal{C}$ with the structure of a tensor category \cite[Theorem 1.5]{kril}. This associated tensor category will be a key tool we use to classify \`etale algebra objects in this paper.

We will write $\mathcal{C}_A$ for this tensor category. It is well-known that if $\mathcal{C}$ is a braided fusion category, then $\mathcal{C}_A$ is a fusion category, i.e. it is semisimple \cite[Section 2.4]{LagrangeUkraine} and has simple unit . Further, if $\mathcal{C}$ is unitary, then $\mathcal{C}_A$ is also unitary \cite{MR4616673}.

The free module functor $\mathcal{F}_A : \mathcal{C} \to \mathcal{C}_A$ defined by $X \mapsto A\otimes X$ can be equipped with the structure of a monoidal functor using the positive braiding on $\mathcal{C}$. In the case that $\mathcal{C}$ is unitary, the functor $\mathcal{F}_A$ is a $\dag$-functor. Furthermore, if $\theta_A = \operatorname{id}_A$ then $\mathcal{C}_A$ is rigid \cite[Theorem 1.15]{kril} and pivotal \cite[Theorem 1.17]{kril}, and the $\mathcal{F}_A$ is a pivotal functor (directly seen using the explicit pivotal structure from \cite{kril}). We also have the forgetful functor $\operatorname{For}: \mathcal{C}_A \to \mathcal{C}$ which simply forgets the module structure map. A recurring technique used through out this paper will be the adjunction between these functors \cite[Theorem 1.6]{kril}. That is for $X\in \mathcal{C}$ and $M \in \mathcal{C}_A$ we have
\[   \Hom_{\mathcal{C}}( X \to \operatorname{For}(M)) \cong   \Hom_{\mathcal{C}_A}( \mathcal{F}_A(X) \to M).       \]

In general the image of the braiding in $\mathcal{C}$ under the free module functor is no longer a braiding in $\mathcal{C}_A$. This is as the image may no longer be natural in both positions in general. However the image of the braiding is still natural in the ``under'' position. 

\begin{lem}\label{lem:OB}\cite[Lemma 2.4]{me} 
Let $\mathcal{C}$ be a braided tensor category, and $A$ an \`etale algebra object. Let $\mathcal{C}_A$ be the tensor category of of $A$-modules internal to $\mathcal{C}$. Then for any $f\in \operatorname{Hom}_{ \mathcal{C}_{A}}(\mathcal{F}_A(Y_1)\to \mathcal{F}_A(Y_2))$ we have that 
 \[\att{overproof2}{.3} =\att{overproof1}{.3} .  \] 
\end{lem}
This in particular implies the following well known commutativity result on the fusion ring for $\mathcal{C}_A$.
\begin{lem}\label{lem:com}
    For any $X\in \mathcal{C}$ and $M \in \mathcal{C}_A$ we have
    \[  \mathcal{F}_A(X) \otimes M \cong M \otimes \mathcal{F}_A(X).  \]
\end{lem}

For this paper we will be working exclusively with $\mathcal{C} = \mathcal{C}(\mathfrak{sl}_N, k)$. As this category is semisimple, the free module functor $\mathcal{F}_A: \mathcal{C}(\mathfrak{sl}_N, k) \to \mathcal{C}(\mathfrak{sl}_N, k)_A$ is faithful. In a slight abuse of notation we will identify $\mathcal{C}(\mathfrak{sl}_N, k)$ with its image under $\mathcal{F}_A$ in $\mathcal{C}(\mathfrak{sl}_N, k)_A$. In particular free modules in $\mathcal{C}(\mathfrak{sl}_N, k)_A$ will be labeled by (pairs of) Young diagrams, and the planar algebraic subcategory generated by $(\ydiagram{1},\emptyset)$ and $(\emptyset,\ydiagram{1})$ in $\mathcal{C}(\mathfrak{sl}_N, k)_A$ will contain the subcategory $\overline{\mathcal{SH}(q, N)}$ for the appropriate $q$ value.

\subsection{\`Etale Algebras from Conformal Embeddings}\label{sub:CI}

A key source of \`etale algebras comes from the theory of conformal embeddings. These are inclusions of Wess-Zumino-Witten vertex operator algebras \cite{WZW}
\[      \mathcal{V}(\mathfrak{sl}_N, k)   \subseteq  \mathcal{V}(\mathfrak{g}, 1)   \]
such that $\mathcal{V}(\mathfrak{g}, 1) $ decomposes as a finite direct sum of irreducible $ \mathcal{V}(\mathfrak{sl}_N, k) $-representations. It follows from \cite{Xu,kril,HKL} that such a conformal embedding gives rise to an \`etale algebra in $\mathcal{C}(\mathfrak{sl}_N, k)$. 
\begin{defn}
    For a conformal inclusion of the form 
    \[      \mathcal{V}(\mathfrak{sl}_N, k)   \subseteq  \mathcal{V}(\mathfrak{g}, 1)   \]
    we will write $A_\mathfrak{g}$ for the corresponding \`etale algebra in $\mathcal{C}(\mathfrak{sl}_N, k)$ under the above correspondence.
\end{defn}

For the purpose of this paper, we will be interested in the following three families of conformal embeddings:
\begin{align*}
    \mathcal{V}(\mathfrak{sl}_N, N-2)&\subseteq \mathcal{V}(\mathfrak{sl}_{N(N-1)/2}, 1)\\
    \mathcal{V}(\mathfrak{sl}_N, N)&\subseteq \mathcal{V}(\mathfrak{so}_{N^2-1}, 1)\\
    \mathcal{V}(\mathfrak{sl}_N, N+2)&\subseteq \mathcal{V}(\mathfrak{sl}_{N(N+1)/2}, 1).
\end{align*}
The corresponding \`etale algebras for these conformal embeddings are 1-super-transitive, and hence we expect them to occur in our classification result. For the purposes of this paper we do not need to know too much information on the structure of the conformal embedding, but we will need the following presentations for their categories of $A$-modules in $\mathcal{C}(\mathfrak{sl}_N,k)$.

In the case of the conformal embedding $\mathcal{V}(\mathfrak{sl}_N, N)\subseteq \mathcal{V}(\mathfrak{so}_{N^2-1},1)$, the corresponding category of $A_{\mathfrak{so}_{N^2-1}}$-modules was the main example of focus in \cite{me}. A main result of that paper was the 
following presentation for this category of $A_{\mathfrak{so}_{N^2-1}}$-modules. Note that this is the alternate presentation described in \cite[Remark 5.3]{me}. This alternate presentation will be more amenable to the techniques we develop in this paper to classify 1-super-transitive \`etale algebra objects.

\begin{defn}\label{def:SEN}\cite{me}
Let $N \in \mathbb{N}_{\geq 2}$. We define $\mathcal{SE}_N$ as the extension of $\mathcal{SH}(q,N)$ where $q = e^{2\pi i \frac{1}{4N}}$ by the additional generator
\[       \att{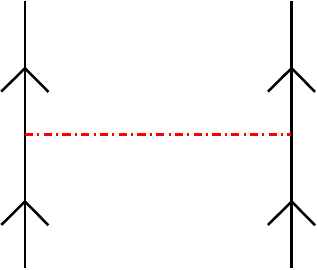}{.3}\in \End_{\mathcal{SE}_N}(+^2)         \]
satisfying the algebra relations:
\begin{align*}
\att{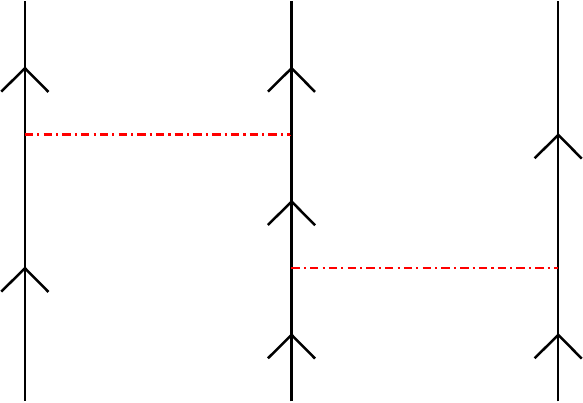}{.3} &= -\att{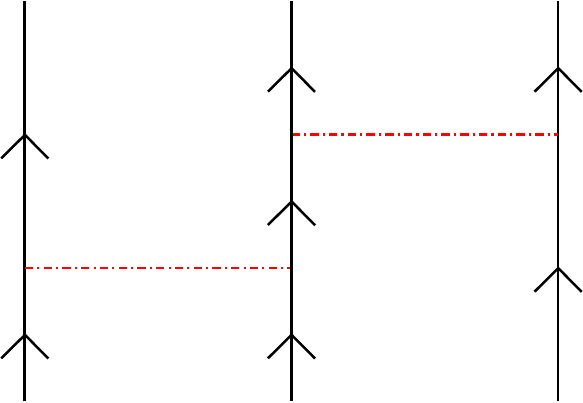}{.3}\qquad 
   && \att{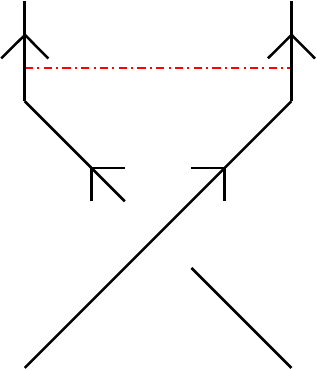}{.3} = -\att{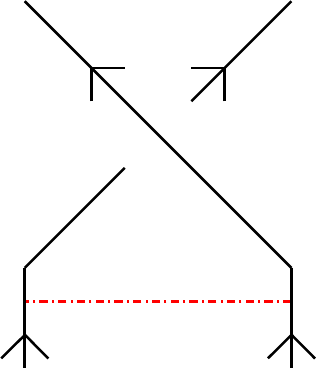}{.3} + (q-q^{-1})\att{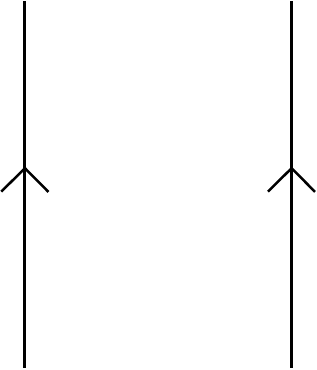}{.3}\\
     \att{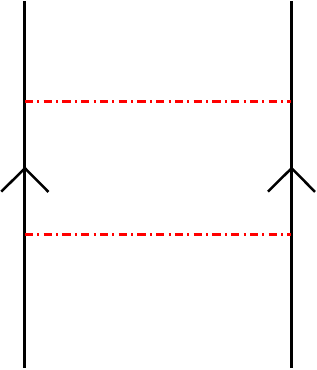}{.3} &= - \att{splittingEndR5}{.3}\qquad 
    &&\att{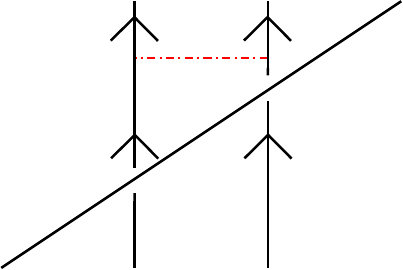}{.3} = \att{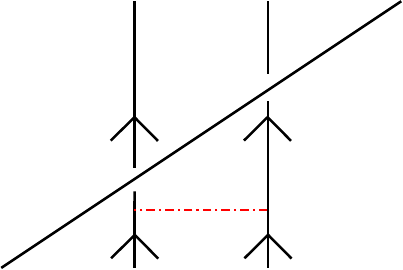}{.3}\\
  \att{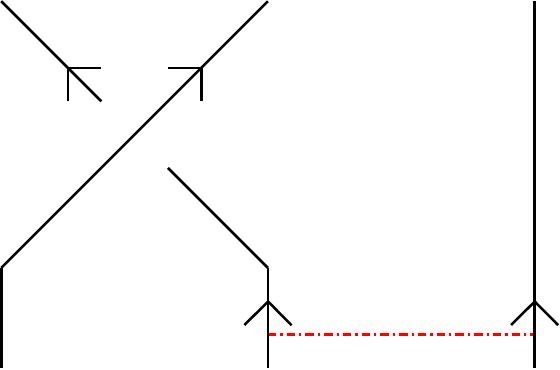}{.3}&= \att{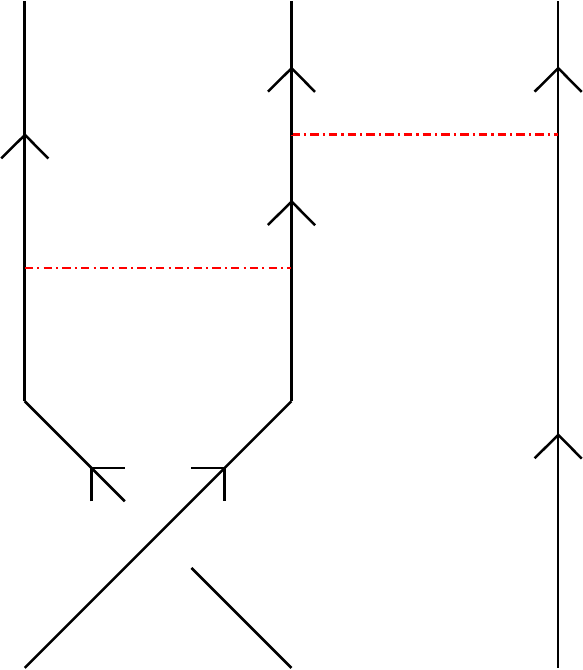}{.3}  .
\end{align*}    
along with the trace relations:
\[\att{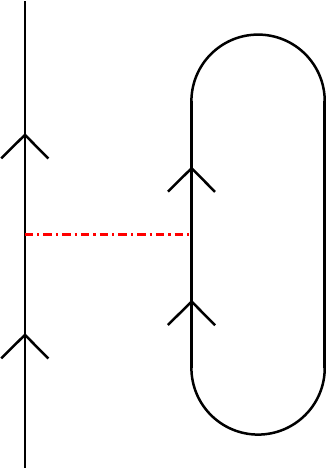}{.3} \quad=\quad 0 \quad=\quad \att{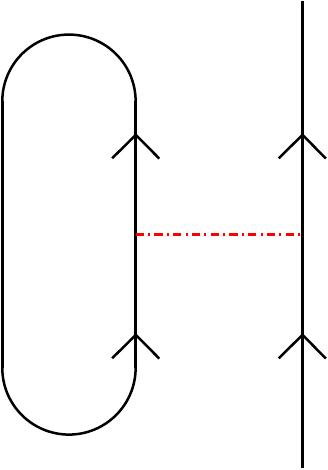}{.3} \qquad \text{and}\qquad \att{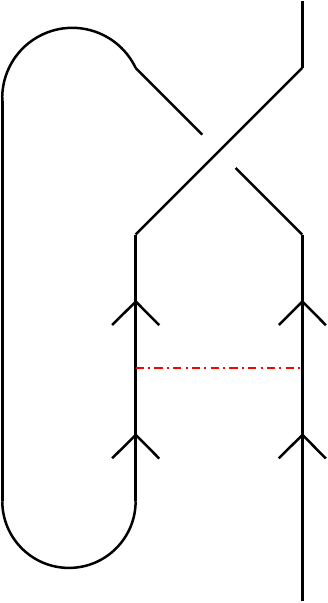}{.3} \quad=\quad \mathbf{i} \; \att{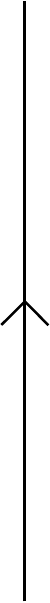}{.3}\quad  = \quad \att{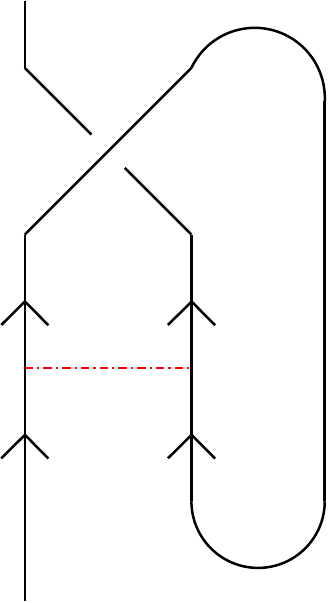}{.3}.\]
\end{defn}
We have that the Cauchy completion of the semisimplification of $\mathcal{SE}_N$ recovers the $A$-module category internal to $\mathcal{C}(\mathfrak{sl}_N, N)$.
\begin{thm}\cite[Theorem 1.5]{me}\label{thm:me}
    Let $N \in \mathbb{N}_{\geq 3}$, and set $q = e^{2\pi i \frac{1}{4N}}$. Then there is a monoidal equivalence
    \[ \operatorname{Ab}(\overline{\mathcal{SE}_N}  ) \simeq \mathcal{C}(\mathfrak{sl}_N, N)_{A_{\mathfrak{so}_{N^2-1}}}.       \]
    
\end{thm}

The techniques developed in \cite{me} are quite general, and work to give presentations for the category of $A$-modules in $\mathcal{C}(\mathfrak{sl}_N,k)$ coming from other conformal embeddings. In particular we can apply these techniques to the families 
\[\mathcal{V}(\mathfrak{sl}_N, N-2)\subseteq \mathcal{V}(\mathfrak{sl}_{N(N-1)/2}, 1)\qquad \text{and} \qquad 
    \mathcal{V}(\mathfrak{sl}_N, N+2)\subseteq \mathcal{V}(\mathfrak{sl}_{N(N+1)/2}, 1)\]
to obtain the following definition.

\begin{defn}\label{def:SDN}
Let $N \in \mathbb{N}_{\geq 2}$. We define $\mathcal{SD}^{\pm}_N$ as the extension of $\mathcal{SH}(q,N)$ where $q = e^{2\pi i \frac{1}{2( 2N \pm 2)}}$ by the additional generator
\[    \att{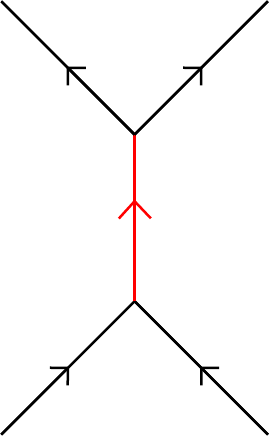}{.3}  \in \End_{\mathcal{SD}_N}(+^2)         \]
satisfying the relations:
\begin{align*}
     \att{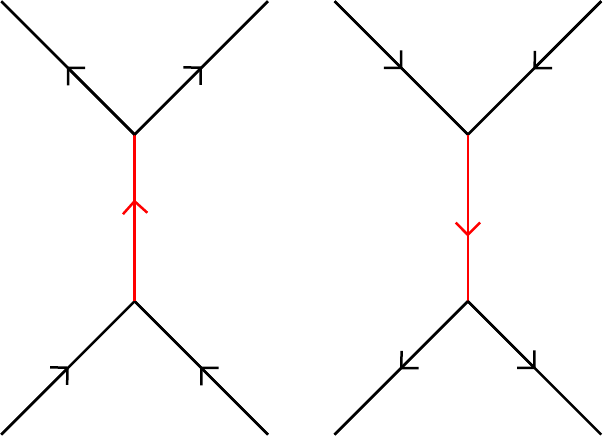}{.3} &=\att{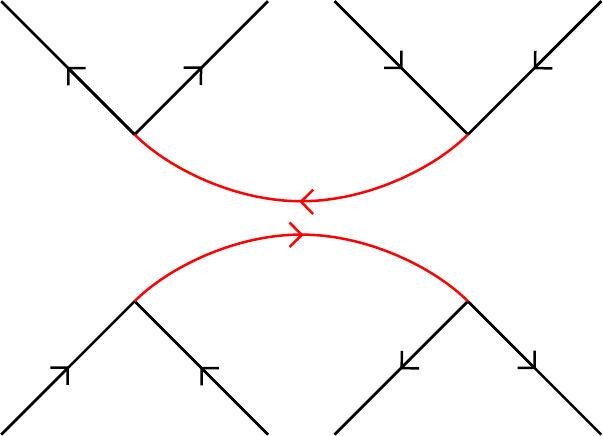}{.3} \qquad &&\att{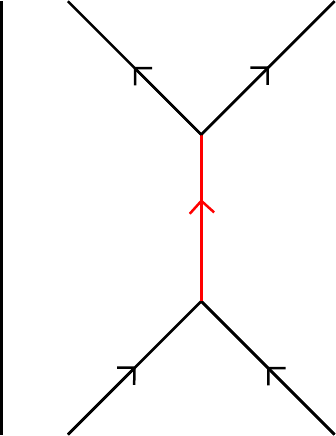}{.3}=\att{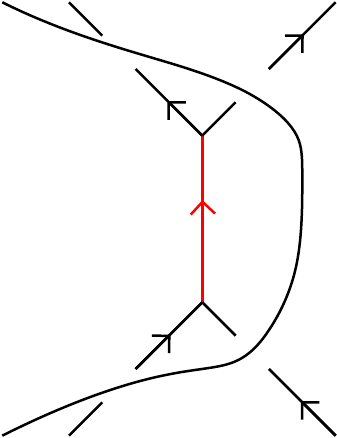}{.3}\\
     \att{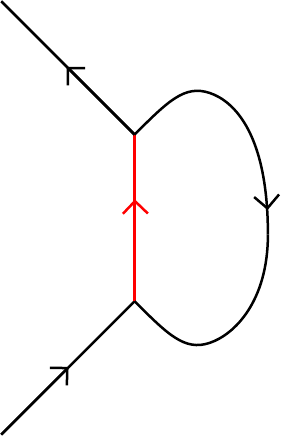}{.3} &= \frac{q-q^{-1}}{i (q + q^{-1})} \att{id1}{.3} = \att{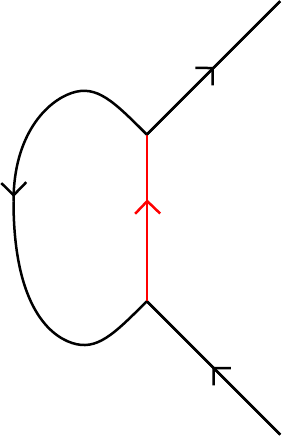}{.3}&&
      \att{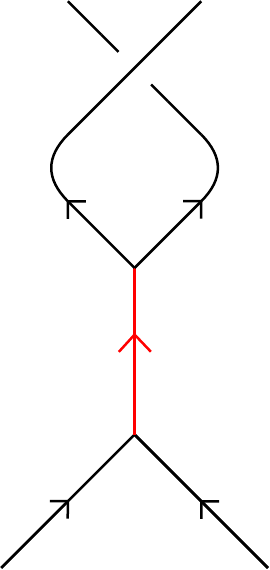}{.3} = \pm q^{\pm 1}\att{splittingLiu}{.3}=\att{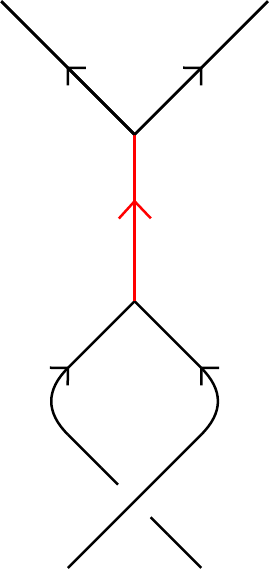}{.3}.
\end{align*}
\end{defn}
We note that the above categories have appeared implicitly in work of Liu \cite{LiuYB}. We can then prove using the same techniques as in \cite{me} that these categories are presentations for the corresponding conformal embedding categories. The following theorem proves conjectures made by Liu in \cite{LiuYB}.
\begin{thm}\label{thm:liu}
    Let $N \in \mathbb{N}_{\geq 3}$, and set $q = e^{2\pi i \frac{1}{2(2N \pm 2)}}$. Then there are monoidal equivalences
    \[ \operatorname{Ab}(\overline{\mathcal{SD}^\pm_N}  ) \simeq \mathcal{C}(\mathfrak{sl}_N, N\pm 2)_{A_{\mathfrak{sl}_{N(N\pm 1)/2}}}    .   \]

\end{thm}
\begin{proof}
    This proof closely mimics the corresponding proof for the categories $\mathcal{SE}_N$ found in \cite{me}. As the details in the current case follow the same outline, we will only detail the parts that are significantly different. We will work the case of $A_{\mathfrak{sl}_{N(N+1)/2}}$, as the $A_{\mathfrak{sl}_{N(N-1)/2}}$ case is nearly identical.

    The first step is to build a monoidal functor $\mathcal{SD}^+_N\to \mathcal{C}(\mathfrak{sl}_N, N+ 2)_{A_{\mathfrak{sl}_{N(N+1)/2}}} $. From \cite[Theorem 5.2]{kril}, we know that $\mathcal{C}(\mathfrak{sl}_N, N+ 2)_{A_{\mathfrak{sl}_{N(N+1)/2}}} $ contains a distinguished invertible local module $g$, and further from \cite{LevLaughLib} we have that $ \ydiagram{2}$ is a summand of $\operatorname{For}(g)$. It follows from the adjunction between the free module functor and the forgetful functor that $\Hom_{ \mathcal{C}(\mathfrak{sl}_N, N+ 2)_{A_{\mathfrak{sl}_{N(N+1)/2}}}}( g \to \ydiagram{1}^{\otimes 2}) =1$. As $\mathcal{C}(\mathfrak{sl}_N, N+ 2)_{A_{\mathfrak{sl}_{N(N+1)/2}}}$ is semisimple there exists a unique minimal central projection onto $g$ in $\End_{ \mathcal{C}(\mathfrak{sl}_N, N+ 2)_{A_{\mathfrak{sl}_{N(N+1)/2}}}}(\ydiagram{1}^{\otimes 2})$. As $g$ is invertible it follows that this projection satisfies the first three relations of $\mathcal{SD}^+_N$. Another adjunction argument shows that $\Hom_{ \mathcal{C}(\mathfrak{sl}_N, N+ 2)_{A_{\mathfrak{sl}_{N(N+1)/2}}}}( g \to \ydiagram{2}) = 1$, and so $g$ is a summand of the free module $\ydiagram{2}$. It follows that the projection onto $g$ is a subprojection of the projection onto $\ydiagram{2}$, and in particular the projection onto $g$ absorbs the braid at a cost of $q$. This gives that the projection onto $g$ satisfies the final defining relation of $\mathcal{SD}^+_N$. Hence there exists a dominant monoidal functor $\mathcal{SD}^+_N\to \mathcal{C}(\mathfrak{sl}_N, N+ 2)_{A_{\mathfrak{sl}_{N(N+1)/2}}} $.

    To show this functor descends to a faithful functor out of the semisimplification, we must show that $\mathcal{SD}^+_N$ is evaluable, i.e. $\dim\End_{\mathcal{SD}^+_N}(\mathbf{1}) = 1$. The evaluation algorithm closely follows the algorithm used for $\mathcal{SE}_N$ almost works verbatim. The only significant difference is that one now uses the local simplification of diagrams (and the horizontal mirror) to progress the induction:
    \[  \att{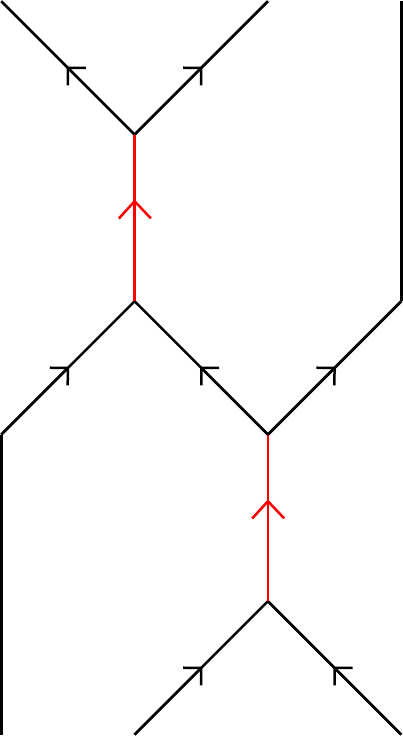}{.3}  = \att{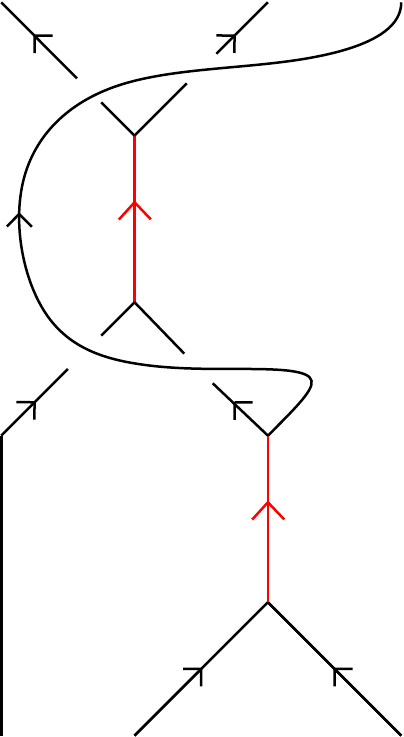}{.3} =\att{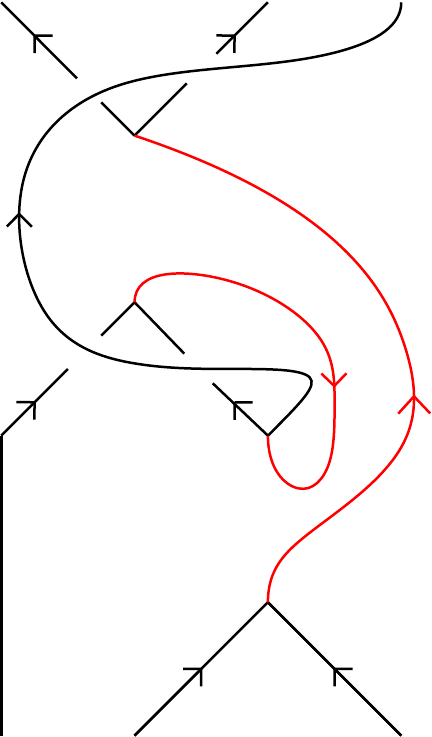}{.3} = q^{-1} \att{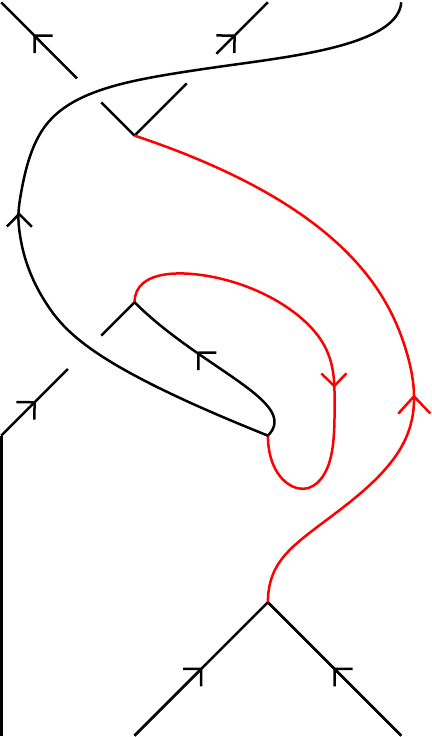}{.3} = - \frac{q-q^{-1}}{q+q^{-1}}\att{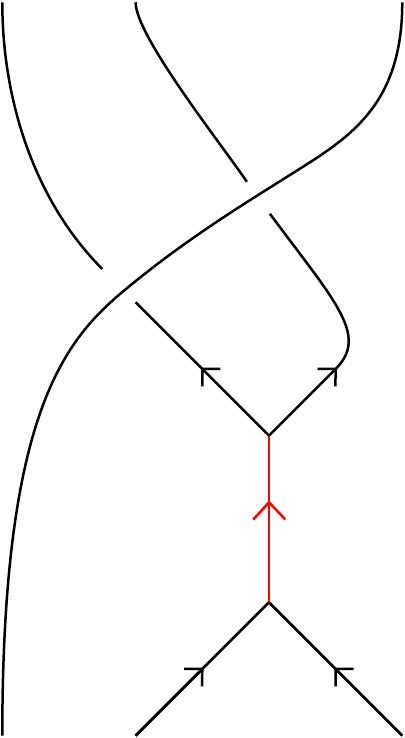}{.3}. \]
    With the evaluation algorithm in hand, we obtain a dominant faithful monoidal functor $\overline{\mathcal{SD}^+_N}\to \mathcal{C}(\mathfrak{sl}_N, N+ 2)_{A_{\mathfrak{sl}_{N(N+1)/2}}} $.

    To show this functor is an equivalence we apply analogous techniques as in \cite[Theorem 4.17]{me} to show that $\operatorname{Ab}( \overline{\mathcal{SD}^+_N}  ) \simeq \mathcal{C}(\mathfrak{sl}_N, N+ 2)_B$ where $B$ is an \`etale subalgebra object of $A_{\mathfrak{sl}_{N(N+1)/2}}$. These subalgebras are classified by \cite[Theorem 1.1]{papi}, where it is shown the only proper subalgebras are pointed. As in the proof of \cite[Theorem 4.17]{me} a contradiction arises in the case that $B$ is pointed for $N\geq 3$. Hence we have shown that $\operatorname{Ab}( \overline{\mathcal{SD}^+_N}  ) \simeq \mathcal{C}(\mathfrak{sl}_N, N+ 2)_{A_{\mathfrak{sl}_{N(N+1)/2}}}$.

\end{proof}

\section{Initial Algebra Structure}\label{sec:ini}
The goal of this section is to study the general structure of 1-super-transitive \`etale algebra objects in $\mathcal{C}(\mathfrak{sl}_N, k)$, and the consequences of this structure on the corresponding category of modules. We will obtain restrictions on the parameter $k$ where a 1-super-transitive \`etale algebra object can exist, along with a finite list of possible fusion graphs (up to depth 3) for the corresponding categories of modules. This structure will be the starting point for the results of the remaining sections which fully classify 1-super-transitive \`etale algebra objects in these categories.

A key result we use throughout this section is that if $X$ is a simple summand of an \`etale algebra object in a pseudo-unitary braided tensor category (for our purposes $\mathcal{C}(\mathfrak{sl}_N, k)$), then the twist of $X$ is trivial \cite{LevelAndy}. Along with the explicit formulas for the twists of the simple objects of $\mathcal{C}(\mathfrak{sl}_N, k)$ given in Subsection~\ref{sub:sl} we obtain strong restrictions on which simple objects can be summands of any \`etale algebra. The following result is an easy consequence of these ideas. Many of the other proofs in this section are just more involved applications of the idea in the short proof below.

\begin{rmk}
    For typesetting purposes for the remainder of the paper we will write $\mathcal{C}$ in place for $\mathcal{C}(\mathfrak{sl}_N, k)$.
\end{rmk}

\begin{lem}\label{lem:simp}
    For any $N,k\in \mathbb{N}$ and $A$ an \`etale algebra in $\mathcal{C}$. Then $\ydiagram{1}$ remains simple in $\mathcal{C}_A$.
\end{lem}
\begin{proof}
   Using the free module functor adjunction and rigidity we compute
   \[  \dim\End_{\mathcal{C}_A}(\ydiagram{1}) =  \dim\Hom_{\mathcal{C}}(A \to \ydiagram{1}\otimes \ydiagram{1}^*) =  \dim\Hom_{\mathcal{C}}(A \to \mathbf{1} \oplus ( \ydiagram{1}, \ydiagram{1}  )). \]
   The twist of $( \ydiagram{1}, \ydiagram{1}  )$ is $q^{2N}$ where $q = e^{2\pi i \frac{1}{2(N+k)}}$. This twist is never trivial for any $N\geq 2$ and $k\geq 1$. Hence $( \ydiagram{1}, \ydiagram{1}  )\not \in A$ and so 
   \[\dim\End_{\mathcal{C}_A}(\ydiagram{1})=1.\]
\end{proof}

An immediate consequence of $\ydiagram{1}$ being simple is the following fusion graph obstruction. 
\begin{lem}\label{lem:basOb}
    For any $N,k\in \mathbb{N}$ and $A$ an \`etale algebra in $\mathcal{C}$, and let $X,Y$ be simple summands of $\ydiagram{1}^{\otimes 2}$ in $\mathcal{C}_A$. Then
    \[  \dim\Hom_{\mathcal{C}_A}( X\otimes \ydiagram{1} \to  Y\otimes \ydiagram{1}   ) > 0. \]
\end{lem} 
\begin{proof}

Let
\[  \att{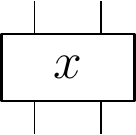}{.5},\quad  \att{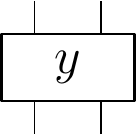}{.5}\in \operatorname{End}_{\mathcal{C}_A}(\ydiagram{1}^{\otimes 2})  \]
be minimal projections onto $X$ and $Y$. As $\ydiagram{1}$ is simple from Lemma~\ref{lem:simp}, we have that
\[   \att{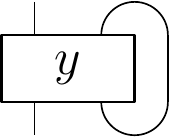}{.5} = \frac{\dim(Y)}{\dim(\ydiagram{1})} \att{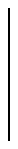}{.5}.   \]

By Lemma~\ref{lem:com} we have that 
\[\Hom_{\mathcal{C}_A}( X\otimes \ydiagram{1} \to  Y\otimes \ydiagram{1}   )\cong \Hom_{\mathcal{C}_A}( X\otimes \ydiagram{1} \to  \ydiagram{1} \otimes Y  ).\]
Suppose this hom space is zero-dimensional. It follows we have the relation
\[  \att{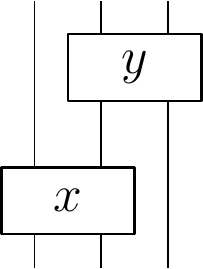}{.5} = 0.  \]
Taking the right partial trace gives
\[    \frac{\dim(Y)}{\dim(\ydiagram{1})} \att{xProj}{.5} = 0  \]
which implies that $\dim(Y) = 0$. This contradicts semisimplicity.
\end{proof}

We now restrict out attention to 1-super-transitive \`etale algebra objects, where we can say more.

\begin{prop}\label{prop:levels}
    Let $N,k \geq 5$, and suppose that $A\in \mathcal{C}$ is 1-super-transitive. Then either
    \begin{enumerate}[(i)]
        \item $k = N-2$ and $\left(\ydiagram{1,1},\ydiagram{1,1} \right)\in A$, or
        \item $k = N$ and both $\left(\ydiagram{2},\ydiagram{1,1} \right)\in A$ and $\left(\ydiagram{1,1},\ydiagram{2} \right)\in A$, or
         \item $k = N+2$, and $(\ydiagram{2},\ydiagram{2} )\in A$.
    \end{enumerate}
\end{prop}
\begin{proof}
    Suppose $A$ is 1-super-transitive, then we have by definition of super-transitivity that
    \[ 2 <   \dim\End_{\mathcal{C}_A}(\ydiagram{1}^{\otimes 2}) = \dim\Hom_{\mathcal{C}}(A \to \ydiagram{1}^{\otimes 2}\otimes\ydiagram{1}^{\otimes -2}).  \]
    The simple decomposition (along with the twists of the summands)
    of $\ydiagram{1}^{\otimes 2}\otimes \ydiagram{1}^{\otimes -2}$ in $\mathcal{C}$ is 
  \[  \begin{tabular}{c|c c}
    Object & Multiplicity & Twist\\ \hline
      $  \left(\emptyset, \emptyset\right)$  & 2 & 1 \\
       $  \left(\ydiagram{1}, \ydiagram{1} \right)$  & 4 & $q^{2N}$ \\   
        $  \left(\ydiagram{2}, \ydiagram{2}\right)$ & 1 &$q^{4N+4}$\\
       $  \left(\ydiagram{2}, \ydiagram{1,1}\right)$ & 1 &$q^{4N}$\\
        $  \left(\ydiagram{1,1}, \ydiagram{2}\right)$ & 1 &$q^{4N}$\\
        $  \left(\ydiagram{1,1}, \ydiagram{1,1}\right)$ & 1 &$q^{4N-4}$,
      \end{tabular}\]
      where $q = e^{2\pi i \frac{1}{2(N+k)}}$.

      As $2<\dim\Hom_{\mathcal{C}}(A \to \ydiagram{1}^{\otimes 2}\otimes\ydiagram{1}^{\otimes -2})$ we must have that at least one of the non-trivial simple summands of $\ydiagram{1}^{\otimes 2}\otimes \ydiagram{1}^{\otimes -2}$ is also a summand of $A$. We saw in the proof of Lemma~\ref{lem:simp} that $q^{2N} \neq 1$ for any parameter choice, hence $\left(\ydiagram{1}, \ydiagram{1} \right)\not\in A$.

      If $  \left(\ydiagram{2}, \ydiagram{2}\right)\in A$ then $q^{4N+4} = 1$. This only has a solution with $N,k \geq 2$ if $k = N+2$. 

      If $  \left(\ydiagram{1,1}, \ydiagram{1,1}\right)\in A$ then $q^{4N-4} = 1$. This only has a solution with $N,k \geq 2$ if $k = N-2$.

      If $  \left(\ydiagram{2}, \ydiagram{1,1}\right)\in A$ then $q^{4N} = 1$. This only has a solution with $N,k \geq 2$ if $k = N$. Furthermore in this case we must also have that $  \left(\ydiagram{1,1}, \ydiagram{2}\right)\in A$ as $A$ is self-dual.
\end{proof}

Our next goal is to bound the multiplicity of the small summands of a 1-super-transitive \`etale algebra $A$. In the case of $k = N$ this falls quickly.
\begin{prop}\label{lem:mult3}
    Let $N,k\geq 4$ and suppose that $A$ is 1-super-transitive and that $k = N$. Then 
    \[  \dim\Hom_{\mathcal{C}}\left( \left(\ydiagram{2},\ydiagram{1,1} \right)\to A\right) = \dim\Hom_\mathcal{C}\left( \left(\ydiagram{1,1},\ydiagram{2} \right)\to A\right)=1.         \]
\end{prop}
\begin{proof}
    As $A$ is self-dual, we have that $\dim\Hom_{\mathcal{C}}\left( \left(\ydiagram{2},\ydiagram{1,1} \right)\to A\right) = \dim\Hom_{\mathcal{C}}\left( \left(\ydiagram{1,1},\ydiagram{2} \right)\to A\right)$ which we denote by $m$. We recall the following fusion rules
    \begin{align*}
        \ydiagram{2}\otimes \ydiagram{2}^* &\cong (\emptyset, \emptyset) \oplus (\ydiagram{1}, \ydiagram{1})\oplus (\ydiagram{2}, \ydiagram{2})\\
        \ydiagram{1,1}\otimes \ydiagram{1,1}^* &\cong (\emptyset, \emptyset) \oplus (\ydiagram{1}, \ydiagram{1})\oplus \left(\ydiagram{1,1}, \ydiagram{1,1}\right)\\
        \ydiagram{2}\otimes \ydiagram{1,1}^* &\cong  (\ydiagram{1}, \ydiagram{1})\oplus \left(\ydiagram{2}, \ydiagram{1,1}\right).
    \end{align*}
    With these rules and our knowledge of the twists of the simple summands, we use the free module functor adjunction to obtain
    \begin{align*}
        \dim\End_{\mathcal{C}_A}(\ydiagram{2})&= 1\\
        \dim\End_{\mathcal{C}_A}\left(\ydiagram{1,1}\right)&= 1\\
        \dim\Hom_{\mathcal{C}_A}\left(\ydiagram{2}\to \ydiagram{1,1}\right)&= m.
    \end{align*}
    Hence $\ydiagram{2}$ and $\ydiagram{1,1}$ are isomorphic simple objects in $\mathcal{C}_A$, so Schurs lemma gives that 
    \[  m =  \dim\Hom_{\mathcal{C}_A}\left(\ydiagram{2}\to \ydiagram{1,1}\right) = 1. \]
\end{proof}

A variation on this idea works in the $k = N \pm 2$ case, except now we have to look at depth 3 objects. This will require the following twist analysis.
\begin{lem}\label{lem:twistanalysis}
    Let $N,k \geq 7$. We have the following decomposition of simples in $\mathcal{C}$:
    \begin{align*}
       \ydiagram{3}\otimes \ydiagram{3}^*&\cong \left(\emptyset, \emptyset\right) \oplus \left(\ydiagram{1}, \ydiagram{1}\right)\oplus \left(\ydiagram{2}, \ydiagram{2}\right)\oplus\left(\ydiagram{3}, \ydiagram{3}\right)\\
        \ydiagram{3}\otimes \ydiagram{2,1}^*&\cong  \left(\ydiagram{1}, \ydiagram{1}\right)\oplus \left(\ydiagram{2}, \ydiagram{2}\right)\oplus\left(\ydiagram{1,1}, \ydiagram{2}\right)\oplus\left(\ydiagram{3}, \ydiagram{2,1}\right)\\
        \ydiagram{3}\otimes \ydiagram{1,1,1}^*&\cong \left(\ydiagram{2},\ydiagram{1,1}\right) \oplus \left(\ydiagram{3},\ydiagram{1,1,1}\right)\\
        \ydiagram{2,1}\otimes \ydiagram{3}^*&\cong  \left(\ydiagram{1}, \ydiagram{1}\right)\oplus \left(\ydiagram{2}, \ydiagram{2}\right)\oplus\left(\ydiagram{2}, \ydiagram{1,1}\right)\oplus\left(\ydiagram{2,1}, \ydiagram{3}\right)\\
        \ydiagram{2,1}\otimes \ydiagram{2,1}^*&\cong \left(\emptyset, \emptyset\right) \oplus 2\left(\ydiagram{1}, \ydiagram{1}\right)\oplus \left(\ydiagram{2}, \ydiagram{2}\right)\oplus\left(\ydiagram{2}, \ydiagram{1,1}\right)\oplus \left(\ydiagram{1,1}, \ydiagram{2}\right)\oplus\left(\ydiagram{1,1}, \ydiagram{1,1}\right)\oplus \left(\ydiagram{2,1}, \ydiagram{2,1}\right) \\
         \ydiagram{2,1}\otimes \ydiagram{1,1,1}^*&\cong \left(\ydiagram{1},\ydiagram{1}\right) \oplus \left(\ydiagram{2},\ydiagram{1,1}\right) \oplus \left(\ydiagram{1,1},\ydiagram{1,1}\right) \oplus \left(\ydiagram{2,1},\ydiagram{1,1,1}\right)\\
         \ydiagram{1,1,1}\otimes \ydiagram{3}^*&\cong \left(\ydiagram{1,1},\ydiagram{2}\right) \oplus \left(\ydiagram{1,1,1},\ydiagram{3}\right)\\
          \ydiagram{1,1,1}\otimes \ydiagram{2,1}^*&\cong \left(\ydiagram{1},\ydiagram{1}\right) \oplus \left(\ydiagram{1,1},\ydiagram{2}\right) \oplus \left(\ydiagram{1,1},\ydiagram{1,1}\right) \oplus \left(\ydiagram{1,1,1},\ydiagram{2,1}\right)\\
         \ydiagram{1,1,1}\otimes \ydiagram{1,1,1}^*&\cong \left(\emptyset, \emptyset\right) \oplus \left(\ydiagram{1}, \ydiagram{1}\right)\oplus \left(\ydiagram{1,1}, \ydiagram{1,1}\right)\oplus\left(\ydiagram{1,1,1}, \ydiagram{1,1,1}\right).
    \end{align*}
  The twists of the simple summands that have not yet been computed in this section are
    \[
\begin{tabular}{c|c c c c c c c c c}
 Object &  $\left(\ydiagram{3}, \ydiagram{3}\right)$  &
        $  \left(\ydiagram{3}, \ydiagram{2,1}\right)$  &
        $  \left(\ydiagram{3}, \ydiagram{1,1,1}\right)$  &
        $  \left(\ydiagram{2,1}, \ydiagram{3}\right)$  &
        $  \left(\ydiagram{2,1}, \ydiagram{2,1}\right)$  &
        $  \left(\ydiagram{2,1}, \ydiagram{1,1,1}\right)$  &
        $  \left(\ydiagram{1,1,1}, \ydiagram{3}\right)$  &
        $  \left(\ydiagram{1,1,1}, \ydiagram{2,1}\right)$  &
        $  \left(\ydiagram{1,1,1}, \ydiagram{1,1,1}\right)$  \\ \hline \\
 Twist &      $q^{6N+12}$
       & $q^{6N+6}$
      & $q^{6N}$
       & $q^{6N+6}$
        &$q^{6N}$
        &$q^{6N-6}$
         &$q^{6N}$
         &$q^{6N-6}$
        &$q^{6N-12}$.  
 \end{tabular}
    \]
    When $k = N-2$ and $N\neq 10$ the only simple summands with trivial twist are $(\emptyset, \emptyset)$ and $\left(\ydiagram{1,1},\ydiagram{1,1}\right)$. When $k = N$ the only simple summands with trivial twist are $(\emptyset, \emptyset)$, $\left(\ydiagram{2},\ydiagram{1,1}\right)$, and $\left(\ydiagram{1,1},\ydiagram{2}\right)$. When $k = N+2$ and $N\neq 8$ the only simple summands with trivial twist are $(\emptyset, \emptyset)$ and $(\ydiagram{2},\ydiagram{2})$.
\end{lem}
\begin{proof}
    Direct computation.
\end{proof}

With this twist analysis we can now pin down the multiplicity in the cases of $k = N \pm 2$.
\begin{prop}\label{lem:mult1}
    Let $N\in \mathbb{N}_{\geq 7} - \{10\}$ and suppose that $A$ is 1-super-transitive and $k = N-2$. Then 
    \[  \dim\Hom_{\mathcal{C}}\left( \left(\ydiagram{1,1},\ydiagram{1,1} \right)\to A\right) = 1.         \]
\end{prop}
\begin{proof}
Define $m:= \dim\Hom_{\mathcal{C}}\left( \left(\ydiagram{1,1},\ydiagram{1,1} \right)\to A\right)$. By Proposition~\ref{prop:levels} we know that $m\geq 1$. Thus it remains to obtain an upper bound.

Using the following simple decompositions in $\mathcal{C}$:
\begin{align*}
    \left(\ydiagram{1,1},\ydiagram{1} \right)\otimes \left(\ydiagram{1,1},\ydiagram{1} \right)^* \cong  &\left(\emptyset, \emptyset \right)\oplus 2\left(\ydiagram{1}, \ydiagram{1}\right)\oplus 2\left(\ydiagram{1,1}, \ydiagram{1,1}\right)\oplus \left(\ydiagram{2}, \ydiagram{2}\right)\oplus \left(\ydiagram{2}, \ydiagram{1,1}\right)\oplus \left(\ydiagram{1,1}, \ydiagram{2}\right)\oplus \left(\ydiagram{1,1}, \ydiagram{1,1}\right)\\
    &\oplus \left(\ydiagram{2,1}, \ydiagram{2,1}\right)\oplus \left(\ydiagram{2,1}, \ydiagram{1,1,1}\right)\oplus \left(\ydiagram{1,1,1}, \ydiagram{2,1}\right)\oplus \left(\ydiagram{1,1,1}, \ydiagram{1,1,1}\right)\\
     \left(\ydiagram{1,1},\ydiagram{1} \right)\otimes \ydiagram{1}^* \cong&\left(\ydiagram{1},\ydiagram{1} \right) \oplus \left(\ydiagram{1,1},\ydiagram{2} \right)\oplus \left(\ydiagram{1,1},\ydiagram{1,1} \right)
\end{align*}
along with Lemma~\ref{lem:twistanalysis} and the free module functor adjunction we compute
\[\dim\End_{\mathcal{C}_A}\left( \left(\ydiagram{1,1},\ydiagram{1} \right)    \right  ) = 1 + 2m \qquad \text{ and } \qquad  \dim\Hom_{\mathcal{C}_A}\left( \left(\ydiagram{1,1},\ydiagram{1} \right)   \to \ydiagram{1} \right  ) = m.\]
As $\ydiagram{1}$ is simple, we must have that $ \left(\ydiagram{1,1},\ydiagram{1} \right)$ decomposes as
\[ \left(\ydiagram{1,1},\ydiagram{1} \right) \cong m \cdot  \ydiagram{1} \oplus \bigoplus \alpha_i \cdot  X_i     \]
    where the $X_i \in \mathcal{C}_A$ are pairwise non-isomorphic simple objects. Further we must have 
    \[      m^2 + \sum \alpha_i^2   = 1 +2m.       \]
    As the $\alpha_i$ are non-negative integers we obtain the inequality
    \[  1 + 2m - m^2 \geq 0       \]
    which is only valid for $m\in \{0,1,2\}$.

    For a contradiction suppose $m=2$. Again using Lemma~\ref{lem:twistanalysis} and the free module functor adjunction we compute 
    \begin{alignat*}{3}
    \dim\End_{\mathcal{C}_A}\left( \ydiagram{1,1}     \right  ) = 1 + m &= 3 \qquad 
    \dim\End_{\mathcal{C}_A}\left( \ydiagram{2}     \right  ) &&= 1   \qquad
    \dim\Hom_{\mathcal{C}_A}\left( \ydiagram{2} \to  \ydiagram{1,1}    \right  ) &= 0\\
    \dim\End_{\mathcal{C}_A}( \ydiagram{3}  ) &= 1 \qquad
    \dim\End_{\mathcal{C}_A}\left( \ydiagram{2,1} \right ) = 1+m &&= 3 \qquad
    \dim\End_{\mathcal{C}_A}\left( \ydiagram{1,1,1}  \right) = 1+m &= 3\\
    \dim\Hom_{\mathcal{C}_A}\left( \ydiagram{3} \to \ydiagram{2,1} \right) &= 0 \qquad
    \dim\Hom_{\mathcal{C}_A}\left( \ydiagram{3} \to \ydiagram{1,1,1} \right) &&= 0 \qquad
    \dim\Hom_{\mathcal{C}_A}\left( \ydiagram{2,1} \to \ydiagram{1,1,1} \right) = m&= 2.
\end{alignat*}
These computations are only compatible with the following simple decompositions of the following free modules:
\begin{align*}
     \ydiagram{1,1} &\cong X_1 \oplus X_2 \oplus X_3\\
     \ydiagram{2,1} &\cong W  \oplus Z_1 \oplus Z_2 \\
     \ydiagram{1,1,1} &\cong Y  \oplus Z_1 \oplus Z_2.
\end{align*}
Where the set $\{X_1, X_2 , X_3, W, Y, Z_1,Z_2\}$ are pairwise non-isomorphic simple objects in $\mathcal{C}_A$.

Using that the free module functor is monoidal, we obtain
\[ (X_1 \otimes \ydiagram{1}) \oplus (X_2 \otimes \ydiagram{1}) \oplus (X_3 \otimes \ydiagram{1}) \cong W \oplus Y \oplus 2 \cdot Z_1 \oplus 2 \cdot Z_2.    \]
On the fusion graph for $\ydiagram{1}$ this means we have 6 edges from the 3 vertices $\{X_1, X_2, X_3\}$ to the 4 vertices $\{W,Y, Z_1, Z_2\}$. Either at least one of the $X_i$ has only one edge out of it, or all $X_i$ have exactly two edges leaving them. Breaking the free symmetries, this means either: 

\begin{enumerate}
    \item $X_1$ connects to $Z_1$ and nothing else, and $X_2$ connects to $Z_1$, in which case $\dim\Hom(X_1 \otimes \ydiagram{1} \to X_3 \otimes \ydiagram{1}) = 0$.
    
    \item $X_1$ connects to $W$ (or $Y$) and nothing else, in which case $\dim\Hom(X_1 \otimes \ydiagram{1} \to X_2 \otimes \ydiagram{1}) = 0$.

     \item $X_1$ connects to $W$ and $Z_1$, $X_2$ connects to $Z_1$ and $Z_2$, and $X_3$ connects to $Z_2$ and $W$, in which case $\dim\Hom(X_1 \otimes \ydiagram{1} \to X_3 \otimes \ydiagram{1}) = 0$.

     \item $X_1$ connects to $W$ and $Y$, $X_2$ connects to $Z_1$ and $Z_2$, and $X_3$ connects to $Z_1$ and $Z_2$, in which case $\dim\Hom(X_1 \otimes \ydiagram{1} \to X_2 \otimes \ydiagram{1}) = 0$.
\end{enumerate} 
In all cases we have a contradiction to Lemma~\ref{lem:basOb} and so $m \neq 2$.
\end{proof}

Essentially the same argument (simply transpose all Young diagrams) gives the following:
\begin{prop}\label{lem:mult2}
    Let $N\in \mathbb{N}_{\geq 7} -\{8\}$ and suppose that $A$ is 1-super-transitive and $k = N+2$. Then 
    \[  \dim\Hom_{\mathcal{C}}\left( \left(\ydiagram{2},\ydiagram{2} \right)\to A\right) = 1.         \]
\end{prop}

With the structure of the simple summands of a 1-super-transitive \`etale now understood up to depth 3, we are now able to determine the fusion graph of $\mathcal{C}_A$ up to depth 3, up to a finite number of possibilities.

\begin{thm}\label{thm:graph2}
    Let $N\in \mathbb{N}_{\geq 7}$, let $A$ a 1-super-transitive \`etale algebra in $\mathcal{C}$ and suppose $k= N$. Then we have the following simple decompositions and isomorphisms of free modules:
    \begin{align*}
        \ydiagram{2,1}&\cong \ydiagram{3}\oplus A \oplus B\\
        \ydiagram{2} &\cong \ydiagram{1,1} \\
         \ydiagram{3} &\cong \ydiagram{1,1,1}
    \end{align*}
    where $\{\ydiagram{2},\ydiagram{3}, A,B\}$ are pairwise non-isomorphic simple objects in $\mathcal{C}_A$. Furthermore, the fusion graph for $\ydiagram{1} \in \mathcal{C}_A$ up to depth 3 is:
    \[    \att{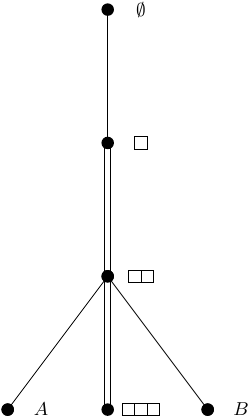}{.6} \]
    In particular we have $\dim \End_{\mathcal{C}_A}(\ydiagram{1}^{\otimes 3}) = 24$.
\end{thm}
\begin{proof}
    We recall from Proposition~\ref{lem:mult3} that
    \[  \dim\Hom_{\mathcal{C}}\left( \left(\ydiagram{2},\ydiagram{1,1} \right)\to A\right) = \dim\Hom_\mathcal{C}\left( \left(\ydiagram{1,1},\ydiagram{2} \right)\to A\right)=1.         \]
    As in the proof of the referenced proposition, we have \[\dim\End_{\mathcal{C}_A}(\ydiagram{2})  = \dim\End_{\mathcal{C}_A}\left(\ydiagram{1,1}\right) = \dim\Hom_{\mathcal{C}_A}\left(\ydiagram{2}\to \ydiagram{1,1}\right) = 1.\]
    This implies that $\ydiagram{2}$ and $\ydiagram{1,1}$ are isomorphic simples in $\mathcal{C}_A$, and that $\ydiagram{1}^{\otimes 2} \cong 2 \cdot \ydiagram{2}$.

    Using Lemma~\ref{lem:twistanalysis} we obtain the following table for the dimensions of the hom spaces in $\mathcal{C}_A$ between the depth 3 free modules $\left\{\ydiagram{3}, \ydiagram{2,1},\ydiagram{1,1,1}\right\}$.
    \[\begin{tabular}{c|c c c}
        $\dim\Hom_{\mathcal{C}_A}(X\to Y)$ & $ \ydiagram{3}$& $\ydiagram{2,1}$ & $\ydiagram{1,1,1}$   \\\hline
        $ \ydiagram{3}$ & 1 & 1 & 1\\
        $ \ydiagram{2,1}$ & 1 & 3 & 1 \\
        $ \ydiagram{1,1,1}$ & 1 & 1 & 1
    \end{tabular}\]
    It follows that $\ydiagram{3}$ and $\ydiagram{1,1,1}$ are isomorphic simples, and that $\ydiagram{2,1}$ decomposes as $\ydiagram{3}\oplus A \oplus B$ where $\{\ydiagram{3}, A, B\}$ are pairwise non-isomorphic.

    Using that the free module functor is monoidal, we compute in $\mathcal{C}_A$ the fusion
    \[  \ydiagram{1}\otimes \ydiagram{2}\cong \ydiagram{3}\oplus \ydiagram{2,1} \cong 2\cdot  \ydiagram{3} \oplus A \oplus B. \]
    The fusion graph for $\ydiagram{1}$ up to depth 3 is thus as in the statement of the Theorem.
\end{proof}

The analogous theorems for $k = N\pm 2$ follow the same proof outline. The only difference is that we now have multiple fusion graphs compatible with decomposition rules for free modules obtained via the free module functor.

\begin{thm}\label{thm:graph1}
    Let $N\in \mathbb{N}_{\geq 7} - \{10\}$, let $A$ a 1-super-transitive \`etale algebra in $\mathcal{C}$ and suppose $k= N-2$. Then we have the following simple decompositions of free modules:
    \begin{align*}
        \ydiagram{1,1}&\cong R \oplus S\\
        \ydiagram{2,1} &\cong X \oplus Y\\
        \ydiagram{1,1,1} &\cong Y \oplus Z
    \end{align*}
    where $\{\ydiagram{2}, \ydiagram{3}, R, S, X, Y, Z\}$ are pairwise non-isomorphic simple objects in $\mathcal{C}_A$. Furthermore, using up the free symmetry $R\leftrightarrow S$ the fusion graph for $\ydiagram{1} \in \mathcal{C}_A$ up to depth 3 is either:
   \[ i),\quad  \att{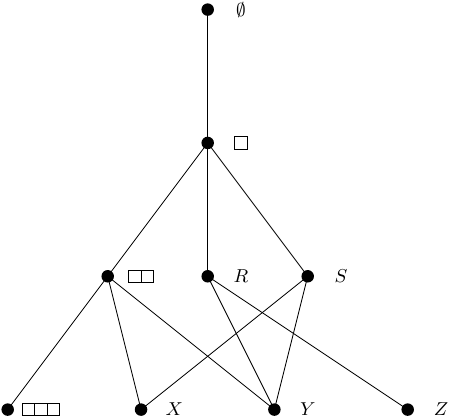}{.6},\qquad\text{or}\qquad  ii),\quad \att{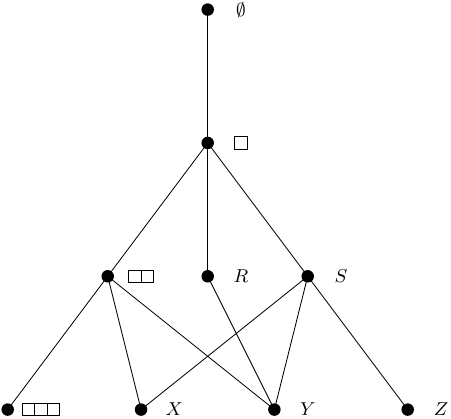}{.6}   .\]
   In particular in both cases we have $\dim \End_{\mathcal{C}_A}(\ydiagram{1}^{\otimes 3}) = 15$.
\end{thm}
\begin{proof}
    This is near identical to the proof of Theorem~\ref{thm:graph2}. There are three possibilities for the fusion graph, but one is incompatible with Lemma~\ref{lem:basOb}.
\end{proof}
The exact same idea works in the remaining case.

\begin{thm}\label{thm:graph3}
     Let $N\in \mathbb{N}_{\geq 7} - \{8\}$, let $A$ a 1-super-transitive \`etale algebra in $\mathcal{C}$ and suppose $k= N+2$. Then we have the following simple decompositions of free modules:
    \begin{align*}
        \ydiagram{2}&\cong R' \oplus S'\\
        \ydiagram{2,1} &\cong X' \oplus Y'\\
        \ydiagram{3} &\cong Y' \oplus Z'
    \end{align*}
    where $\left\{\ydiagram{1,1}, \ydiagram{1,1,1}, R', S', X', Y', Z'\right\}$ are pairwise non-isomorphic simple objects in $\mathcal{C}_A$. Furthermore, using up the free symmetry $R'\leftrightarrow S'$ the fusion graph for $\ydiagram{1} \in \mathcal{C}_A$ up to depth 3 is either:
   \[ i),\quad  \att{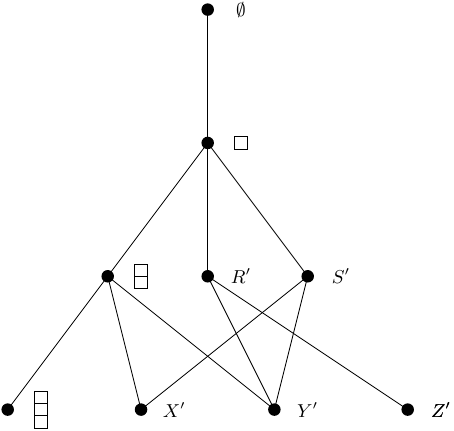}{.6},\qquad\text{or}\qquad  ii),\quad \att{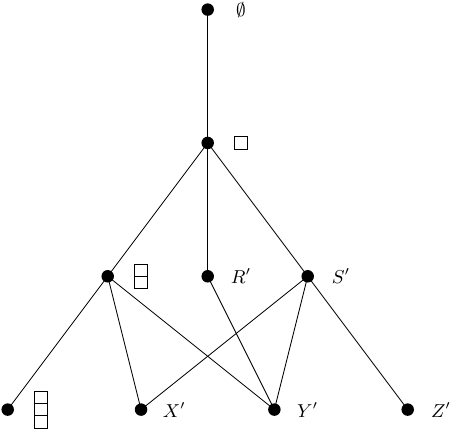}{.6}   .\]
   In particular in both cases we have $\dim \End_{\mathcal{C}_A}(\ydiagram{1}^{\otimes 3}) = 15$.
\end{thm}
\begin{proof}
    Near identical to Theorem~\ref{thm:graph1}
\end{proof}

\section{The case of \texorpdfstring{$k = N$}{k=N}}\label{sec:kN}
The main goal in this section is to classify 1-super-transitive \`etale algebras $A\in \mathcal{C}$ in the $k=N$. We recall that in this case the fusion graph of $\ydiagram{1} \in \mathcal{C}_A$ up to depth 3 is:
 \[    \att{CasekN}{.6} \]
 Just from this information we will be able to deduce that there exists a faithful dominant monoidal functor
 \[  \overline{\mathcal{SE}_N} \to  \mathcal{C}_A   \]
 where $\mathcal{SE}_N$ is the tensor category from Definition~\ref{def:SEN}. With this functor in hand, the classification of $A$ then falls to standard techniques. 

\subsection{Initial Relations}\label{sec:inrel}
From the fusion graph and Theorem~\ref{thm:graph2} we have that $\operatorname{End}_{\mathcal{C}_A}( \ydiagram{1}^{\otimes 2}) \cong M_2(\mathbb{C})$, and that the minimal projections $p_{\ydiagram{2}}, p_{\ydiagram{1,1}} \in \operatorname{End}_{\mathcal{C}}( \ydiagram{1}^{\otimes 2}) = H_2(q)$ remain minimal in $\operatorname{End}_{\mathcal{C}_A}( \ydiagram{1}^{\otimes 2})$, but are now isomorphic. Let 
\[  \att{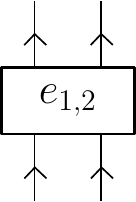}{.5} \in \operatorname{End}_{\mathcal{C}_A}( \ydiagram{1}^{\otimes 2}) \]
be a unitary isomorphism $p_{\ydiagram{2}} \to p_{\ydiagram{1,1}}$, and define
\[  \att{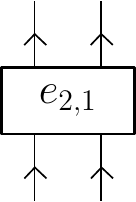}{.5} := \left(\att{e12}{.5}\right)^\dag \]
It follows that $\{e_{1,1} := p_{\ydiagram{2}}, e_{2,2} := p_{\ydiagram{1,1}}, e_{1,2},e_{2,1}\}$
is a system of matrix units for $\operatorname{End}_{\mathcal{C}_A}( \ydiagram{1}^{\otimes 2})$. In particular we have the following relations after translating to the braid basis $\left\{ \att{idid}{.5}, \att{YY}{.5}, \att{e12}{.5} , \att{e21}{.5}    \right\}$.
\begin{align*}
    \att{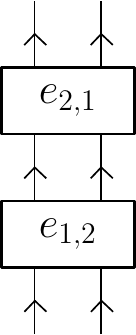}{.5} &= \frac{1}{1+q^2}\left( \att{idid}{.5} + q\att{YY}{.5}  \right), \qquad \qquad \att{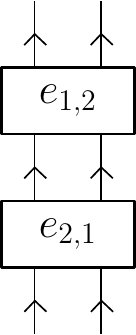}{.5} = \frac{1}{1+q^{-2}}\left( \att{idid}{.5} -q^{-1}\att{YY}{.5}  \right),\\
   \att{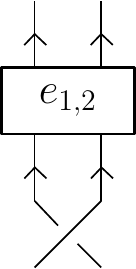}{.5}  &= q\att{e12}{.5},\qquad  \qquad \att{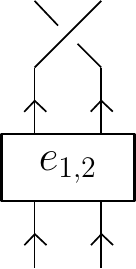}{.5}  = -q^{-1}\att{e12}{.5},  \qquad  \qquad\att{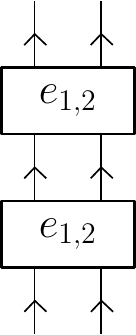}{.5} = 0= \att{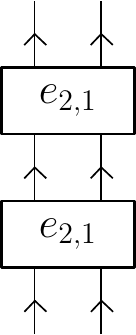}{.5}, \\
     \att{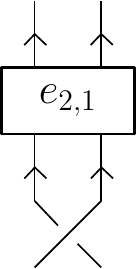}{.5}  &= -q^{-1}\att{e21}{.5},\qquad \hspace{.2em} \att{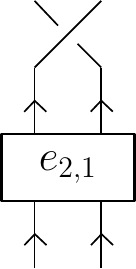}{.5}  = q\att{e21}{.5}.
\end{align*}
As $e_{1,2}$ and $e_{2,1}$ are nilpotent, and $\mathcal{C}_A$ is semisimple, these elements have trace zero (as the categorical trace on $\operatorname{End}_{\mathcal{C}_A}(\ydiagram{1}^{\otimes 2}) \cong M_2(\mathbb{C})$ is equal to the standard matrix trace up to a scalar). Hence we have the partial trace relations:
\[   \att{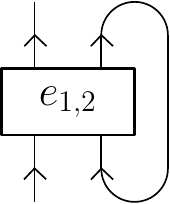}{.5} \quad = \att{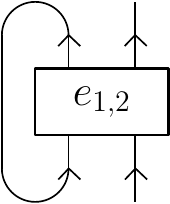}{.5} = \att{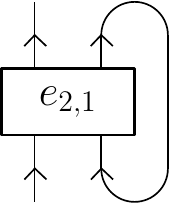}{.5} = \att{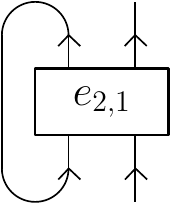}{.5} = 0.\]
From Lemma~\ref{lem:OB} we also have the half-braid style relations
\[          \att{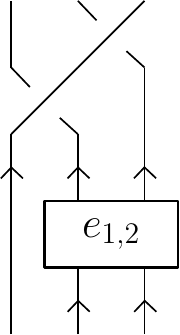}{.5} =   \att{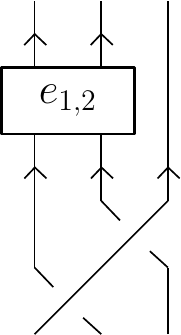}{.5}  \qquad \text{ and } \qquad   \att{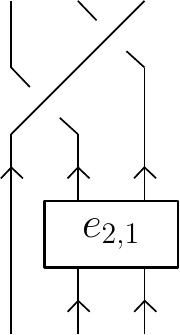}{.5} =   \att{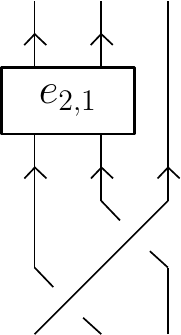}{.5}.\]
Finally we have $U(1)$-gauge freedom for $e_{1,2}$ in the above relations, and hence are free to rescale $e_{1,2}$ by a unimodular scalar.

\subsection{The Convolution Product}\label{sub:conv}
To determine more complicated relations in $\mathcal{C}_A$ we will study the \textit{convolution product} on $\operatorname{End}_{\mathcal{C}_A}( \ydiagram{1}^{\otimes 2})$. This is defined as follows.
\begin{defn}
    For $f,g \in \operatorname{End}_{\mathcal{C}_A}( \ydiagram{1}^{\otimes 2})$, we define $f\bowtie g \in \operatorname{End}_{\mathcal{C}_A}( \ydiagram{1}^{\otimes 2})$ by
    \[  f\bowtie g:= \att{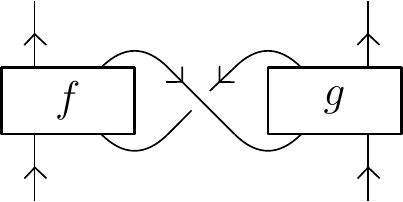}{.5}.  \]
\end{defn}
It follows that $\operatorname{End}_{\mathcal{C}_A}( \ydiagram{1}^{\otimes 2})$ is an associative algebra under the operation $\bowtie$. It turns out that the convolution product is closely related to the standard composition on the space $\operatorname{End}_{\mathcal{C}_A}( \ydiagram{1}\otimes \ydiagram{1}^*)$. To make this precise we define the isomorphism $\tau: \operatorname{End}_{\mathcal{C}_A}(\ydiagram{1}^{\otimes 2}) \to \operatorname{End}_{\mathcal{C}_A}(\ydiagram{1}\otimes \ydiagram{1}^*)$ by
\[  \tau\left( \att{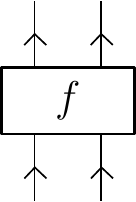}{.5}\right) :=    \att{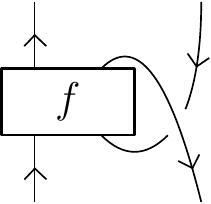}{.5}. \]
We then have the following.
\begin{lem}\label{lem:inter}
    We have
    \begin{align*} 
    \tau\left(\att{e12}{.5}\right)\circ \tau\left(\att{e21}{.5}\right) &= -i\cdot \tau\left(\att{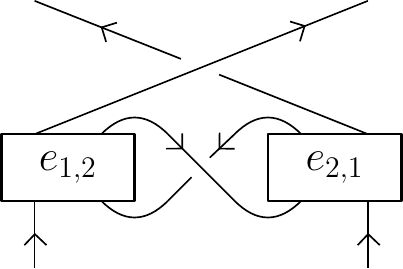}{.5}\right) ,\\
    \tau\left(\att{e21}{.5}\right)\circ \tau\left(\att{e12}{.5}\right) &= -i\cdot \tau\left(\att{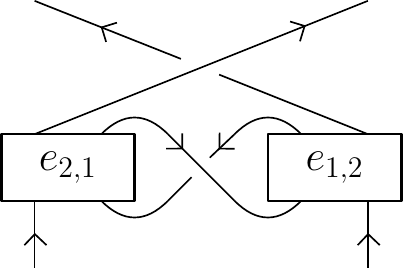}{.5}\right).
    \end{align*}
\end{lem}
\begin{proof}
We compute $\tau(e_{1,2})\circ \tau(e_{2,1})$ as
\[   \att{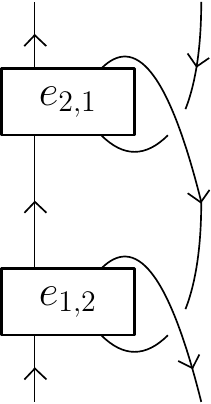}{.45} = \att{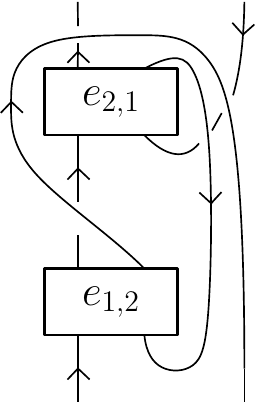}{.45} =-q \att{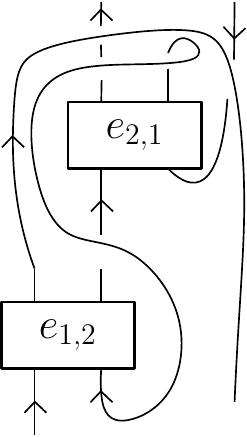}{.45} = (-q)i q^{-1}\att{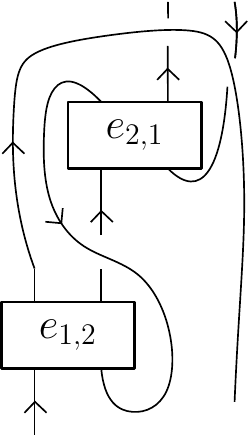}{.45} = -i\cdot \tau\left(\att{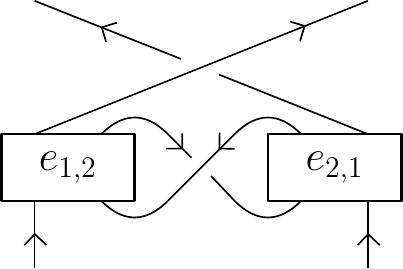}{.45}       \right).    \]
The first result then follows as
\[   \att{tildee12e21brinv}{.5} \quad = \quad \att{tildee12e21br}{.5}  \]
due to the Hecke relation and partial trace relations.

The proof of the second equality is near identical.
\end{proof}

This in particular implies the following commutativity of the $\bowtie$ operation. The more general result that $\bowtie$ is commutative in general is in fact true, but we will not need this stronger result.

\begin{cor}\label{cor:com}
    We have that
    \[   e_{1,2} \bowtie e_{2,1} =  e_{2,1} \bowtie e_{1,2}.   \]
\end{cor}
\begin{proof}
    As $ \operatorname{End}_{\mathcal{C}_A}(\ydiagram{1}^{\otimes 2})$ is 4-dimensional, we also have that $\operatorname{End}_{\mathcal{C}_A}(\ydiagram{1}\otimes \ydiagram{1}^*)$ is 4-dimensional. We know $\mathbf{1}$ is a rank 1 summand of $\ydiagram{1}\otimes \ydiagram{1}^*$. Hence $\operatorname{End}_{\mathcal{C}_A}(\ydiagram{1}\otimes \ydiagram{1}^*)$ contains a subalgebra isomorphic to $\mathbb{C}$. It follows that $\operatorname{End}_{\mathcal{C}_A}(\ydiagram{1}\otimes \ydiagram{1}^*)$ is isomorphic to $\mathbb{C}^4$ and hence commutative. In particular
    \[   \tau\left(\att{e12}{.5}\right)\circ \tau\left(\att{e21}{.5}\right)=   \tau\left(\att{e21}{.5}\right)\circ \tau\left(\att{e12}{.5}\right)  \]
    and the claim follows from Lemma~\ref{lem:inter}.
\end{proof}

As 
$\operatorname{End}_{\mathcal{C}_A}( \ydiagram{1}^{\otimes 2})$ has the basis
\[   \left\{   \att{idid}{.5},\quad \att{YY}{.5}  ,\quad \att{e12}{.5},\quad \att{e21}{.5} \right\}   \]
we have structure constants of $\bowtie$ with respect to this basis. Let $\alpha, \beta, \gamma, \lambda, \delta, \epsilon, \omega, \kappa\in \mathbb{C}$ be the structure constants:
\begin{align*}
    e_{1,2}\bowtie e_{1,2}  &=  \alpha \att{idid}{.5}  +\beta \att{YY}{.5}   + \gamma \att{e12}{.5} + \lambda \att{e21}{.5}, \\
    e_{1,2}\bowtie e_{2,1}  &=  \delta \att{idid}{.5}  +\epsilon \att{YY}{.5}   + \omega \att{e12}{.5} + \kappa\att{e21}{.5} .
\end{align*}
Taking daggers then gives
\begin{align*}
    e_{2,1}\bowtie e_{2,1}  &=  (\overline{\alpha} - (q-q^{-1}) \overline{\beta}) \att{idid}{.5}  +\overline{\beta} \att{YY}{.5}   + \overline{\lambda} \att{e12}{.5} + \overline{\gamma}\att{e21}{.5}, \\
    e_{2,1}\bowtie e_{1,2}  &=  (\overline{\delta} - (q-q^{-1}) \overline{\epsilon}) \att{idid}{.5}  +\overline{\epsilon} \att{YY}{.5}   + \overline{\kappa} \att{e12}{.5} + \overline{\omega}\att{e21}{.5} .
\end{align*}
In particular from Corollary~\ref{cor:com} we see that $\kappa = \overline{\omega}$.

We use up the gauge freedom in $e_{1,2}$ to arrange that $\lambda \geq 0$.


\begin{lem}\label{lem:conv}
We have that
\[    \alpha = \beta = 0 ,\qquad \delta =  - \frac{q-q^{-1}}{(q+q^{-1})^2} \qquad \epsilon =  \frac{2}{(q+q^{-1})^2} \qquad \omega = \overline{\gamma}, \quad \text{ and } \quad \lambda^2 = \frac{2}{(q+q^{-1})^2} + |\gamma|^2.\]
\end{lem}
\begin{proof}
We first compute 
\[\att{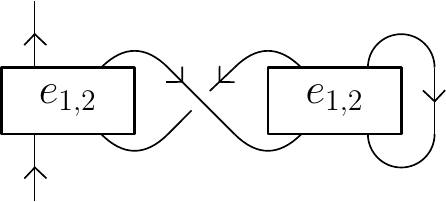}{.5} = 0 \qquad \text{and}\qquad \att{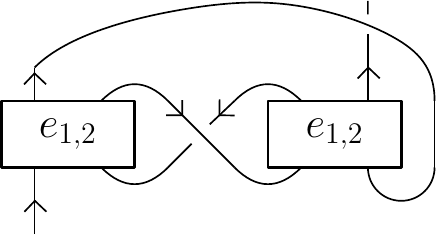}{.5} = 0.\]
This gives the equations
\[    \alpha \frac{2i}{q-q^{-1}} + \beta i = 0 \qquad \text{and}\qquad \alpha i  + \beta\left( \frac{2i}{q-q^{-1}} + i(q-q^{-1})  \right)=0.\]
For $q\neq \pm i$ this linear system has the unique solution $\alpha = \beta = 0$.

In the same fashion we compute 
\[\att{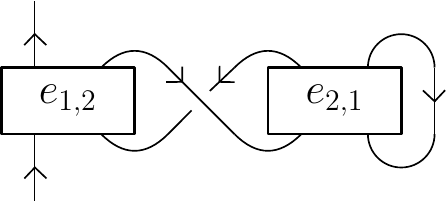}{.5} = 0 \qquad \text{and}\qquad \att{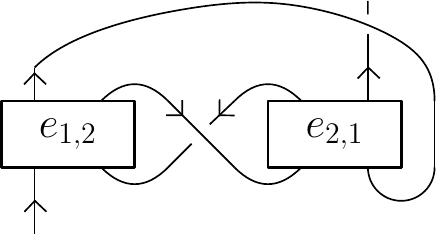}{.5} = \frac{i}{q-q^{-1}}.\]
This gives the equations
\[    \delta \frac{2i}{q-q^{-1}} + \epsilon i = 0 \qquad \text{and}\qquad \delta i  + \epsilon\left( \frac{2i}{q-q^{-1}} + i(q-  q^{-1})  \right)= \frac{i}{q-q^{-1}}\]
which has the unique solution $\delta =  - \frac{q-q^{-1}}{(q+q^{-1})^2}$ and $\epsilon=  \frac{2}{(q+q^{-1})^2}$ if $q \neq \pm i$.

We now use associativity of $\bowtie$ to compute
\begin{align*}( e_{1,2}\bowtie e_{1,2}) \bowtie e_{2,1} &= \gamma \cdot e_{1,2} \bowtie e_{2,1}  + \lambda \cdot e_{2,1} \bowtie e_{2,1}\\
&= \gamma(\delta \att{idid}{.4}  + \epsilon  \att{YY}{.4} + \omega e_{1,2} + \overline{\omega} e_{2,1}) + \lambda( \lambda e_{1,2} + \overline{\gamma} e_{2,1})\\
&= -\gamma  \frac{q-q^{-1}}{(q+q^{-1})^2}\att{idid}{.4}  + \gamma \frac{2}{(q+q^{-1})^2}\att{YY}{.4} + (\gamma \omega + \lambda^2)e_{1,2} + (\gamma \overline{\omega} + \lambda \overline{\gamma})e_{2,1},
\end{align*}
and
\begin{align*}
    e_{1,2}\bowtie ( e_{1,2}\bowtie e_{2,1})   &= \delta \cdot e_{1,2} \bowtie \att{idid}{.4}  + \epsilon \cdot e_{1,2} \bowtie \att{YY}{.4} +\omega \cdot e_{1,2} \bowtie e_{1,2} +\overline{\omega} \cdot e_{1,2} \bowtie e_{2,1}\\
    &=\epsilon \cdot e_{1,2} + \omega(   \gamma e_{1,2} + \lambda e_{2,1}       ) + \overline{\omega}(\delta \att{idid}{.4}  + \epsilon  \att{YY}{.4} + \omega e_{1,2} + \overline{\omega} e_{2,1})\\
    &= \overline{\omega} \frac{q^{-1}-q}{(q+q^{-1})^2} \att{idid}{.4} + \frac{2\cdot  \overline{\omega} }{(q+q^{-1})^2}\att{YY}{.4} + \left(\frac{2}{(q+q^{-1})^2}  + \omega (\gamma + \overline{\omega})\right)e_{1,2} + (\omega \lambda  + \overline{\omega}^2)e_{2,1}.
\end{align*}
This yields
\begin{align*}
    \omega &= \overline{\gamma}\\
     \lambda^2 &= \frac{2}{(q+q^{-1})^2}  +|\gamma|^2.
\end{align*}


\end{proof}

\subsection{A basis for \texorpdfstring{$\operatorname{End}_{\mathcal{C}_A}( \ydiagram{1}^{\otimes 3})$}{the three box space}}

In order to pin down the free coefficient $\gamma$ we will have to study the algebra $\operatorname{End}_{\mathcal{C}_A}( \ydiagram{1}^{\otimes 3})$. Recall from Theorem~\ref{thm:graph2} that this algebra is 24-dimensional. We begin by finding a basis for this algebra.

\begin{defn}\label{def:knBas}
    We define
    \begin{align*} S:= \{&  \att{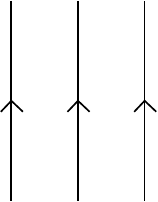}{.5} ,\att{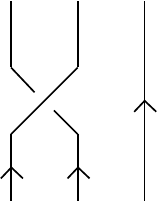}{.5} ,\att{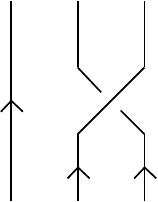}{.5} ,\att{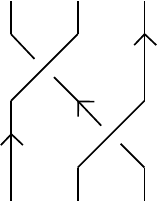}{.5} ,\att{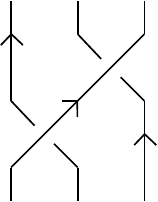}{.5} ,\att{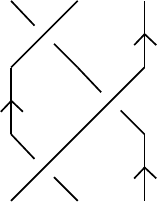}{.5} ,\att{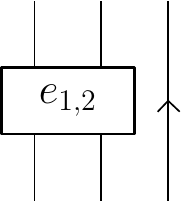}{.5} ,\att{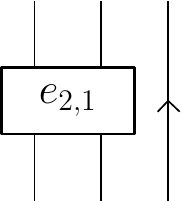}{.5}  ,\att{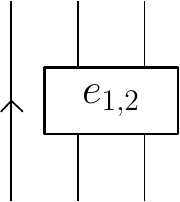}{.5} ,  \\
    & \att{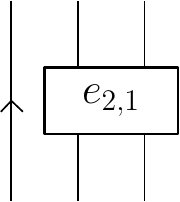}{.5} ,\att{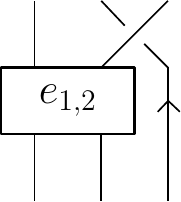}{.5}  ,\att{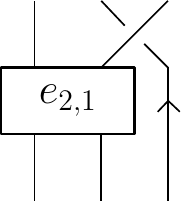}{.5}  ,\att{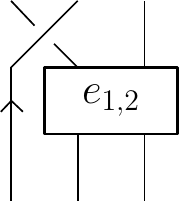}{.5}  ,\att{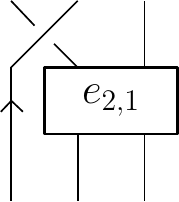}{.5}  ,\att{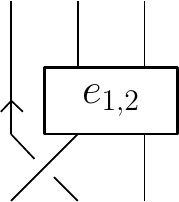}{.5}  ,\att{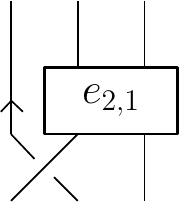}{.5}  ,\att{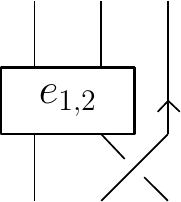}{.5}  ,\\
    & \att{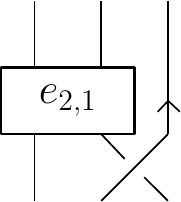}{.5},\att{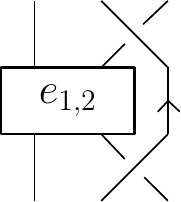}{.5}  ,\att{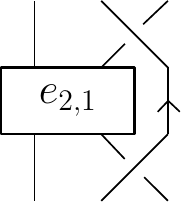}{.5}  ,\att{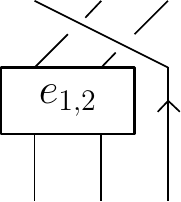}{.5}  ,\att{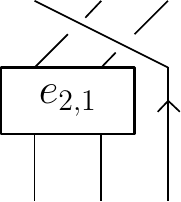}{.5} ,\att{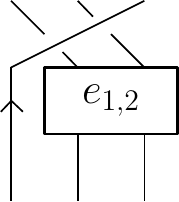}{.5}  ,\att{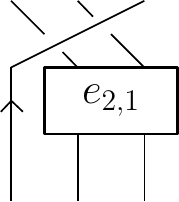}{.5} \} \subset \operatorname{End}_{\mathcal{C}_A}(\ydiagram{1}^{\otimes 3}).
    \end{align*}
\end{defn}
It can be helpful to think of $S$ as the ``braided annular consequences'' of $e_{1,2}$ and $e_{2,1}$. Somewhat surprisingly we can compute the inner-product between any two elements of $S$ just using the relations of Subsection~\ref{sec:inrel}. This allows us to show that $S$ is a basis. 
\begin{lem}
    We have that $S$ is a basis for $ \operatorname{End}_{\mathcal{C}_A}(\ydiagram{1}^{\otimes 3})$ for $N\geq 4$.
\end{lem}
\begin{proof}
    We compute the matrix of inner products for $S$ as:
    \[ \frac{1}{(q-q^{-1})^2}\begin{bmatrix}
        A & 0 & 0 \\
        0 & B & C \\
        0 & C^\dag & D
    \end{bmatrix}\]
    where 
    \begin{align*}
        A &= \begin{bmatrix}
 -\frac{8 i q}{q^2-1} & -4 i & -4 i & -\frac{2 i \left(q^2-1\right)}{q} & -\frac{2 i \left(q^2-1\right)}{q} & -\frac{2 i \left(q^4+1\right)}{q^2} \\
 4 i & -\frac{8 i q}{q^2-1} & \frac{2 i \left(q^2-1\right)}{q} & -4 i & -4 i & -\frac{2 i \left(q^2-1\right)}{q} \\
 4 i & \frac{2 i \left(q^2-1\right)}{q} & -\frac{8 i q}{q^2-1} & -4 i & -4 i & -\frac{2 i \left(q^2-1\right)}{q} \\
 -\frac{2 i \left(q^2-1\right)}{q} & 4 i & 4 i & -\frac{8 i q}{q^2-1} & \frac{2 i \left(q^2-1\right)}{q} & -4 i \\
 -\frac{2 i \left(q^2-1\right)}{q} & 4 i & 4 i & \frac{2 i \left(q^2-1\right)}{q} & -\frac{8 i q}{q^2-1} & -4 i \\
 \frac{2 i \left(q^4+1\right)}{q^2} & -\frac{2 i \left(q^2-1\right)}{q} & -\frac{2 i \left(q^2-1\right)}{q} & 4 i & 4 i & -\frac{8 i q}{q^2-1} \\
\end{bmatrix}\\
B& = \begin{bmatrix}
 -\frac{4 i q}{q^2-1} & 0 & 0 & 0 & -2 i & 0 & 0 & 0 & 0 & 0 \\
 0 & -\frac{4 i q}{q^2-1} & 0 & 0 & 0 & -2 i & 0 & 0 & 0 & 0 \\
 0 & 0 & -\frac{4 i q}{q^2-1} & 0 & 0 & 0 & -2 i & 0 & -2 i & 0 \\
 0 & 0 & 0 & -\frac{4 i q}{q^2-1} & 0 & 0 & 0 & -2 i & 0 & -2 i \\
 2 i & 0 & 0 & 0 & -\frac{4 i q}{q^2-1} & 0 & -2 i q & 0 & 0 & 0 \\
 0 & 2 i & 0 & 0 & 0 & -\frac{4 i q}{q^2-1} & 0 & \frac{2 i}{q} & 0 & 0 \\
 0 & 0 & 2 i & 0 & \frac{2 i}{q} & 0 & -\frac{4 i q}{q^2-1} & 0 & 0 & 0 \\
 0 & 0 & 0 & 2 i & 0 & -2 i q & 0 & -\frac{4 i q}{q^2-1} & 0 & 0 \\
 0 & 0 & 2 i & 0 & 0 & 0 & 0 & 0 & -\frac{4 i q}{q^2-1} & 0 \\
 0 & 0 & 0 & 2 i & 0 & 0 & 0 & 0 & 0 & -\frac{4 i q}{q^2-1} \\
\end{bmatrix}\\
C&=
\begin{bmatrix}
 -2 i & 0 & 0 & 0 & -2 i q & 0 & -2 i q & 0 \\
 0 & -2 i & 0 & 0 & 0 & \frac{2 i}{q} & 0 & \frac{2 i}{q} \\
 0 & 0 & 0 & 0 & \frac{2 i}{q} & 0 & \frac{2 i}{q} & 0 \\
 0 & 0 & 0 & 0 & 0 & -2 i q & 0 & -2 i q \\
 0 & 0 & -2 i & 0 & 2 i q^2 & 0 & 0 & 0 \\
 0 & 0 & 0 & -2 i & 0 & \frac{2 i}{q^2} & 0 & 0 \\
 0 & 0 & 2 i & 0 & \frac{2 i}{q^2}-2 i & 0 & -2 i & 0 \\
 0 & 0 & 0 & 2 i & 0 & 2 i \left(q^2-1\right) & 0 & -2 i \\
 \frac{2 i}{q} & 0 & -2 i & 0 & \frac{2 i}{q^2} & 0 & 0 & 0 \\
 0 & -2 i q & 0 & -2 i & 0 & 2 i q^2 & 0 & 0 \\ 
\end{bmatrix}
\\
D&=\begin{bmatrix}
 -\frac{4 i q}{q^2-1} & 0 & 2 i & 0 & 2 i \left(q^2-1\right) & 0 & -2 i & 0 \\
 0 & -\frac{4 i q}{q^2-1} & 0 & 2 i & 0 & \frac{2 i}{q^2}-2 i & 0 & -2 i \\
 -2 i & 0 & -\frac{4 i q}{q^2-1} & 0 & -\frac{2 i \left(q^2-1\right)}{q} & 0 & -\frac{2 i \left(q^2-1\right)}{q} & 0 \\
 0 & -2 i & 0 & -\frac{4 i q}{q^2-1} & 0 & -\frac{2 i \left(q^2-1\right)}{q} & 0 & -\frac{2 i \left(q^2-1\right)}{q} \\
 2 i-\frac{2 i}{q^2} & 0 & -\frac{2 i \left(q^2-1\right)}{q} & 0 & -\frac{4 i q}{q^2-1} & 0 & \frac{2 i \left(q^2-1\right)}{q} & 0 \\
 0 & -2 i \left(q^2-1\right) & 0 & -\frac{2 i \left(q^2-1\right)}{q} & 0 & -\frac{4 i q}{q^2-1} & 0 & \frac{2 i \left(q^2-1\right)}{q} \\
 2 i & 0 & -\frac{2 i \left(q^2-1\right)}{q} & 0 & \frac{2 i \left(q^2-1\right)}{q} & 0 & -\frac{4 i q}{q^2-1} & 0 \\
 0 & 2 i & 0 & -\frac{2 i \left(q^2-1\right)}{q} & 0 & \frac{2 i \left(q^2-1\right)}{q} & 0 & -\frac{4 i q}{q^2-1} \\
\end{bmatrix}
    \end{align*}
    The determinant of this matrix is
    \[  17179869184 \frac{q^{34} \left(q^2+1\right)^{18} \left(q^4+1\right)^2 \left(q^4-q^2+1\right)^8}{(q-1)^{72} (q+1)^{72}}  \]
    and so the matrix is non-singular for $q\not\in \{0,  i, -i, \zeta_8 , \zeta_8^3, \zeta_8^5, \zeta^7_8, \zeta_{12},\zeta_{12}^5, \zeta_{12}^7, \zeta_{12}^{11}\}$ where $\zeta_\ell = e^{2\pi i \frac{1}{\ell}}$. In particular $S$ is linearly independent for $N \geq 4$. By Theorem~\ref{thm:graph2} we have that for $\operatorname{dimEnd}_{\mathcal{C}_A}(\ydiagram{1}^{\otimes 3}) = 24$, and it follows that $S$ is a basis.
\end{proof}
 With the basis $S$ in hand we can now compute more complicated 3-strand relations in $\mathcal{C}_A$. These will be in terms of the convolution product structure constants.
\begin{lem}
    We have the following relations:
\[  \begin{blockarray}{cccc}
\att{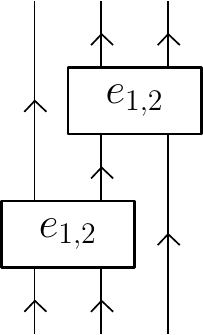}{.5}&\att{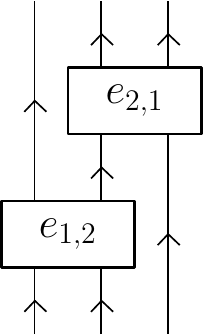}{.5}&\att{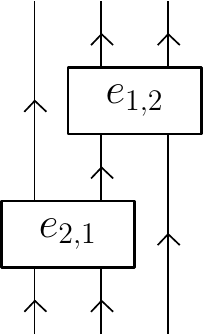}{.5}&\att{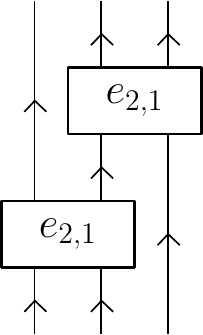}{.5} \\
\begin{block}{[cccc]}
 0 & \frac{2}{\left(q^2+1\right)^2}-\frac{1}{q^4+1} & \frac{q^4 \left(q^2-1\right)^2}{\left(q^2+1\right)^2 \left(q^4+1\right)} & 0 \\
 0 & q \left(\frac{2}{\left(q^2+1\right)^2}-\frac{1}{q^4+1}\right) & -\frac{q^3 \left(q^2-1\right)^2}{\left(q^2+1\right)^2 \left(q^4+1\right)} & 0 \\
 0 & q \left(\frac{2}{\left(q^2+1\right)^2}-\frac{1}{q^4+1}\right) & -\frac{q^3 \left(q^2-1\right)^2}{\left(q^2+1\right)^2 \left(q^4+1\right)} & 0 \\
 0 & \frac{1-q^2}{q^6+q^4+q^2+1} & \frac{q^4 \left(q^2-1\right)}{q^6+q^4+q^2+1} & 0 \\
 0 & \frac{q^2 \left(q^2-1\right)^2}{\left(q^2+1\right)^2 \left(q^4+1\right)} & \frac{q^2 \left(q^2-1\right)^2}{\left(q^2+1\right)^2 \left(q^4+1\right)} & 0 \\
 0 & \frac{q-q^3}{q^6+q^4+q^2+1} & q \left(\frac{1}{q^4+1}-\frac{1}{q^2+1}\right) & 0 \\
 -\frac{\gamma  q \left(q^2-1\right)^2}{2 \left(q^4-q^2+1\right)} & \frac{\left(q^2-1\right)^2 \overline{\gamma} }{2 \left(q^5-q^3+q\right)} & 0 & 0 \\
 0 & 0 & -\frac{q \left(q^2-1\right)^2 \overline{\gamma} }{2 \left(q^4-q^2+1\right)} & \frac{\left(q^2-1\right)^2 \overline{\gamma}}{2 \left(q^5-q^3+q\right)} \\
 \frac{\gamma  \left(q^2-1\right)^2}{2 \left(q^5-q^3+q\right)} & 0 & -\frac{q \left(q^2-1\right)^2 \gamma}{2 \left(q^4-q^2+1\right)} & 0 \\
 0 & \frac{\left(q^2-1\right)^2 \gamma}{2 \left(q^5-q^3+q\right)} & 0 & -\frac{q \left(q^2-1\right)^2 \overline{\gamma}}{2 \left(q^4-q^2+1\right)} \\
 \frac{\gamma  \left(q^2-1\right)^2}{2 \left(q^4-q^2+1\right)} & \frac{\left(q^2-1\right)^2 \overline{\gamma} }{2 \left(q^4-q^2+1\right)} & 0 & 0 \\
 0 & 0 & \frac{\left(q^2-1\right)^2 \overline{\gamma} }{2 \left(q^4-q^2+1\right)} & \frac{\left(q^2-1\right)^2 \overline{\gamma}}{2 \left(q^4-q^2+1\right)} \\
 -\frac{\gamma  q^2 \left(q^2-1\right)}{2 \left(q^4-q^2+1\right)} & \frac{\overline{\gamma} -q^2 \overline{\gamma} }{2 q^4-2 q^2+2} & \frac{q^2 \left(q^2-1\right) \gamma}{2 \left(q^4-q^2+1\right)} & \frac{\left(q^2-1\right) \lambda}{2 q^2} \\
 -\frac{1}{2} \lambda  (q-1) (q+1) & -\frac{\left(q^2-1\right) \gamma}{2 \left(q^4-q^2+1\right)} & \frac{q^2 \left(q^2-1\right) \overline{\gamma} }{2 \left(q^4-q^2+1\right)} & \frac{\left(q^2-1\right) \overline{\gamma}}{2 \left(q^4-q^2+1\right)} \\
 \frac{\gamma  \left(q^2-1\right)^2}{2 \left(q^4-q^2+1\right)} & 0 & \frac{\left(q^2-1\right)^2 \gamma}{2 \left(q^4-q^2+1\right)} & 0 \\
 0 & \frac{\left(q^2-1\right)^2 \gamma}{2 \left(q^4-q^2+1\right)} & 0 & \frac{\left(q^2-1\right)^2 \overline{\gamma}}{2 \left(q^4-q^2+1\right)} \\
 \frac{\gamma  \left(q^2-1\right)}{2 \left(q^4-q^2+1\right)} & \frac{\overline{\gamma} -q^2 \overline{\gamma} }{2 q^4-2 q^2+2} & \frac{q^2 \left(q^2-1\right) \gamma}{2 \left(q^4-q^2+1\right)} & -\frac{1}{2} \left(q^2-1\right) \lambda \\
 \frac{\lambda  \left(q^2-1\right)}{2 q^2} & -\frac{\left(q^2-1\right) \gamma}{2 \left(q^4-q^2+1\right)} & \frac{q^2 \left(q^2-1\right) \overline{\gamma} }{2 \left(q^4-q^2+1\right)} & -\frac{q^2 \left(q^2-1\right) \overline{\gamma}}{2 \left(q^4-q^2+1\right)} \\
 -\frac{\gamma  q \left(q^2-1\right)}{2 \left(q^4-q^2+1\right)} & -\frac{\left(q^2-1\right) \overline{\gamma} }{2 \left(q^5-q^3+q\right)} & -\frac{q^3 \left(q^2-1\right) \gamma}{2 \left(q^4-q^2+1\right)} & -\frac{\left(q^2-1\right) \lambda}{2 q} \\
 -\frac{\lambda  \left(q^2-1\right)}{2 q} & -\frac{\left(q^2-1\right) \gamma}{2 \left(q^5-q^3+q\right)} & -\frac{q^3 \left(q^2-1\right) \overline{\gamma} }{2 \left(q^4-q^2+1\right)} & -\frac{q \left(q^2-1\right) \overline{\gamma}}{2 \left(q^4-q^2+1\right)} \\
 0 & 0 & 0 & 0 \\
 0 & 0 & 0 & 0 \\
 \frac{\gamma  q \left(q^2-1\right)}{2 \left(q^4-q^2+1\right)} & -\frac{q \left(q^2-1\right) \overline{\gamma} }{2 \left(q^4-q^2+1\right)} & -\frac{q \left(q^2-1\right) \gamma}{2 \left(q^4-q^2+1\right)} & \frac{\left(q^2-1\right) \lambda}{2 q} \\
 \frac{\lambda  \left(q^2-1\right)}{2 q} & -\frac{q \left(q^2-1\right) \gamma}{2 \left(q^4-q^2+1\right)} & -\frac{q \left(q^2-1\right) \overline{\gamma} }{2 \left(q^4-q^2+1\right)} & \frac{q \left(q^2-1\right) \overline{\gamma}}{2 \left(q^4-q^2+1\right)} \\
\end{block}
\end{blockarray} \]
Here the columns are indexed by the elements of the basis $S$.
\end{lem}
\begin{proof}
    We demonstrate the proof for the first column. The remaining three follow the same idea.

    We compute the following using the relations of $\mathcal{C}_A$ from Subsections~\ref{sec:inrel} and \ref{sub:conv}.
    \begin{alignat*}{3}
        \att{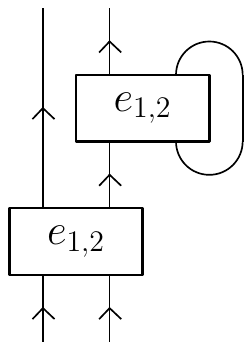}{.4} &= 0 , 
        \att{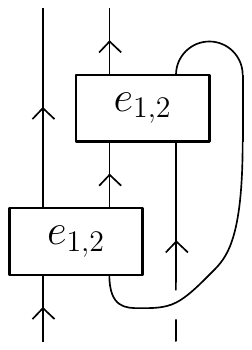}{.4} &= iq \left(\gamma \att{e12}{.4}+\lambda  \att{e21}{.4}\right),
         \att{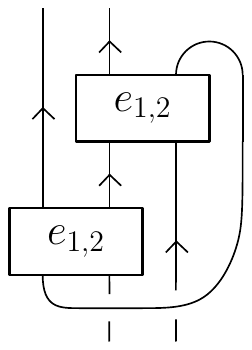}{.4} &= i\left(\gamma \att{e12}{.4}+\lambda  \att{e21}{.4}\right),\\
         \att{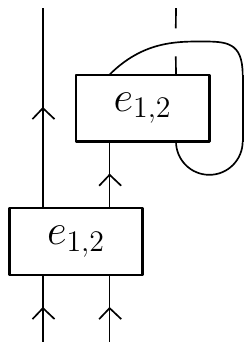}{.4} &= 0,
         \att{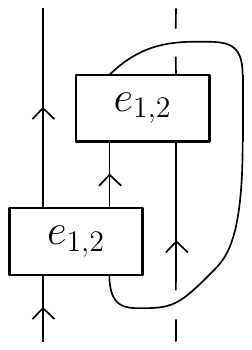}{.4} &= -i \left(\gamma \att{e12}{.4}+\lambda  \att{e21}{.4}\right),
         \att{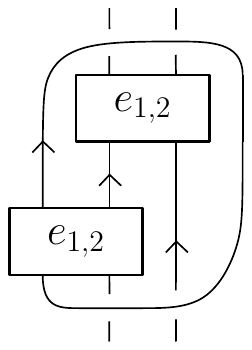}{.4} &= 0\\
         \att{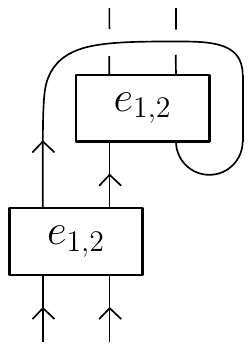}{.4} &= 0,
         \att{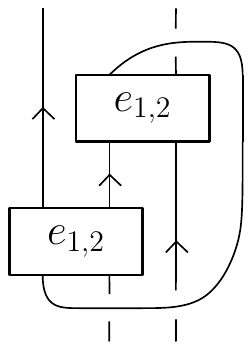}{.4} &= -i q^{-1}\left(\gamma \att{e12}{.4}+\lambda  \att{e21}{.4}\right), \att{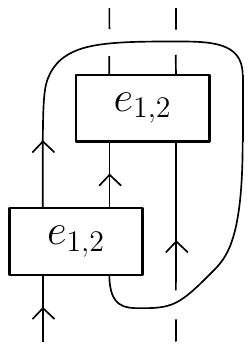}{.4} &= 0.
    \end{alignat*}
    On the other hand we can compute these partial traces for each of the elements of $S$. As $S$ is a basis for $\operatorname{End}_{\mathcal{C}_A}(\ydiagram{1}^{\otimes 3})$ we have that $\att{doubleC}{.3}$ can be written as a linear combination of the elements of $S$. Each of the above equations gives four individual equations (one for each basis element of $\operatorname{End}_{\mathcal{C}_A}(\ydiagram{1}^{\otimes 2})$) in the coefficients of $\att{doubleC}{.3}$. Solving these 36 equations yields the unique solution as in the statement of the lemma.
\end{proof}

At this point it becomes useful to write down explicit matrix representations of $\operatorname{End}_{\mathcal{C}_A}(\ydiagram{1}^{\otimes 3})$ for the elements 
\[   \att{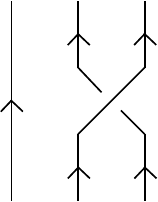}{.5},\qquad  \att{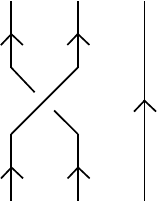}{.5} , \qquad \att{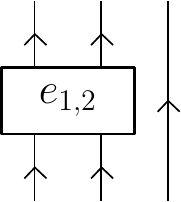}{.5} \quad \text{ and } \att{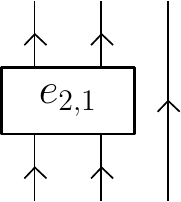}{.5}.\]
Note that via the half-braid relation, these four elements generate the algebra $\operatorname{End}_{\mathcal{C}_A}(\ydiagram{1}^{\otimes 3})$ \footnote{We thank Hans Wenzl for bring this to our attention.}. That is we give the matrix for pre-composition on $\operatorname{End}_{\mathcal{C}_A}(\ydiagram{1}^{\otimes 3})$ with respect to the basis $S$ for these four elements. These 24 dimensional matrices are found in the Appendix~\ref{app:A}, as well as in a Mathematica file attached to the arXiv submission of this paper. In this Mathematica file we also include the subsequent computations of this section which use our matrix representation. Note that as this representation is the regular representation of the algebra, it is faithful. While the following results of this section could be obtained by hand, these matrix representations will allow for computer assisted computations in $\operatorname{End}_{\mathcal{C}_A}(\ydiagram{1}^{\otimes 3})$, which are fast, and more accurate than a human. Using this matrix representation we can now pin down our convolution structure constants.
\begin{lem}\label{lem:param}
    We have 
    \[   \lambda = \frac{\sqrt{2} i}{(q^2 - q^{-2})\sqrt{q^2 + q^{-2}}},\qquad \text{and} \qquad  \gamma = \frac{(q^2 - 1 + q^{-2})\sqrt{2} i}{(q^2 - q^{-2})\sqrt{q^2 + q^{-2}}}.      \]
\end{lem}
\begin{proof}
We recall the relation
\[  \att{matRel3}{.5} = \frac{1}{1+q^2}\left( \att{idid}{.5} + q\att{YY}{.5}  \right).\]
Using the standard left embedding of this relation into $\operatorname{End}_{\mathcal{C}_A}(\ydiagram{1}^{\otimes 3})$, and using our explicit matrix representation, we obtain the equality of two explicit $24\times 24$ matrices. The $(9,13)$-th entry gives the equality
\[q \left(q^2-1\right)^2 \left(q^2+1\right) \left(\gamma ^2-\overline{\gamma}^2\right)=0.\]
This implies that $\gamma^2 = \overline{\gamma}^2$.

We now consider the $(7,9)$-th entry, which gives
\[q \left(q^2-1\right) \left(q^2 \left(q^{10}+q^8-q^2-1\right) \overline{\gamma}^2+4 \left(q^5-q^3+q\right)^2+\gamma ^2 (q-1) (q+1) \left(q^2-2\right) \left(q^2+1\right)^2 \left(q^4+1\right)\right)=0.\]
Solving this equation gives that
\[\gamma = \pm \frac{(q^2 - 1 + q^{-2})\sqrt{2} i}{(q^2 - q^{-2})\sqrt{q^2 + q^{-2}}}\in \mathbb{R}.\]

Recalling the equation
\[\lambda^2 = \frac{2}{(q+q^{-1})^2} + |\gamma|^2\]
from Lemma~\ref{lem:conv} we obtain (recalling the positive gauge choice for $\lambda$)
\[ \lambda = \frac{\sqrt{2} i}{(q^2 - q^{-2})\sqrt{q^2 + q^{-2}}}.\]
Finally we consider the $(14,9)$-th entry to obtain
\[  \gamma  \left(q^4-1\right) \left(\lambda  \left(q^4-q^2+1\right)-q^2 \overline{\gamma}\right) = 0.  \]
This equation is only compatible with the solution as in the statement of the lemma.

\end{proof}


We have now explicitly determined all relations in $\mathcal{C}_A$ up to the 3-box space. While we will not prove it, this is sufficient to evaluate, and hence we have uniqueness of the subcategory generated by the braid and 2-box matrix units. By uniqueness this category must be $\mathcal{SE}_N$. Instead, we take the more constructive approach of explicitly building a functor from $\mathcal{SE}_N$ into $\mathcal{C}_A$. A benefit of this approach is that we get a sanity check on our previous computations.

\begin{thm}\label{thm:classN}
    Let $A\in \mathcal{C}$ be 1-super-transitive, and let $k= N$. Then there exists a faithful dominant monoidal functor
    \[   \overline{ \mathcal{SE}_N } \to \mathcal{C}_A                 \]
    In particular, $A$ is a simple current extension of the algebra $A_{\mathfrak{so}_{N^2-1}}$.
\end{thm}
\begin{proof}
    We define the functor $\mathcal{F}$ which sends $+\mapsto \mathcal{F}_A(\ydiagram{1})$, $-\mapsto \mathcal{F}_A(\ydiagram{1}^*)$, and
    \begin{align*}
        \att{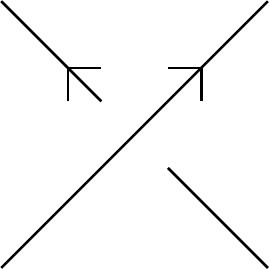}{.25}&\mapsto \att{b1}{.25}\\
        \att{kw}{.3}&\mapsto \att{kw}{.3}\\
        \att{splittingEnd2}{.25}& \mapsto   -\frac{\left(q^2-1\right)^2}{\left(q^2+1\right)^2} \att{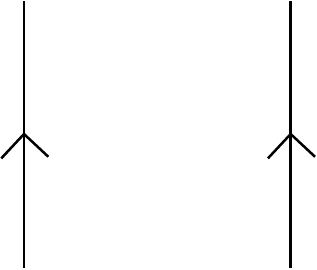}{.25} + \frac{2 q \left(q^2-1\right)}{\left(q^2+1\right)^2} \att{b1}{.25} + \frac{i \sqrt{2} \sqrt{q^4+1}}{q^2+1}\left( \att{e12}{.4} + \att{e21}{.4}\right).
    \end{align*}
    We now just need to check that the defining relations of $\mathcal{SE}_N$ are preserved by this functor. While this could be done by hand, we instead use our explicit matrix representation of $\operatorname{End}_{\mathcal{C}_A}(\ydiagram{1}^{\otimes 3})$. This allows a computer to verify relations (Exchange), (Slide), (Stack), and (Commute). We include this computer verification in the arXiv submission of this paper. Note that (Over-Braid) holds automatically. The only relations we need to check by hand are the partial trace relations. The right partial trace of the image of $\att{splittingEnd2}{.2}$ is 
    \[   -\frac{\left(q^2-1\right)^2}{\left(q^2+1\right)^2} \frac{2i}{q-q^{-1}} + \frac{2 q \left(q^2-1\right)}{\left(q^2+1\right)^2}i = 0   \]
    as required, and the right partial trace of the image of $\att{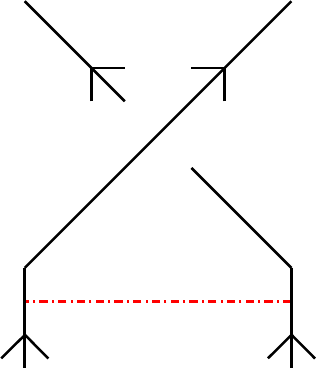}{.2}$ is 
    \[    -\frac{\left(q^2-1\right)^2}{\left(q^2+1\right)^2} i + \frac{2 q \left(q^2-1\right)}{\left(q^2+1\right)^2}\left( \frac{2i}{q-q^{-1}} + (q-q^{-1})i \right) = i   \]
    as required. 
    
    It follows that $\mathcal{F}$ is a dominant monoidal functor $\mathcal{SE}_N \to \mathcal{C}_A$. As $\mathcal{C}_A$ is unitary, we get that $\mathcal{F}$ descends to faithful dominant monoidal functor 
    \[   \overline{ \mathcal{SE}_N } \to \mathcal{C}_A .                \]

    By Theorem~\ref{thm:me} we have that $\operatorname{Ab}(\overline{ \mathcal{SE}_N }) \simeq \mathcal{C}_{ A_{\mathfrak{so}_{N^2-1}}  }$. As in the proof of \cite[Theorem 4.17]{me}, this implies that $A_{\mathfrak{so}_{N^2-1}}$ is an \`etale sub-algebra of $A$. By \cite[Section 3.6]{LagrangeUkraine} we have that \`etale algebra extensions of $A_{\mathfrak{so}_{N^2-1}}$ are in bijection with \`etale algebras in $\mathcal{C}(\mathfrak{sl}_N, N)_{A_{\mathfrak{so}_{N^2-1}}}$. By \cite[Theorem 5.2]{kril} we have that there is a braided equivalence $\mathcal{C}(\mathfrak{sl}_N, N)_{A_{\mathfrak{so}_{N^2-1}}} \cong \mathcal{C}(\mathfrak{so}_{N^2-1}, 1)$. This category is either equivalence to an Ising category, or is pointed. In either case we have that all \`etale algebras in these categories are pointed. Thus $A$ is a simple current extension of $A_{\mathfrak{so}_{N^2-1}}$.
\end{proof}

\section{Classification in the case of \texorpdfstring{$k = N\pm 2$}{k = N +-2}}\label{sec:kN2}
In this section we deal with the cases of $k = N\pm 2$. As these two cases are level-rank dual to each other we will detail the proof in the $k= N-2$ case, and leave the $k = N+2$ case to the reader (where one just transposes all Young diagrams and replaces the braid action on the projection with $q$ instead of $-q^{-1}$).

Let $A$ be a 1-super-transitive \`etale algebra in $\mathcal{C}$, and let $k = N-2$. Our goal is to classify $A$ by showing that there is a faithful dominant functor
\[     \overline{\mathcal{SD}^-_N} \to \mathcal{C}_A.      \]
We recall from Theorem~\ref{thm:graph1} that we have two cases for the fusion graph of $\mathcal{C}_A$ up to depth 3. These are
 \[ i),\quad  \att{CaseHard}{.6},\qquad\text{or}\qquad  ii),\quad \att{CaseKnown}{.6}   .\]
 These two cases we will deal with separately. For case i) we will use similar techniques as in Section~\ref{sec:kN} to derive a contradiction, and hence rule out this fusion graph.
\subsection{Initial Relations}

Before we branch into the two cases, we can find some basic relations in general. Let $r$ be the unique minimal central projection onto $R$ in $\operatorname{End}_{\mathcal{C}_A}(\ydiagram{1}^{\otimes 2})$. As $R$ is a minimal central subprojection of $p_{\ydiagram{1,1}}$ it follows that we have the relation
\[   \att{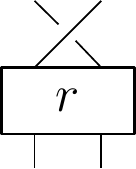}{.5} =\att{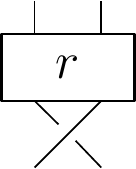}{.5} =  -q^{-1} \att{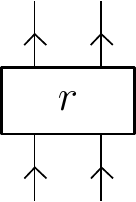}{.5}.\]
Further, from \cite{me} we have the half-braid relation
\[   \att{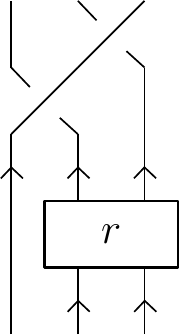}{.5} = \att{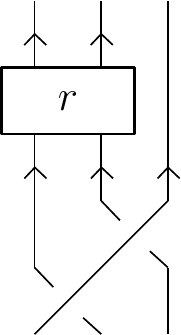}{.5}.   \]
It is convenient to recall that $q^N = i \cdot q$ in this case. Let us write $d_R$ for $\dim(R)$, and let us write
\[ d_{\ydiagram{1}} := \dim(\ydiagram{1}) = i \frac{q + q^{-1}}{q-q^{-1}}.  \]
We recall that we have two possibilities for the fusion graph of $\square$ in $\mathcal{C}_A$ up to depth three. To obtain further relations we have to analyze each case in detail.



\subsection{Case i)}

We recall that in this case the fusion graph for $\ydiagram{1}$ in $\mathcal{C}_A$ up to depth three is
\[\att{CaseHard}{.6}\]
From this we obtain two facts which will ultimately lead to a contradiction. The first is that
\[     R \otimes \ydiagram{1} \cong Y\oplus Z \cong \ydiagram{1,1,1}.    \]
In particular this implies that 
\[   d_R = \frac{\dim\left({\ydiagram{1,1,1}}\right)}{d_{\ydiagram{1}}} = \frac{2}{-q^{-4} + q^{-2} + q^2 - q^4}  , \]
and so we have the partial trace relation
\[  \att{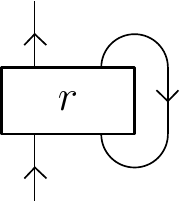}{.5} =  \frac{2i}{-q^{-4} - q^{-2} + q^2 + q^4}\att{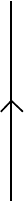}{.5} =  \att{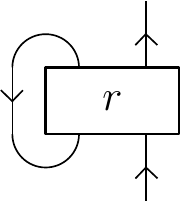}{.5} \]
The second fact is that
\[  \dim\Hom_{\mathcal{C}_A}(R  \otimes \ydiagram{1}\to  \ydiagram{2}\otimes \ydiagram{1}) = 1.  \]

We begin by studying the convolution product of $r$ with itself. To simplify computations, we will work with the basis $\left\{  p_{\ydiagram{2}}, p_{\ydiagram{1,1}}, r\right\}$ for $\operatorname{End}_{\mathcal{C}_A}(+^{\otimes 2})$. It follows that there exist scalars $\alpha, \beta, \gamma \in \mathbb{C}$ such that
\[\att{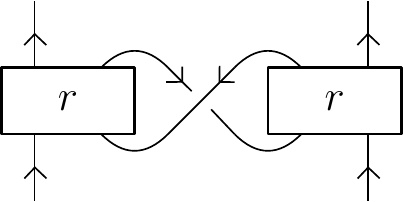}{.5} = \alpha \att{p2a}{.5}+\beta \att{p22}{.5} + \gamma \att{rProj}{.5}.\]

\begin{lem}
    We have that $\gamma \in \mathbb{R}$, and that 
    \begin{align*}
        \alpha &= -\frac{2 q^5 \left(q^4+q^2-1\right)}{\left(q^2-1\right) \left(q^2+1\right)^2 \left(q^4+q^2+1\right)^2}\\ 
        \beta &= \frac{q^2 \left(q \left(q^8-q^6-2 q^4+q^2-1\right)-\gamma  \left(q^2+1\right)^2 \left(q^6-1\right)\right)}{\left(q^2-1\right) \left(q^2+1\right)^2 \left(q^4+q^2+1\right)^2}.
    \end{align*}
\end{lem}
\begin{proof}
    We compute
    \[  \att{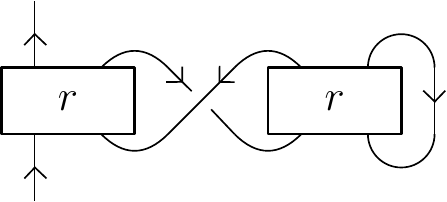}{.5}= -\frac{4 i q}{\left(q^4 + q^2- \frac{1}{q^2}-\frac{1}{q^4}\right)^2}\att{str}{.5}.  \]
    Hence
    \begin{align*}   -\frac{4 i q}{\left(q^4 + q^2- \frac{1}{q^2}-\frac{1}{q^4}\right)^2} &= \alpha\frac{\dim(\ydiagram{2})}{d_{\ydiagram{1}}} +\beta\frac{\dim(\ydiagram{1,1})}{d_{\ydiagram{1}}} + \gamma   \frac{2i}{-q^{-4} - q^{-2} + q^2 + q^4}  \\ &= \frac{i \left(2 \gamma  q^4+\left(q^4+q^2+1\right) \left(\alpha +\alpha  q^4+2 \beta  q^2\right)\right)}{q^8+q^6-q^2-1}. \end{align*}
    We also compute 
    \[    \att{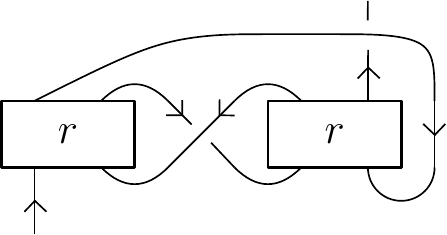}{.5}  =   -\frac{2 i \left(q^{10}+2 q^8+q^4\right)}{\left(q^2-1\right) \left(q^6+2 q^4+2 q^2+1\right)^2}. \]
    Thus
   \begin{align*}   -\frac{2 i \left(q^{10}+2 q^8+q^4\right)}{\left(q^2-1\right) \left(q^6+2 q^4+2 q^2+1\right)^2} &= \alpha q\frac{\dim(\ydiagram{2})}{d_{\ydiagram{1}}} -\beta q^{-1}\frac{\dim(\ydiagram{1,1})}{d_{\ydiagram{1}}} - \gamma q^{-1}  \frac{2i}{-q^{-4} - q^{-2} + q^2 + q^4}  \\ &= \frac{i \left(\left(q^5+q^3+q\right) \left(\alpha -2 \beta +\alpha  q^4\right)-2 \gamma  q^3\right)}{q^8+q^6-q^2-1}. \end{align*}
   Solving this linear system gives the solutions for $\alpha$ and $\beta$ as in the statement of the lemma.

   Taking daggers gives 
   \begin{align*}   \att{rConv}{.5}^\dag &=\att{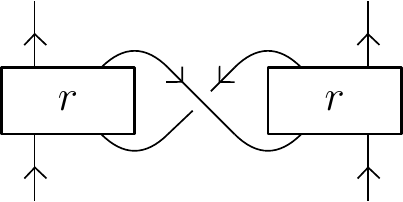}{.5} \\
   &= \att{rConv}{.5} + (q-q^{-1})\frac{4 q^8}{\left(q^8+q^6-q^2-1\right)^2} \att{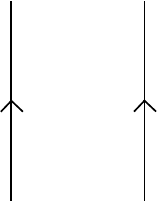}{.5}    \end{align*}
   This implies that $\gamma = \overline{\gamma}$.
\end{proof}

As in the $k= N$ case, we will study $\End_{\mathcal{C}_A}(\ydiagram{1}^{\otimes 3})$ in order to determine $\gamma$. We recall from Theorem~\ref{thm:classN-2} that this space is 15-dimensional.

\begin{defn}
We define   
\begin{align*} S:= \{&  \att{c1}{.5} ,\att{c2}{.5} ,\att{c3}{.5} ,\att{c4}{.5} ,\att{c5}{.5} ,\att{c6}{.5} ,\att{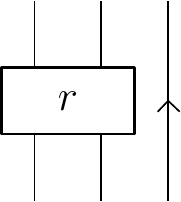}{.5} ,\att{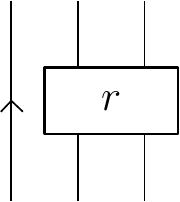}{.5}  ,\att{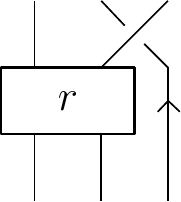}{.5} ,  \\
    & \att{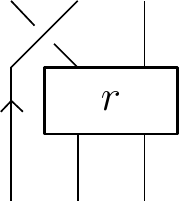}{.5}  ,\att{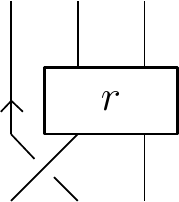}{.5}  ,\att{r12r}{.5}   ,\att{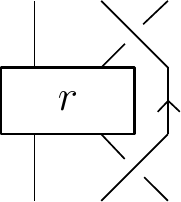}{.5}   ,\att{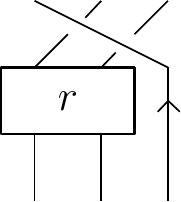}{.5}   ,\att{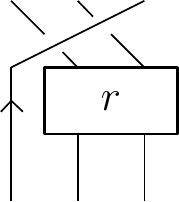}{.5}  ,\att{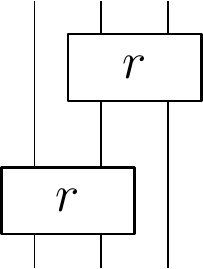}{.5}   \}\subset \End_{\mathcal{C}_A}(+^{\otimes 3}).
    \end{align*}
\end{defn}
Note that this set over-spans $\End_{\mathcal{C}_A}(\ydiagram{1}^{\otimes 3})$, and so there must be a linear relation between these elements. Using our known relations we can compute the inner-product between any two elements of $S$. This computation yields a $16\times 16$ matrix over $\mathbb{C}(q,\gamma)$ which is too large to include here. This matrix along with the subsequent computations can be found in a Mathematica file attached to the arXiv submission of this paper. As a consequence of this matrix of inner products, we obtain the following.
\begin{lem}\label{lem:rrRel}
    We have that the first 15 elements of $S$ are a basis of $\End_{\mathcal{C}_A}(\ydiagram{1}^{\otimes 3})$ for $N\geq 5$. Furthermore for $N\geq 5$ we have that $\gamma$ is one of 
    \[  \gamma_{\pm} :=   i \frac{ \pm q^3 + 2i q - 2i q^{-1} \pm q^{-3}}{ q^4 + q^2 -q^{-2} - q^{-4}  }.  \]
    For these two solutions of $\gamma$ we have that the element $\att{rr}{.5}$ can be written with respect to the basis of the first 15 elements of $S$ as
    \[\frac{1}{q^3 + 2q + 2q^{-1} + q^{-3}}\begin{bmatrix}\frac{(1+i) q^3 \left(i q^4+q^2-1\right)}{q^6+2 q^4+2 q^2+1}\\\frac{(1+i) \left(-i q^6-q^4+q^2\right)}{q^6+2 q^4+2 q^2+1}\\\frac{(1+i) \left(-i q^6-q^4+q^2\right)}{q^6+2 q^4+2 q^2+1}\\\frac{(1+i) \left(q^3+i q^5\right)}{q^6+2 q^4+2 q^2+1}\\\frac{(1+i) q \left(i q^4+q^2-1\right)}{q^6+2 q^4+2 q^2+1}\\\frac{(1-i) q^2 \left(q^2-i\right)}{q^6+2 q^4+2 q^2+1}\\\frac{(1-i) q^4+(1+i)}{2 q}\\\frac{(1-i) q^4+(1+i)}{2 q}\\-\frac{\left(\frac{1}{2}-\frac{i}{2}\right) \left(q^4+i\right)}{q^2}\\i\\-\frac{\left(\frac{1}{2}-\frac{i}{2}\right) \left(q^4+i\right)}{q^2}\\i\\-i q\\\frac{1}{q}-q\\-\frac{i}{q}\end{bmatrix} \qquad \text{ and } \qquad    
\frac{1}{q^3 + 2q + 2q^{-1} + q^{-3}}\begin{bmatrix}{}
 -\frac{(1+i) q^3 \left(q^4+i q^2-i\right)}{q^6+2 q^4+2 q^2+1} \\
 \frac{(1+i) q^2 \left(q^4+i q^2-i\right)}{q^6+2 q^4+2 q^2+1} \\
 \frac{(1+i) q^2 \left(q^4+i q^2-i\right)}{q^6+2 q^4+2 q^2+1} \\
 -\frac{(1+i) q^3 \left(q^2+i\right)}{q^6+2 q^4+2 q^2+1} \\
 -\frac{(1+i) q \left(q^4+i q^2-i\right)}{q^6+2 q^4+2 q^2+1} \\
 \frac{(1+i) q^2 \left(q^2+i\right)}{q^6+2 q^4+2 q^2+1} \\
 \frac{(1+i) q^4+(1-i)}{2 q} \\
 \frac{(1+i) q^4+(1-i)}{2 q} \\
 -\frac{(1+i) q^4+(1-i)}{2 q^2} \\
 -i \\
 -\frac{(1+i) q^4+(1-i)}{2 q^2} \\
 -i \\
 i q \\
 \frac{1}{q}-q \\
 \frac{i}{q} \\
\end{bmatrix}
  \]
  for the solutions $\gamma_+$ and $\gamma_-$ respectively.
\end{lem}
\begin{proof}
    We compute the determinant of the matrix of inner products of the first 15 elements of $S$ as
    \[\frac{262144 i q^{28} \left(q^2+1\right)^{13} \left(q^4+1\right)^{14} \left(q^4-q^2+1\right)^6}{(q-1)^{45} (q+1)^{45} \left(q^2-q+1\right)^{18} \left(q^2+q+1\right)^{18}}.\]
    This determinant is only zero for $N\leq 4$, and so it follows that the first 15 elements of $S$ are linearly independent for $N\geq 5$. As we know $\dim\End_{\mathcal{C}_A}(\ydiagram{1}^{\otimes 3}) = 15$ it follows that they are a basis.

    We compute the determinant of the matrix of all inner products of the elements of $S$ as
   \newline\resizebox{ \textwidth}{!} 
{  $    \frac{524288 q^{32} \left(q^2+1\right)^{12} \left(q^4+1\right)^{15} \left(q^4-q^2+1\right)^5 \left(q^{14}+4 q^{10}-6 q^8+4 q^6+q^2+\gamma ^2 \left(q^8+q^6-q^2-1\right)^2+4 \gamma  \left(q^{10}-q^6-q^4+1\right) q^3\right)}{\left(q^2-1\right)^{48} \left(q^4+q^2+1\right)^{22}} .   $}
As we know $\dim\End_{\mathcal{C}_A}(+^{\otimes 3}) = 15$ this matrix must be singular, and assuming $N\geq 5$ we must have that \[q^{14}+4 q^{10}-6 q^8+4 q^6+q^2+\gamma ^2 \left(q^8+q^6-q^2-1\right)^2+4 \gamma  \left(q^{10}-q^6-q^4+1\right) q^3=0.\]
This equation has the two solutions as in the statement of the lemma. For each of these two solution we then compute the null-space of the matrix of inner-products to obtain the relations involving the element
\[\att{rr}{.5}.\]
\end{proof}

We will now see that the relation we have obtained in the above lemma is incompatible with 
\[  \dim\Hom_{\mathcal{C}_A}(R  \otimes \ydiagram{1}\to  \ydiagram{2}\otimes \ydiagram{1}) = 1.  \]

\begin{thm}
    Let $A \in \mathcal{C}$ be a 1-super-transitive \`etale algebra object, let $N\geq 5$, and let $k = N-2$. The the fusion graph for $\ydiagram{1}\in \mathcal{C}_A$ up to depth 3 is not the graph (i).
\end{thm}
\begin{proof}
    Suppose that the fusion graph for $\ydiagram{1}\in \mathcal{C}_A$ up to depth 3 is graph (i), then we have one of the relations as in Lemma~\ref{lem:rrRel}. Further, we have that 
    \[  \dim\Hom_{\mathcal{C}_A}(R  \otimes \ydiagram{1}\to  \ydiagram{2}\otimes \ydiagram{1}) = 1.  \]
    This implies that there exists a scalar $\omega \in \mathbb{C}$ such that 
    \[  \att{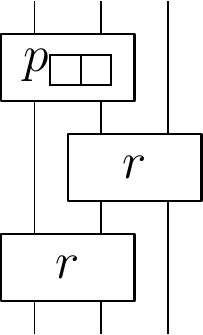}{.5} = \omega  \att{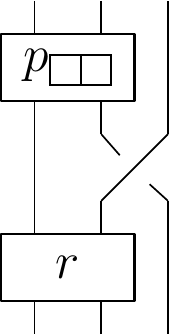}{.5}.   \]
    In order to simplify the computation of both sides of this equation we compute that the action for the top-composition via the braid $\att{mrep0}{.5}$ on $\operatorname{End}_{\mathcal{C}_A}(\ydiagram{1}^{\otimes 3})$ with respect to our basis is 
    \[\left(
\begin{array}{ccccccccccccccc}
 0 & 1 & 0 & 0 & 0 & 0 & 0 & 0 & 0 & 0 & 0 & 0 & 0 & 0 & 0 \\
 1 & q-\frac{1}{q} & 0 & 0 & 0 & 0 & 0 & 0 & 0 & 0 & 0 & 0 & 0 & 0 & 0 \\
 0 & 0 & 0 & 1 & 0 & 0 & 0 & 0 & 0 & 0 & 0 & 0 & 0 & 0 & 0 \\
 0 & 0 & 1 & q-\frac{1}{q} & 0 & 0 & 0 & 0 & 0 & 0 & 0 & 0 & 0 & 0 & 0 \\
 0 & 0 & 0 & 0 & 0 & 1 & 0 & 0 & 0 & 0 & 0 & 0 & 0 & 0 & 0 \\
 0 & 0 & 0 & 0 & 1 & q-\frac{1}{q} & 0 & 0 & 0 & 0 & 0 & 0 & 0 & 0 & 0 \\
 0 & 0 & 0 & 0 & 0 & 0 & -\frac{1}{q} & 0 & 1-q^2 & 0 & 0 & 0 & 0 & -q + q^{-1} & 0 \\
 0 & 0 & 0 & 0 & 0 & 0 & 0 & 0 & 0 & 1 & 0 & 0 &-q + q^{-1} & 0 & 0 \\
 0 & 0 & 0 & 0 & 0 & 0 & 0 & 0 & q-\frac{1}{q} & 0 & 0 & 0 & 0 & 1 & 0 \\
 0 & 0 & 0 & 0 & 0 & 0 & 0 & 1 & 0 & q-\frac{1}{q} & q-\frac{1}{q} & 0 & 0 & 0 & 0 \\
 0 & 0 & 0 & 0 & 0 & 0 & 0 & 0 & 0 & 0 & 0 & 0 & 1 & 0 & 0 \\
 0 & 0 & 0 & 0 & 0 & 0 & 0 & 0 & 0 & 0 & 0 & -\frac{1}{q} & 0 & 0 & 0 \\
 0 & 0 & 0 & 0 & 0 & 0 & 0 & 0 & 0 & 0 & 1 & 0 & q-\frac{1}{q} & 0 & 0 \\
 0 & 0 & 0 & 0 & 0 & 0 & 0 & 0 & 1 & 0 & 0 & 0 & 0 & 0 & 0 \\
 0 & 0 & 0 & 0 & 0 & 0 & 0 & 0 & 0 & 0 & 0 & 0 & 0 & 0 & -\frac{1}{q} \\
\end{array}
\right)\]
    With the relation from Lemma~\ref{lem:rrRel} corresponding to $\gamma = \gamma_+$ we obtain the equality
    \[ \begin{bmatrix}
 0 \\
 0 \\
 \frac{(1+i) q^5}{\left(q^2+1\right)^3 \left(q^4+q^2+1\right)^2} \\
 \frac{(1+i) q^6}{\left(q^2+1\right)^3 \left(q^4+q^2+1\right)^2} \\
 -\frac{(1+i) q^4}{\left(q^2+1\right)^3 \left(q^4+q^2+1\right)^2} \\
 -\frac{(1+i) q^5}{\left(q^2+1\right)^3 \left(q^4+q^2+1\right)^2} \\
 \frac{\left(\frac{1}{2}-\frac{i}{2}\right) q^2 \left(q^2-1\right) \left(q^4+(1+i) q^2-1\right)}{\left(q^2+1\right)^2 \left(q^4+q^2+1\right)} \\
 \frac{\left(\frac{1}{2}+\frac{i}{2}\right) \left(q^6+q^2\right)}{\left(q^2+1\right)^2 \left(q^4+q^2+1\right)} \\
 -\frac{\left(\frac{1}{2}-\frac{i}{2}\right) q^3 \left(q^4+(1+i) q^2-1\right)}{\left(q^2+1\right)^2 \left(q^4+q^2+1\right)} \\
 \frac{\left(\frac{1}{2}+\frac{i}{2}\right) \left(q^5+q\right)}{\left(q^2+1\right)^2 \left(q^4+q^2+1\right)} \\
 -\frac{\left(\frac{1}{2}+\frac{i}{2}\right) \left(q^5+q\right)}{\left(q^2+1\right)^2 \left(q^4+q^2+1\right)} \\
 0 \\
 -\frac{\left(\frac{1}{2}+\frac{i}{2}\right) \left(q^6+q^2\right)}{\left(q^2+1\right)^2 \left(q^4+q^2+1\right)} \\
 -\frac{\left(\frac{1}{2}-\frac{i}{2}\right) q^2 \left(q^4+(1+i) q^2-1\right)}{\left(q^2+1\right)^2 \left(q^4+q^2+1\right)} \\
 0 \\
    \end{bmatrix}   = \frac{\omega}{1 + q^2}  
\begin{bmatrix}
 0 \\
 0 \\
 0 \\
 0 \\
 0 \\
 0 \\
 -q \left(q^2-1\right)  \\
 0 \\
 q^2   \\
 0 \\
 0 \\
 0 \\
 0 \\
 q  \\
 0 \\
\end{bmatrix} .   \]
This gives $q = 0$, and hence we have our contradiction. A similar contradiction occurs for the relation corresponding to $\gamma = \gamma_-$.
\end{proof}

\subsection{Case ii)} In this case we have from the fusion graph, along with Lemma~\ref{lem:com} that $\dim\Hom_{\mathcal{C}_A}(  R \otimes \ydiagram{1} \to \ydiagram{1}\otimes R)= 1$. This along with the half-braid relation implies a relation of the form
\[ \att{rr}{.6} = \alpha \att{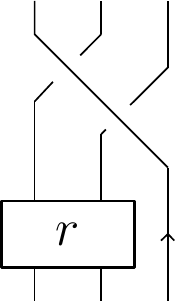}{.6}   \]
for some $\alpha \in \mathbb{C}$. Taking the right partial trace, we obtain
\[  \frac{d_R}{d_{\ydiagram{1}}}\att{rProj}{.6} = \alpha (- i q^{-1})(-q) \att{rProj}{.6} \]
and so $\alpha = -\frac{i d_R}{d_{\ydiagram{1}}}$. On the other hand, using the half-braid relation (twice) we compute
\[ \att{rProj}{.6} = \frac{1}{d_{\ydiagram{1}}} \att{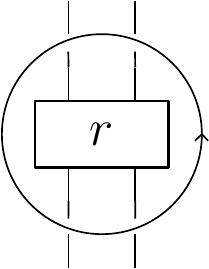}{.6} = \frac{1}{d_{\ydiagram{1}}\alpha } \att{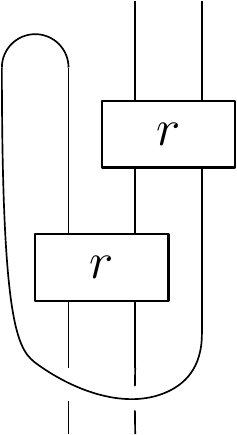}{.6}= \frac{1}{d_{\ydiagram{1}}\alpha } \att{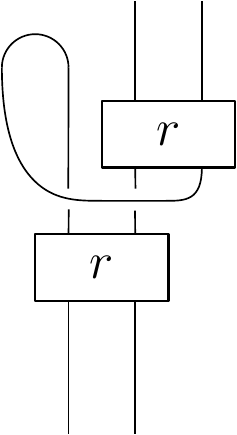}{.6}  = \frac{1}{d_{\ydiagram{1}}\alpha } (iq)(-q^{-1}) \att{rProj}{.6}.    \]
This gives that $\alpha = -\frac{i}{d_{\ydiagram{1}}}$, and so $d_R = 1$. In particular this shows that $R$ is an invertible object.

We thus have the following relations:
\[   \att{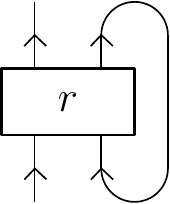}{.5} = \frac{q-q^{-1}}{i(q+q^{-1})}\att{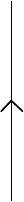}{.5}= \att{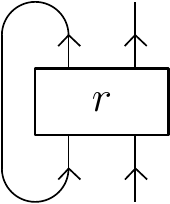}{.5}\qquad \text{and} \qquad   \att{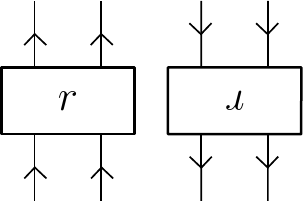}{.5} = \att{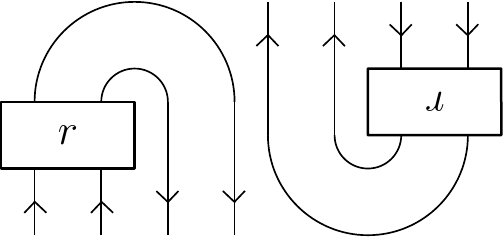}{.5}. \]

With these relations in $\mathcal{C}_A$ we can prove the following theorem.
\begin{thm}\label{thm:classN-2}
    Let $A\in \mathcal{C}$ be 1-super-transitive, and let $k= N-2$. Then there exists a faithful dominant monoidal functor
    \[   \overline{ \mathcal{SD}^-_N } \to \mathcal{C}_A                 \]
    In particular, $A$ is a simple current extension of $A_{\mathfrak{sl}_{N(N-1)/2}}$.
\end{thm}
\begin{proof}
    We define the functor $\mathcal{F}$ which sends $+\mapsto \mathcal{F}_A(\ydiagram{1})$, $-\mapsto \mathcal{F}_A(\ydiagram{1}^*)$, and
    \begin{align*}
        \att{b1}{.25}&\mapsto \att{b1}{.25}\\
        \att{kw}{.3}&\mapsto \att{kw}{.3}\\
        \att{splittingLiu}{.25}& \mapsto  \att{rProj}{.5}.
    \end{align*}
It follows from the above discussion that this functor preserves the defining relations of $\mathcal{SD}^-_N$, and hence is monoidal. As $\mathcal{C}_A$ is unitary, this functor descends to a faithful dominant monoidal functor as in the statement of the theorem.
Showing that $A$ is a simple current extension of $A_{\mathfrak{sl}_{N(N-1)/2}}$ is now essentially a verbatim copy of the final paragraph of the proof of Theorem~\ref{thm:classN}.
\end{proof}

As described at the start of this section, the case of $k= N+2$ is almost identical. Attached to the arXiv submission of this paper is the $16\times16$ matrix of inner products for the set $S$ in this case, and a derivation of the analogous contradiction for case i). For case ii) the argument is essentially verbatim. This gives the following theorem.

\begin{thm}\label{thm:classN+2}
    Let $A\in \mathcal{C}$ be 1-super-transitive, and let $k= N+2$. Then there exists a faithful dominant monoidal functor
    \[   \overline{ \mathcal{SD}^+_N } \to \mathcal{C}_A                 \]
    In particular, $A$ is a simple current extension of $A_{\mathfrak{sl}_{N(N+1)/2}}$.
\end{thm}

\bibliography{rigid}

\begin{thebibliography}{KMFPX15}

\bibitem[AP95]{MR1328736}
Henning~Haahr Andersen and Jan Paradowski.
\newblock Fusion categories arising from semisimple {L}ie algebras.
\newblock {\em Comm. Math. Phys.}, 169(3):563--588, 1995.

\bibitem[BK01]{SisiBakBak}
Bojko Bakalov and Alexander Kirillov, Jr.
\newblock {\em Lectures on tensor categories and modular functors}, volume~21 of {\em University Lecture Series}.
\newblock American Mathematical Society, Providence, RI, 2001.

\bibitem[CEM23]{dan}
Daniel Copeland and Cain Edie-Michell.
\newblock Cell systems for {$\overline{\operatorname{Rep}(U_q(\mathfrak{sl}_N))}$} module categories, 2023.

\bibitem[CGGH23]{MR4616673}
Sebastiano Carpi, Tiziano Gaudio, Luca Giorgetti, and Robin Hillier.
\newblock Haploid algebras in {$C^*$}-tensor categories and the {S}chellekens list.
\newblock {\em Comm. Math. Phys.}, 402(1):169--212, 2023.

\bibitem[DLM96]{simplesyrup}
Chongying Dong, Haisheng Li, and Geoffrey Mason.
\newblock Simple currents and extensions of vertex operator algebras.
\newblock {\em Comm. Math. Phys.}, 180(3):671--707, 1996.

\bibitem[DMNO13]{LagrangeUkraine}
Alexei Davydov, Michael M\"{u}ger, Dmitri Nikshych, and Victor Ostrik.
\newblock The {W}itt group of non-degenerate braided fusion categories.
\newblock {\em J. Reine Angew. Math.}, 677:135--177, 2013.

\bibitem[EGNO15]{book}
Pavel Etingof, Shlomo Gelaki, Dmitri Nikshych, and Victor Ostrik.
\newblock {\em Tensor categories}, volume 205 of {\em Mathematical Surveys and Monographs}.
\newblock American Mathematical Society, Providence, RI, 2015.

\bibitem[EMG25]{ModulesPt2}
Cain Edie-Michell and Terry Gannon.
\newblock Type ${II}$ quantum subgroups of $\mathfrak{sl}_{N}$ ${II}$: Classification, 2025.

\bibitem[EMS25]{me}
Cain Edie-Michell and Noah Snyder.
\newblock Interpolation categories for conformal embeddings, 2025.

\bibitem[EO22]{simp}
Pavel Etingof and Victor Ostrik.
\newblock On semisimplification of tensor categories.
\newblock In {\em Representation theory and algebraic geometry---a conference celebrating the birthdays of {S}asha {B}eilinson and {V}ictor {G}inzburg}, Trends Math., pages 3--35. Birkh\"{a}user/Springer, Cham, [2022] \copyright 2022.

\bibitem[Fin96]{MR1384612}
Micheal Finkelberg.
\newblock An equivalence of fusion categories.
\newblock {\em Geom. Funct. Anal.}, 6(2):249--267, 1996.

\bibitem[FZ92]{WZW}
Igor~B. Frenkel and Yongchang Zhu.
\newblock Vertex operator algebras associated to representations of affine and {V}irasoro algebras.
\newblock {\em Duke Math. J.}, 66(1):123--168, 1992.

\bibitem[Gan23]{LevelTerry}
Terry Gannon.
\newblock Exotic quantum subgroups and extensions of affine lie algebra {VOA}s – part {I}, 2023.
\newblock \arxiv{2301.07287}.

\bibitem[GJ07]{sup}
Pinhas Grossman and Vaughan F.~R. Jones.
\newblock Intermediate subfactors with no extra structure.
\newblock {\em J. Amer. Math. Soc.}, 20(1):219--265, 2007.

\bibitem[HKL15]{HKL}
Yi-Zhi Huang, Alexander Kirillov, Jr., and James Lepowsky.
\newblock Braided tensor categories and extensions of vertex operator algebras.
\newblock {\em Comm. Math. Phys.}, 337(3):1143--1159, 2015.

\bibitem[Jon12]{sup2}
Vaughan F.~R. Jones.
\newblock Quadratic tangles in planar algebras.
\newblock {\em Duke Math. J.}, 161(12):2257--2295, 2012.

\bibitem[KL93]{KL1}
D.~Kazhdan and G.~Lusztig.
\newblock Tensor structures arising from affine {L}ie algebras. {I}, {II}.
\newblock {\em J. Amer. Math. Soc.}, 6(4):905--947, 949--1011, 1993.

\bibitem[KMFPX15]{papi}
Victor~G. Kac, Pierluigi M\"{o}seneder~Frajria, Paolo Papi, and Feng Xu.
\newblock Conformal embeddings and simple current extensions.
\newblock {\em Int. Math. Res. Not. IMRN}, 2015(14):5229--5288, 2015.

\bibitem[KO02]{kril}
Alexander Kirillov, Jr. and Viktor Ostrik.
\newblock On a {$q$}-analogue of the {M}c{K}ay correspondence and the {ADE} classification of {$\mathfrak {sl}_2$} conformal field theories.
\newblock {\em Adv. Math.}, 171(2):183--227, 2002.

\bibitem[KW88]{KacMan}
Victor~G. Kac and Minoru Wakimoto.
\newblock Modular and conformal invariance constraints in representation theory of affine algebras.
\newblock {\em Adv. in Math.}, 70(2):156--236, 1988.

\bibitem[KW93]{sovietHans}
David Kazhdan and Hans Wenzl.
\newblock Reconstructing monoidal categories.
\newblock In {\em I. {M}. {G}elfand {S}eminar}, volume~16 of {\em Adv. Soviet Math.}, pages 111--136. Amer. Math. Soc., Providence, RI, 1993.

\bibitem[Liu15]{LiuYB}
Zhengwei Liu.
\newblock Yang-{B}axter relation planar algebras, 2015.
\newblock \arxiv{1507.06030}.

\bibitem[LL95]{LevLaughLib}
F.~Levstein and J.~I. Liberati.
\newblock Branching rules for conformal embeddings.
\newblock {\em Comm. Math. Phys.}, 173(1):1--16, 1995.

\bibitem[Nik25]{iwenttodmitrisofficeandaskedhim}
Dmitri Nikshych.
\newblock Private communication, 2025.

\bibitem[OS14]{lrdual}
Victor Ostrik and Michael Sun.
\newblock Level-rank duality via tensor categories.
\newblock {\em Comm. Math. Phys.}, 326(1):49--61, 2014.

\bibitem[Pag13]{Paggo}
Pablo~Gonzalez Pagotto.
\newblock Graphical calculus for the hecke algebra, 2013.
\newblock \href{https://alistairsavage.ca/pubs/Pagotto-Hecke-graphical-calculus.pdf}{https://alistairsavage.ca/pubs/Pagotto-Hecke-graphical-calculus.pdf}.

\bibitem[Saw06]{sawin}
Stephen~F. Sawin.
\newblock Quantum groups at roots of unity and modularity.
\newblock {\em J. Knot Theory Ramifications}, 15(10):1245--1277, 2006.

\bibitem[Sch18]{LevelAndy}
Andrew Schopieray.
\newblock Level bounds for exceptional quantum subgroups in rank two.
\newblock {\em Internat. J. Math.}, 29(5):1850034, 33, 2018.

\bibitem[Xu98]{Xu}
Feng Xu.
\newblock New braided endomorphisms from conformal inclusions.
\newblock {\em Comm. Math. Phys.}, 192(2):349--403, 1998.

\bibitem[Xu07]{Mirror}
Feng Xu.
\newblock Mirror extensions of local nets.
\newblock {\em Comm. Math. Phys.}, 270(3):835--847, 2007.

\end{thebibliography}
\bibliographystyle{alpha}
\appendix

\section{Matrix Representations}\label{app:A}

Let $k = N$, and $A\in \mathcal{C}$ a 1-super-transitive \`etale algebra object. In this appendix we list the matrices for the four elements 
\[   \att{mrep0}{.5},\qquad  \att{mrep-1}{.5} , \qquad \att{mrep1}{.5} \quad \text{ and } \att{mrep2}{.5}\]
acting via pre-composition on $\operatorname{End}_{\mathcal{C}_A}(\ydiagram{1}^{\otimes 3})$ with respect to the basis $S$ from Definition~\ref{def:knBas}. For the two braid elements we have that the matrices are respectively (where we write $z = q - q^{-1}$):
\[\left[
\begin{array}{cccccccccccccccccccccccc}
 0 & 1 & 0 & 0 & 0 & 0 & 0 & 0 & 0 & 0 & 0 & 0 & 0 & 0 & 0 & 0 & 0 & 0 & 0 & 0 & 0 & 0 & 0 & 0 \\
 1 & z& 0 & 0 & 0 & 0 & 0 & 0 & 0 & 0 & 0 & 0 & 0 & 0 & 0 & 0 & 0 & 0 & 0 & 0 & 0 & 0 & 0 & 0 \\
 0 & 0 & 0 & 0 & 1 & 0 & 0 & 0 & 0 & 0 & 0 & 0 & 0 & 0 & 0 & 0 & 0 & 0 & 0 & 0 & 0 & 0 & 0 & 0 \\
 0 & 0 & 0 & 0 & 0 & 1 & 0 & 0 & 0 & 0 & 0 & 0 & 0 & 0 & 0 & 0 & 0 & 0 & 0 & 0 & 0 & 0 & 0 & 0 \\
 0 & 0 & 1 & 0 & z & 0 & 0 & 0 & 0 & 0 & 0 & 0 & 0 & 0 & 0 & 0 & 0 & 0 & 0 & 0 & 0 & 0 & 0 & 0 \\
 0 & 0 & 0 & 1 & 0 & z& 0 & 0 & 0 & 0 & 0 & 0 & 0 & 0 & 0 & 0 & 0 & 0 & 0 & 0 & 0 & 0 & 0 & 0 \\
 0 & 0 & 0 & 0 & 0 & 0 & q & 0 & 0 & 0 & 0 & 0 & 0 & 0 & 0 & 0 & 0 & 0 & 0 & 0 & 0 & 0 & 0 & 0 \\
 0 & 0 & 0 & 0 & 0 & 0 & 0 & -\frac{1}{q} & 0 & 0 & 0 & 0 & 0 & 0 & 0 & 0 & 0 & 0 & 0 & 0 & 0 & 0 & 0 & 0 \\
 0 & 0 & 0 & 0 & 0 & 0 & 0 & 0 & 0 & 0 & 0 & 0 & 0 & 0 & 1 & 0 & 0 & 0 & 0 & 0 & 0 & 0 & 0 & 0 \\
 0 & 0 & 0 & 0 & 0 & 0 & 0 & 0 & 0 & 0 & 0 & 0 & 0 & 0 & 0 & 1 & 0 & 0 & 0 & 0 & 0 & 0 & 0 & 0 \\
 0 & 0 & 0 & 0 & 0 & 0 & 0 & 0 & 0 & 0 & q & 0 & 0 & 0 & 0 & 0 & 0 & 0 & 0 & 0 & 0 & 0 & 0 & 0 \\
 0 & 0 & 0 & 0 & 0 & 0 & 0 & 0 & 0 & 0 & 0 & -\frac{1}{q} & 0 & 0 & 0 & 0 & 0 & 0 & 0 & 0 & 0 & 0 & 0 & 0 \\
 0 & 0 & 0 & 0 & 0 & 0 & 0 & 0 & 0 & 0 & 0 & 0 & z & 0 & 0 & 0 & 0 & 0 & 1 & 0 & 0 & 0 & 0 & 0 \\
 0 & 0 & 0 & 0 & 0 & 0 & 0 & 0 & 0 & 0 & 0 & 0 & 0 & z & 0 & 0 & 0 & 0 & 0 & 1 & 0 & 0 & 0 & 0 \\
 0 & 0 & 0 & 0 & 0 & 0 & 0 & 0 & 1 & 0 & 0 & 0 & 0 & 0 & z & 0 & 0 & 0 & 0 & 0 & 0 & 0 & 0 & 0 \\
 0 & 0 & 0 & 0 & 0 & 0 & 0 & 0 & 0 & 1 & 0 & 0 & 0 & 0 & 0 & z & 0 & 0 & 0 & 0 & 0 & 0 & 0 & 0 \\
 0 & 0 & 0 & 0 & 0 & 0 & 0 & 0 & 0 & 0 & 0 & 0 & 0 & 0 & 0 & 0 & 0 & 0 & 0 & 0 & 0 & 0 & 1 & 0 \\
 0 & 0 & 0 & 0 & 0 & 0 & 0 & 0 & 0 & 0 & 0 & 0 & 0 & 0 & 0 & 0 & 0 & 0 & 0 & 0 & 0 & 0 & 0 & 1 \\
 0 & 0 & 0 & 0 & 0 & 0 & 0 & 0 & 0 & 0 & 0 & 0 & 1 & 0 & 0 & 0 & 0 & 0 & 0 & 0 & 0 & 0 & 0 & 0 \\
 0 & 0 & 0 & 0 & 0 & 0 & 0 & 0 & 0 & 0 & 0 & 0 & 0 & 1 & 0 & 0 & 0 & 0 & 0 & 0 & 0 & 0 & 0 & 0 \\
 0 & 0 & 0 & 0 & 0 & 0 & 0 & 0 & 0 & 0 & 0 & 0 & 0 & 0 & 0 & 0 & 0 & 0 & 0 & 0 & q & 0 & 0 & 0 \\
 0 & 0 & 0 & 0 & 0 & 0 & 0 & 0 & 0 & 0 & 0 & 0 & 0 & 0 & 0 & 0 & 0 & 0 & 0 & 0 & 0 & -\frac{1}{q} & 0 & 0 \\
 0 & 0 & 0 & 0 & 0 & 0 & 0 & 0 & 0 & 0 & 0 & 0 & 0 & 0 & 0 & 0 & 1 & 0 & 0 & 0 & 0 & 0 & z & 0 \\
 0 & 0 & 0 & 0 & 0 & 0 & 0 & 0 & 0 & 0 & 0 & 0 & 0 & 0 & 0 & 0 & 0 & 1 & 0 & 0 & 0 & 0 & 0 & z \\
\end{array}
\right]\]

\[ \left[
\begin{array}{cccccccccccccccccccccccc}
 0 & 0 & 1 & 0 & 0 & 0 & 0 & 0 & 0 & 0 & 0 & 0 & 0 & 0 & 0 & 0 & 0 & 0 & 0 & 0 & 0 & 0 & 0 & 0 \\
 0 & 0 & 0 & 1 & 0 & 0 & 0 & 0 & 0 & 0 & 0 & 0 & 0 & 0 & 0 & 0 & 0 & 0 & 0 & 0 & 0 & 0 & 0 & 0 \\
 1 & 0 & z & 0 & 0 & 0 & 0 & 0 & 0 & 0 & 0 & 0 & 0 & 0 & 0 & 0 & 0 & 0 & 0 & 0 & 0 & 0 & 0 & 0 \\
 0 & 1 & 0 & z & 0 & 0 & 0 & 0 & 0 & 0 & 0 & 0 & 0 & 0 & 0 & 0 & 0 & 0 & 0 & 0 & 0 & 0 & 0 & 0 \\
 0 & 0 & 0 & 0 & 0 & 1 & 0 & 0 & 0 & 0 & 0 & 0 & 0 & 0 & 0 & 0 & 0 & 0 & 0 & 0 & 0 & 0 & 0 & 0 \\
 0 & 0 & 0 & 0 & 1 & z & 0 & 0 & 0 & 0 & 0 & 0 & 0 & 0 & 0 & 0 & 0 & 0 & 0 & 0 & 0 & 0 & 0 & 0 \\
 0 & 0 & 0 & 0 & 0 & 0 & 0 & 0 & 0 & 0 & 0 & 0 & 0 & 0 & 0 & 0 & 1 & 0 & -z & 0 & 0 & 0 & 0 & 0 \\
 0 & 0 & 0 & 0 & 0 & 0 & 0 & 0 & 0 & 0 & 0 & 0 & 0 & 0 & 0 & 0 & 0 & 1 & 0 &-z & 0 & 0 & 0 & 0 \\
 0 & 0 & 0 & 0 & 0 & 0 & 0 & 0 & q & 0 & 0 & 0 & 0 & 0 & \frac{z}{q} & 0 & 0 & 0 & 0 & 0 & -z & 0 & 0 & 0 \\
 0 & 0 & 0 & 0 & 0 & 0 & 0 & 0 & 0 & -\frac{1}{q} & 0 & 0 & 0 & 0 & 0 & -q z & 0 & 0 & 0 & 0 & 0 & -z & 0 & 0 \\
 0 & 0 & 0 & 0 & 0 & 0 & 0 & 0 & 0 & 0 & 0 & 0 & 0 & 0 & 0 & 0 & 0 & 0 & 1 & 0 & 0 & 0 & 0 & 0 \\
 0 & 0 & 0 & 0 & 0 & 0 & 0 & 0 & 0 & 0 & 0 & 0 & 0 & 0 & 0 & 0 & 0 & 0 & 0 & 1 & 0 & 0 & 0 & 0 \\
 0 & 0 & 0 & 0 & 0 & 0 & 0 & 0 & 0 & 0 & 0 & 0 & q & 0 & 0 & 0 & 0 & 0 & 0 & 0 & 0 & 0 & 0 & 0 \\
 0 & 0 & 0 & 0 & 0 & 0 & 0 & 0 & 0 & 0 & 0 & 0 & 0 & -\frac{1}{q} & 0 & 0 & 0 & 0 & 0 & 0 & 0 & 0 & 0 & 0 \\
 0 & 0 & 0 & 0 & 0 & 0 & 0 & 0 & 0 & 0 & 0 & 0 & 0 & 0 & z & 0 & 0 & 0 & 0 & 0 & 1 & 0 & 0 & 0 \\
 0 & 0 & 0 & 0 & 0 & 0 & 0 & 0 & 0 & 0 & 0 & 0 & 0 & 0 & 0 & z & 0 & 0 & 0 & 0 & 0 & 1 & 0 & 0 \\
 0 & 0 & 0 & 0 & 0 & 0 & 1 & 0 & 0 & 0 & z & 0 & 0 & 0 & 0 & 0 & z & 0 & 0 & 0 & 0 & 0 & 0 & 0 \\
 0 & 0 & 0 & 0 & 0 & 0 & 0 & 1 & 0 & 0 & 0 & z & 0 & 0 & 0 & 0 & 0 &z & 0 & 0 & 0 & 0 & 0 & 0 \\
 0 & 0 & 0 & 0 & 0 & 0 & 0 & 0 & 0 & 0 & 1 & 0 & 0 & 0 & 0 & 0 & 0 & 0 & z & 0 & 0 & 0 & 0 & 0 \\
 0 & 0 & 0 & 0 & 0 & 0 & 0 & 0 & 0 & 0 & 0 & 1 & 0 & 0 & 0 & 0 & 0 & 0 & 0 & z & 0 & 0 & 0 & 0 \\
 0 & 0 & 0 & 0 & 0 & 0 & 0 & 0 & 0 & 0 & 0 & 0 & 0 & 0 & 1 & 0 & 0 & 0 & 0 & 0 & 0 & 0 & 0 & 0 \\
 0 & 0 & 0 & 0 & 0 & 0 & 0 & 0 & 0 & 0 & 0 & 0 & 0 & 0 & 0 & 1 & 0 & 0 & 0 & 0 & 0 & 0 & 0 & 0 \\
 0 & 0 & 0 & 0 & 0 & 0 & 0 & 0 & 0 & 0 & 0 & 0 & 0 & 0 & 0 & 0 & 0 & 0 & 0 & 0 & 0 & 0 & q & 0 \\
 0 & 0 & 0 & 0 & 0 & 0 & 0 & 0 & 0 & 0 & 0 & 0 & 0 & 0 & 0 & 0 & 0 & 0 & 0 & 0 & 0 & 0 & 0 & -\frac{1}{q} \\
\end{array}
\right]\]

The matrix for $e_{1,2}$ we give in the block form
\[   \begin{bmatrix}
    A & B & C 
\end{bmatrix}  \]
where
\tiny

\[A=\left[
\begin{array}{ccccccccccccc}
 0 & 0 & 0 & 0 & 0 & 0 & 0 & \frac{1}{q^2+1} & 0 & \frac{2}{\left(q^2+1\right)^2}-\frac{1}{q^4+1} & 0 & 0 & 0 \\
 0 & 0 & 0 & 0 & 0 & 0 & 0 & \frac{q}{q^2+1} & 0 &  \frac{2q}{\left(q^2+1\right)^2}-\frac{q}{q^4+1} & 0 & 0 & 0 \\
 0 & 0 & 0 & 0 & 0 & 0 & 0 & 0 & 0 & \frac{2q}{\left(q^2+1\right)^2}-\frac{q}{q^4+1} & 0 & \frac{1}{q^2+1} & 0 \\
 0 & 0 & 0 & 0 & 0 & 0 & 0 & 0 & 0 & \frac{1-q^2}{q^6+q^4+q^2+1} & 0 & 0 & 0 \\
 0 & 0 & 0 & 0 & 0 & 0 & 0 & 0 & 0 & \frac{q^2 \left(q^2-1\right)^2}{\left(q^2+1\right)^2 \left(q^4+1\right)} & 0 & \frac{q}{q^2+1} & 0 \\
 0 & 0 & 0 & 0 & 0 & 0 & 0 & 0 & 0 & \frac{q-q^3}{q^6+q^4+q^2+1} & 0 & 0 & 0 \\
 1 & -\frac{1}{q} & 0 & 1-q^2 & 0 & q-\frac{1}{q} & 0 & 0 & -\frac{\gamma  q \left(q^2-1\right)^2}{2 \left(q^4-q^2+1\right)} & \frac{\left(q^2-1\right)^2 \overline{\gamma}}{2 \left(q^5-q^3+q\right)} & 0 & 0 & -\frac{\gamma  \left(q^2-2\right) \left(q^2-1\right)^2}{2 \left(q^4-q^2+1\right)} \\
 0 & 0 & 0 & 0 & 0 & 0 & 0 & 0 & 0 & 0 & 0 & 0 & 0 \\
 0 & 0 & 0 & 0 & 0 & 0 & 0 & 0 & \frac{\gamma  \left(q^2-1\right)^2}{2 \left(q^5-q^3+q\right)} & 0 & 0 & 0 & \frac{\gamma -\gamma  q^2}{2 q^4-2 q^2+2} \\
 0 & 0 & 0 & 0 & 0 & 0 & 0 & 0 & 0 & \frac{\gamma  \left(q^2-1\right)^2}{2 \left(q^5-q^3+q\right)} & 0 & 0 & \frac{1}{2} \lambda  \left(\frac{1}{q^2}-1\right) \\
 0 & 0 & 1 & q-\frac{1}{q} & -\frac{1}{q} & \frac{1}{q^2}-1 & 0 & 0 & \frac{\gamma  \left(q^2-1\right)^2}{2 \left(q^4-q^2+1\right)} & \frac{\left(q^2-1\right)^2 \overline{\gamma}}{2 \left(q^4-q^2+1\right)} & 0 & 0 & \frac{\gamma  \left(q^2-1\right)^3}{2 \left(q^5-q^3+q\right)} \\
 0 & 0 & 0 & 0 & 0 & 0 & 0 & 0 & 0 & 0 & 0 & 0 & 0 \\
 0 & 0 & 0 & 0 & 0 & 0 & 0 & 0 & -\frac{\gamma  q^2 \left(q^2-1\right)}{2 \left(q^4-q^2+1\right)} & -\frac{\left(q^2-1\right) \overline{\gamma}}{2 \left(q^4-q^2+1\right)} & 0 & 0 & 0 \\
 0 & 0 & 0 & 0 & 0 & 0 & 0 & 0 & -\frac{1}{2} \lambda  (q-1) (q+1) & \frac{\gamma -\gamma  q^2}{2 q^4-2 q^2+2} & 0 & 0 & -\frac{\lambda  \left(q^2-1\right)^2}{2 q} \\
 0 & 0 & 0 & 0 & 0 & 0 & 0 & 0 & \frac{\gamma  \left(q^2-1\right)^2}{2 \left(q^4-q^2+1\right)} & 0 & 0 & 0 & -\frac{\gamma  q \left(q^2-1\right)}{2 \left(q^4-q^2+1\right)} \\
 0 & 0 & 0 & 0 & 0 & 0 & 0 & 0 & 0 & \frac{\gamma  \left(q^2-1\right)^2}{2 \left(q^4-q^2+1\right)} & 0 & 0 & -\frac{\lambda  \left(q^2-1\right)}{2 q} \\
 0 & 0 & 0 & 0 & 0 & 0 & 0 & 0 & \frac{\gamma  \left(q^2-1\right)}{2 \left(q^4-q^2+1\right)} & -\frac{\left(q^2-1\right) \overline{\gamma}}{2 \left(q^4-q^2+1\right)} & 0 & 0 & -\frac{\gamma  \left(q^2-1\right)}{2 \left(q^5-q^3+q\right)} \\
 0 & 0 & 0 & 0 & 0 & 0 & 0 & 0 & \frac{\lambda  \left(q^2-1\right)}{2 q^2} & \frac{\gamma -\gamma  q^2}{2 q^4-2 q^2+2} & 0 & 0 & \frac{\lambda  \left(q^2-1\right)}{2 q} \\
 0 & 0 & 0 & 0 & 0 & 0 & 0 & 0 & -\frac{\gamma  q \left(q^2-1\right)}{2 \left(q^4-q^2+1\right)} & -\frac{\left(q^2-1\right) \overline{\gamma}}{2 \left(q^5-q^3+q\right)} & 0 & 0 & 0 \\
 0 & 0 & 0 & 0 & 0 & 0 & 0 & 0 & -\frac{\lambda  \left(q^2-1\right)}{2 q} & -\frac{\gamma  \left(q^2-1\right)}{2 \left(q^5-q^3+q\right)} & 0 & 0 & -\frac{\lambda  \left(q^2-1\right)^2}{2 q^2} \\
 0 & 0 & 0 & 1 & 0 & -\frac{1}{q} & 0 & 0 & 0 & 0 & 0 & 0 & \frac{\gamma  \left(q^2-1\right)^2}{2 \left(q^4-q^2+1\right)} \\
 0 & 0 & 0 & 0 & 0 & 0 & 0 & 0 & 0 & 0 & 0 & 0 & 0 \\
 0 & 0 & 0 & 0 & 0 & 0 & 0 & 0 & \frac{\gamma  q \left(q^2-1\right)}{2 \left(q^4-q^2+1\right)} & -\frac{q \left(q^2-1\right) \overline{\gamma}}{2 \left(q^4-q^2+1\right)} & 0 & 0 & \frac{\gamma -\gamma  q^2}{2 q^4-2 q^2+2} \\
 0 & 0 & 0 & 0 & 0 & 0 & 0 & 0 & \frac{\lambda  \left(q^2-1\right)}{2 q} & -\frac{\gamma  q \left(q^2-1\right)}{2 \left(q^4-q^2+1\right)} & 0 & 0 & \frac{1}{2} \lambda  \left(q^2-1\right) \\
\end{array}
\right]\]
\[B:= \left[
\begin{array}{ccccc}
 q \left(\frac{2}{\left(q^2+1\right)^2}-\frac{1}{q^4+1}\right) & 0 & -\frac{\left(q^2-1\right)^2}{\left(q^2+1\right)^2 \left(q^5+q\right)} & 0 & \frac{q \left(q^2-1\right)}{q^6+q^4+q^2+1} \\
 \frac{q^2 \left(q^2-1\right)^2}{\left(q^2+1\right)^2 \left(q^4+1\right)} & 0 & \frac{1}{q^4+1}-\frac{2}{\left(q^2+1\right)^2} & 0 & \frac{1}{q^2+1}-\frac{1}{q^4+1} \\
 \frac{1-q^2}{q^6+q^4+q^2+1} & 0 & \frac{1}{q^4+1}-\frac{2}{\left(q^2+1\right)^2} & 0 & \frac{2}{\left(q^2+1\right)^2}-\frac{1}{q^4+1} \\
 -\frac{\left(q^2-1\right)^2}{\left(q^2+1\right)^2 \left(q^5+q\right)} & 0 & \frac{q^2-1}{q^7+q^5+q^3+q} & 0 & q \left(\frac{2}{\left(q^2+1\right)^2}-\frac{1}{q^4+1}\right) \\
 \frac{q-q^3}{q^6+q^4+q^2+1} & 0 & q \left(\frac{1}{q^4+1}-\frac{2}{\left(q^2+1\right)^2}\right) & 0 & q \left(\frac{2}{\left(q^2+1\right)^2}-\frac{1}{q^4+1}\right) \\
 \frac{1}{q^4+1}-\frac{2}{\left(q^2+1\right)^2} & 0 & \frac{q^2-1}{q^6+q^4+q^2+1} & 0 & \frac{q^2 \left(q^2-1\right)^2}{\left(q^2+1\right)^2 \left(q^4+1\right)} \\
 -\frac{\left(q^2-1\right)^2 \overline{\gamma}}{2 q^2} & \frac{\gamma  \left(q^2-1\right)^2}{2 \left(q^4-q^2+1\right)} & -\frac{\left(q^2-1\right)^2 \overline{\gamma}}{2 \left(q^6-q^4+q^2\right)} & -\frac{\gamma  \left(q^2-1\right)^3}{2 \left(q^4-q^2+1\right)} & -\frac{\left(q^2-1\right)^3 \overline{\gamma}}{2 \left(q^4-q^2+1\right)} \\
 0 & 0 & 0 & 0 & 0 \\
 -\frac{\left(q^2-1\right) \overline{\gamma}}{2 \left(q^6-q^4+q^2\right)} & -\frac{\gamma  \left(q^2-1\right)^2}{2 \left(q^6-q^4+q^2\right)} & 0 & \frac{\gamma -\gamma  q^2}{2 q^4-2 q^2+2} & -\frac{\left(q^2-1\right) \overline{\gamma}}{2 \left(q^6-q^4+q^2\right)} \\
 \frac{\gamma -\gamma  q^2}{2 \left(q^6-q^4+q^2\right)} & 0 & -\frac{\gamma  \left(q^2-1\right)^2}{2 \left(q^6-q^4+q^2\right)} & \frac{1}{2} \lambda  \left(q^2-1\right) & \frac{\gamma  \left(q^2-1\right)}{2 \left(q^4-q^2+1\right)} \\
 \frac{\left(q^2-1\right)^3 \overline{\gamma}}{2 \left(q^5-q^3+q\right)} & -\frac{\gamma  \left(q^2-1\right)^2}{2 \left(q^5-q^3+q\right)} & -\frac{\left(q^2-1\right)^2 \overline{\gamma}}{2 \left(q^5-q^3+q\right)} & -\frac{\gamma  \left(q^2-1\right)^2}{2 \left(q^5-q^3+q\right)} & \frac{q \left(q^2-1\right)^2 \overline{\gamma}}{2 \left(q^4-q^2+1\right)} \\
 0 & 0 & 0 & 0 & 0 \\
 -\frac{\left(q^2-1\right)^2 \overline{\gamma}}{2 \left(q^5-q^3+q\right)} & \frac{\gamma  q \left(q^2-1\right)}{2 \left(q^4-q^2+1\right)} & \frac{\left(q^2-1\right) \overline{\gamma}}{2 \left(q^5-q^3+q\right)} & \frac{\gamma  q \left(q^2-1\right)}{2 \left(q^4-q^2+1\right)} & -\frac{q \left(q^2-1\right) \overline{\gamma}}{2 \left(q^4-q^2+1\right)} \\
 \frac{\gamma  \left(q^2-1\right)^3}{2 \left(q^5-q^3+q\right)} & \frac{\lambda  \left(q^2-1\right)}{2 q} & \frac{\gamma  \left(q^2-1\right)}{2 \left(q^5-q^3+q\right)} & -\frac{1}{2} \lambda  (q-1) q (q+1) & \frac{\gamma  q^3 \left(q^2-1\right)}{2 \left(q^4-q^2+1\right)} \\
 -\frac{\left(q^2-1\right) \overline{\gamma}}{2 \left(q^5-q^3+q\right)} & -\frac{\gamma  \left(q^2-1\right)^2}{2 \left(q^5-q^3+q\right)} & 0 & -\frac{\gamma  q \left(q^2-1\right)}{2 \left(q^4-q^2+1\right)} & -\frac{\left(q^2-1\right) \overline{\gamma}}{2 \left(q^5-q^3+q\right)} \\
 -\frac{\gamma  \left(q^2-1\right)}{2 \left(q^5-q^3+q\right)} & 0 & -\frac{\gamma  \left(q^2-1\right)^2}{2 \left(q^5-q^3+q\right)} & \frac{1}{2} \lambda  q \left(q^2-1\right) & \frac{\gamma  q \left(q^2-1\right)}{2 \left(q^4-q^2+1\right)} \\
 \frac{\left(q^2-1\right) \overline{\gamma}}{2 \left(q^5-q^3+q\right)} & -\frac{\gamma  \left(q^2-1\right)}{2 \left(q^5-q^3+q\right)} & \frac{\left(q^2-1\right) \overline{\gamma}}{2 \left(q^5-q^3+q\right)} & \frac{\gamma  \left(q^2-1\right)^2}{2 \left(q^5-q^3+q\right)} & 0 \\
 -\frac{\gamma  q \left(q^2-1\right)}{2 \left(q^4-q^2+1\right)} & -\frac{\lambda  \left(q^2-1\right)}{2 q^3} & \frac{\gamma  \left(q^2-1\right)}{2 \left(q^5-q^3+q\right)} & 0 & \frac{\gamma  \left(q^2-1\right)^2}{2 \left(q^5-q^3+q\right)} \\
 -\frac{\left(q^2-1\right)^2 \overline{\gamma}}{2 \left(q^6-q^4+q^2\right)} & \frac{\gamma  \left(q^2-1\right)}{2 \left(q^4-q^2+1\right)} & \frac{\left(q^2-1\right) \overline{\gamma}}{2 \left(q^6-q^4+q^2\right)} & \frac{\gamma  \left(q^2-1\right)}{2 \left(q^4-q^2+1\right)} & -\frac{\left(q^2-1\right) \overline{\gamma}}{2 \left(q^4-q^2+1\right)} \\
 \frac{\gamma  \left(q^2-1\right)^3}{2 \left(q^6-q^4+q^2\right)} & \frac{\lambda  \left(q^2-1\right)}{2 q^2} & \frac{\gamma  \left(q^2-1\right)}{2 \left(q^6-q^4+q^2\right)} & -\frac{1}{2} \lambda  \left(q^2-1\right) & \frac{\gamma  q^2 \left(q^2-1\right)}{2 \left(q^4-q^2+1\right)} \\
 \frac{\left(q^2-1\right)^2 \overline{\gamma}}{2 \left(q^4-q^2+1\right)} & 0 & 0 & \frac{\gamma  \left(q^2-1\right)^2}{2 \left(q^4-q^2+1\right)} & \frac{\left(q^2-1\right)^2 \overline{\gamma}}{2 \left(q^4-q^2+1\right)} \\
 0 & 0 & 0 & 0 & 0 \\
 \frac{\left(q^2-1\right) \overline{\gamma}}{2 \left(q^4-q^2+1\right)} & \frac{\gamma -\gamma  q^2}{2 q^4-2 q^2+2} & \frac{\left(q^2-1\right) \overline{\gamma}}{2 \left(q^4-q^2+1\right)} & \frac{\gamma  \left(q^2-1\right)^2}{2 \left(q^4-q^2+1\right)} & 0 \\
 -\frac{\gamma  q^2 \left(q^2-1\right)}{2 \left(q^4-q^2+1\right)} & \frac{1}{2} \lambda  \left(\frac{1}{q^2}-1\right) & \frac{\gamma  \left(q^2-1\right)}{2 \left(q^4-q^2+1\right)} & 0 & \frac{\gamma  \left(q^2-1\right)^2}{2 \left(q^4-q^2+1\right)} \\
\end{array}
\right]\]
\[C:= \left[
\begin{array}{cccccc}
 0 & -\frac{q^2 \left(q^2-1\right)^2}{\left(q^2+1\right)^2 \left(q^4+1\right)} & 0 & \frac{1}{q^2}-\frac{2}{q^2+1} & 0 & \frac{1-q^2}{q^6+q^4+q^2+1} \\
 0 & -\frac{q^3 \left(q^2-1\right)^2}{\left(q^2+1\right)^2 \left(q^4+1\right)} & 0 & \frac{1-q^2}{q^3+q} & 0 & \frac{q-q^3}{q^6+q^4+q^2+1} \\
 0 & \frac{q \left(q^2-1\right)}{q^6+q^4+q^2+1} & 0 & \frac{1-q^2}{q^3+q} & 0 & -\frac{\left(q^2-1\right)^2}{\left(q^2+1\right)^2 \left(q^5+q\right)} \\
 0 & \frac{2}{\left(q^2+1\right)^2}-\frac{1}{q^4+1} & 0 & \frac{1}{q^2+1} & 0 & \frac{1}{q^4+1}-\frac{2}{\left(q^2+1\right)^2} \\
 0 & \frac{1}{q^2+1}-\frac{1}{q^4+1} & 0 & \frac{2}{q^2+1}-1 & 0 & \frac{1}{q^4+1}-\frac{2}{\left(q^2+1\right)^2} \\
 0 & q \left(\frac{2}{\left(q^2+1\right)^2}-\frac{1}{q^4+1}\right) & 0 & \frac{q}{q^2+1} & 0 & q \left(\frac{1}{q^4+1}-\frac{2}{\left(q^2+1\right)^2}\right) \\
 \frac{\gamma  q \left(q^2-2\right) \left(q^2-1\right)^2}{2 \left(q^4-q^2+1\right)} & \frac{\left(q^2-1\right)^2 \overline{\gamma}}{2 q} & 0 & 0 & \frac{\gamma  \left(q^2-1\right)^3}{2 \left(q^5-q^3+q\right)} & \frac{\left(q^2-1\right)^3 \overline{\gamma}}{2 \left(q^5-q^3+q\right)} \\
 0 & 0 & 0 & 0 & 0 & 0 \\
 \frac{\gamma  q \left(q^2-1\right)}{2 \left(q^4-q^2+1\right)} & \frac{\left(q^2-1\right) \overline{\gamma}}{2 \left(q^5-q^3+q\right)} & 0 & 0 & \frac{\gamma  \left(q^2-1\right)}{2 \left(q^5-q^3+q\right)} & \frac{\left(q^2-1\right) \overline{\gamma}}{2 \left(q^7-q^5+q^3\right)} \\
 \frac{\lambda  \left(q^2-1\right)}{2 q} & \frac{\gamma  \left(q^2-1\right)}{2 \left(q^5-q^3+q\right)} & 0 & 0 & -\frac{\lambda  \left(q^2-1\right)}{2 q} & -\frac{\gamma  \left(q^2-1\right)}{2 \left(q^5-q^3+q\right)} \\
 -\frac{\gamma  \left(q^2-1\right)^3}{2 \left(q^4-q^2+1\right)} & -\frac{\left(q^2-1\right)^3 \overline{\gamma}}{2 \left(q^4-q^2+1\right)} & 0 & 0 & \frac{\gamma  \left(q^2-1\right)^2}{2 \left(q^6-q^4+q^2\right)} & -\frac{\left(q^2-1\right)^2 \overline{\gamma}}{2 \left(q^4-q^2+1\right)} \\
 0 & 0 & 0 & 0 & 0 & 0 \\
 0 & \frac{\left(q^2-1\right)^2 \overline{\gamma}}{2 \left(q^4-q^2+1\right)} & 0 & 0 & \frac{\gamma -\gamma  q^2}{2 q^4-2 q^2+2} & \frac{\left(q^2-1\right) \overline{\gamma}}{2 \left(q^4-q^2+1\right)} \\
 \frac{1}{2} \lambda  \left(q^2-1\right)^2 & -\frac{\gamma  \left(q^2-1\right)^3}{2 \left(q^4-q^2+1\right)} & 0 & 0 & \frac{1}{2} \lambda  \left(q^2-1\right) & -\frac{\gamma  q^2 \left(q^2-1\right)}{2 \left(q^4-q^2+1\right)} \\
 \frac{\gamma  q^2 \left(q^2-1\right)}{2 \left(q^4-q^2+1\right)} & \frac{\left(q^2-1\right) \overline{\gamma}}{2 \left(q^4-q^2+1\right)} & 0 & 0 & \frac{\gamma  \left(q^2-1\right)}{2 \left(q^4-q^2+1\right)} & \frac{\left(q^2-1\right) \overline{\gamma}}{2 \left(q^6-q^4+q^2\right)} \\
 \frac{1}{2} \lambda  \left(q^2-1\right) & \frac{\gamma  \left(q^2-1\right)}{2 \left(q^4-q^2+1\right)} & 0 & 0 & -\frac{1}{2} \lambda  \left(q^2-1\right) & \frac{\gamma -\gamma  q^2}{2 q^4-2 q^2+2} \\
 \frac{\gamma  \left(q^2-1\right)}{2 \left(q^4-q^2+1\right)} & -\frac{\left(q^2-1\right) \overline{\gamma}}{2 \left(q^4-q^2+1\right)} & 0 & 0 & -\frac{\gamma  \left(q^2-1\right)^2}{2 \left(q^6-q^4+q^2\right)} & 0 \\
 -\frac{1}{2} \lambda  \left(q^2-1\right) & \frac{\gamma  q^2 \left(q^2-1\right)}{2 \left(q^4-q^2+1\right)} & 0 & 0 & 0 & -\frac{\gamma  \left(q^2-1\right)^2}{2 \left(q^6-q^4+q^2\right)} \\
 0 & \frac{\left(q^2-1\right)^2 \overline{\gamma}}{2 \left(q^5-q^3+q\right)} & 0 & 0 & -\frac{\gamma  \left(q^2-1\right)}{2 \left(q^5-q^3+q\right)} & \frac{\left(q^2-1\right) \overline{\gamma}}{2 \left(q^5-q^3+q\right)} \\
 \frac{\lambda  \left(q^2-1\right)^2}{2 q} & -\frac{\gamma  \left(q^2-1\right)^3}{2 \left(q^5-q^3+q\right)} & 0 & 0 & \frac{\lambda  \left(q^2-1\right)}{2 q} & -\frac{\gamma  q \left(q^2-1\right)}{2 \left(q^4-q^2+1\right)} \\
 -\frac{\gamma  q \left(q^2-1\right)^2}{2 \left(q^4-q^2+1\right)} & -\frac{q \left(q^2-1\right)^2 \overline{\gamma}}{2 \left(q^4-q^2+1\right)} & 0 & 0 & -\frac{\gamma  \left(q^2-1\right)^2}{2 \left(q^5-q^3+q\right)} & -\frac{\left(q^2-1\right)^2 \overline{\gamma}}{2 \left(q^5-q^3+q\right)} \\
 0 & 0 & 0 & 0 & 0 & 0 \\
 \frac{\gamma  q \left(q^2-1\right)}{2 \left(q^4-q^2+1\right)} & -\frac{q \left(q^2-1\right) \overline{\gamma}}{2 \left(q^4-q^2+1\right)} & 0 & 0 & -\frac{\gamma  \left(q^2-1\right)^2}{2 \left(q^5-q^3+q\right)} & 0 \\
 -\frac{1}{2} \lambda  (q-1) q (q+1) & \frac{\gamma  q^3 \left(q^2-1\right)}{2 \left(q^4-q^2+1\right)} & 0 & 0 & 0 & -\frac{\gamma  \left(q^2-1\right)^2}{2 \left(q^5-q^3+q\right)} \\
\end{array}
\right]\]
\normalsize
In a similar fashion, we give the matrix for $e_{2,1}$ in the block form
\[   \begin{bmatrix}
    D &E& F
\end{bmatrix}  \]
where\tiny
\[D:= \left[
\begin{array}{ccccccccccccc}
 0 & 0 & 0 & 0 & 0 & 0 & \frac{1}{\frac{1}{q^2}+1} & 0 & \frac{q^4 \left(q^2-1\right)^2}{\left(q^2+1\right)^2 \left(q^4+1\right)} & 0 & 0 & 0 & -\frac{q^3 \left(q^2-1\right)^2}{\left(q^2+1\right)^2 \left(q^4+1\right)} \\
 0 & 0 & 0 & 0 & 0 & 0 & -\frac{q}{q^2+1} & 0 & -\frac{q^3 \left(q^2-1\right)^2}{\left(q^2+1\right)^2 \left(q^4+1\right)} & 0 & 0 & 0 & \frac{q^2 \left(q^2-1\right)^2}{\left(q^2+1\right)^2 \left(q^4+1\right)} \\
 0 & 0 & 0 & 0 & 0 & 0 & 0 & 0 & -\frac{q^3 \left(q^2-1\right)^2}{\left(q^2+1\right)^2 \left(q^4+1\right)} & 0 & \frac{1}{\frac{1}{q^2}+1} & 0 & \frac{q^4 \left(q^2-1\right)}{q^6+q^4+q^2+1} \\
 0 & 0 & 0 & 0 & 0 & 0 & 0 & 0 & \frac{q^4 \left(q^2-1\right)}{q^6+q^4+q^2+1} & 0 & 0 & 0 & \frac{q^5 \left(q^2-1\right)^2}{\left(q^2+1\right)^2 \left(q^4+1\right)} \\
 0 & 0 & 0 & 0 & 0 & 0 & 0 & 0 & \frac{q^2 \left(q^2-1\right)^2}{\left(q^2+1\right)^2 \left(q^4+1\right)} & 0 & -\frac{q}{q^2+1} & 0 & q \left(\frac{1}{q^4+1}-\frac{1}{q^2+1}\right) \\
 0 & 0 & 0 & 0 & 0 & 0 & 0 & 0 & q \left(\frac{1}{q^4+1}-\frac{1}{q^2+1}\right) & 0 & 0 & 0 & -\frac{q^4 \left(q^2-1\right)^2}{\left(q^2+1\right)^2 \left(q^4+1\right)} \\
 0 & 0 & 0 & 0 & 0 & 0 & 0 & 0 & 0 & 0 & 0 & 0 & 0 \\
 1 & q & 0 & 1-\frac{1}{q^2} & 0 & q-\frac{1}{q} & 0 & 0 & -\frac{q \left(q^2-1\right)^2 \overline{\gamma}}{2 \left(q^4-q^2+1\right)} & \frac{\left(q^2-1\right)^2 \overline{\gamma}}{2 \left(q^5-q^3+q\right)} & 0 & 0 & -\frac{\left(q^2-1\right)^2 \overline{\gamma}}{2 q^2} \\
 0 & 0 & 0 & 0 & 0 & 0 & 0 & 0 & -\frac{\gamma  q \left(q^2-1\right)^2}{2 \left(q^4-q^2+1\right)} & 0 & 0 & 0 & \frac{\gamma  q^4 \left(q^2-1\right)}{2 \left(q^4-q^2+1\right)} \\
 0 & 0 & 0 & 0 & 0 & 0 & 0 & 0 & 0 & -\frac{q \left(q^2-1\right)^2 \overline{\gamma}}{2 \left(q^4-q^2+1\right)} & 0 & 0 & \frac{q^4 \left(q^2-1\right) \overline{\gamma}}{2 \left(q^4-q^2+1\right)} \\
 0 & 0 & 0 & 0 & 0 & 0 & 0 & 0 & 0 & 0 & 0 & 0 & 0 \\
 0 & 0 & 1 & q-\frac{1}{q} & q & q^2-1 & 0 & 0 & \frac{\left(q^2-1\right)^2 \overline{\gamma}}{2 \left(q^4-q^2+1\right)} & \frac{\left(q^2-1\right)^2 \overline{\gamma}}{2 \left(q^4-q^2+1\right)} & 0 & 0 & \frac{\left(q^2-1\right)^3 \overline{\gamma}}{2 \left(q^5-q^3+q\right)} \\
 0 & 0 & 0 & 0 & 0 & 0 & 0 & 0 & \frac{\gamma  q^2 \left(q^2-1\right)}{2 \left(q^4-q^2+1\right)} & \frac{\left(q^2-1\right) \lambda}{2 q^2} & 0 & 0 & \frac{\gamma  \left(q^2-1\right)^3}{2 \left(q^5-q^3+q\right)} \\
 0 & 0 & 0 & 0 & 0 & 0 & 0 & 0 & \frac{q^2 \left(q^2-1\right) \overline{\gamma}}{2 \left(q^4-q^2+1\right)} & \frac{\left(q^2-1\right) \overline{\gamma}}{2 \left(q^4-q^2+1\right)} & 0 & 0 & \frac{q \left(q^2-1\right)^2 \overline{\gamma}}{2 \left(q^4-q^2+1\right)} \\
 0 & 0 & 0 & 0 & 0 & 0 & 0 & 0 & \frac{\gamma  \left(q^2-1\right)^2}{2 \left(q^4-q^2+1\right)} & 0 & 0 & 0 & -\frac{\gamma  q^3 \left(q^2-1\right)}{2 \left(q^4-q^2+1\right)} \\
 0 & 0 & 0 & 0 & 0 & 0 & 0 & 0 & 0 & \frac{\left(q^2-1\right)^2 \overline{\gamma}}{2 \left(q^4-q^2+1\right)} & 0 & 0 & -\frac{q^3 \left(q^2-1\right) \overline{\gamma}}{2 \left(q^4-q^2+1\right)} \\
 0 & 0 & 0 & 0 & 0 & 0 & 0 & 0 & \frac{\gamma  q^2 \left(q^2-1\right)}{2 \left(q^4-q^2+1\right)} & -\frac{1}{2} \left(q^2-1\right) \lambda & 0 & 0 & -\frac{\gamma  q \left(q^2-1\right)}{2 \left(q^4-q^2+1\right)} \\
 0 & 0 & 0 & 0 & 0 & 0 & 0 & 0 & \frac{q^2 \left(q^2-1\right) \overline{\gamma}}{2 \left(q^4-q^2+1\right)} & -\frac{q^2 \left(q^2-1\right) \overline{\gamma}}{2 \left(q^4-q^2+1\right)} & 0 & 0 & \frac{q^3 \left(q^2-1\right) \overline{\gamma}}{2 \left(q^4-q^2+1\right)} \\
 0 & 0 & 0 & 0 & 0 & 0 & 0 & 0 & -\frac{\gamma  q^3 \left(q^2-1\right)}{2 \left(q^4-q^2+1\right)} & -\frac{\left(q^2-1\right) \lambda}{2 q} & 0 & 0 & -\frac{\gamma  \left(q^2-1\right)^3}{2 \left(q^4-q^2+1\right)} \\
 0 & 0 & 0 & 0 & 0 & 0 & 0 & 0 & -\frac{q^3 \left(q^2-1\right) \overline{\gamma}}{2 \left(q^4-q^2+1\right)} & -\frac{q \left(q^2-1\right) \overline{\gamma}}{2 \left(q^4-q^2+1\right)} & 0 & 0 & -\frac{q^2 \left(q^2-1\right)^2 \overline{\gamma}}{2 \left(q^4-q^2+1\right)} \\
 0 & 0 & 0 & 0 & 0 & 0 & 0 & 0 & 0 & 0 & 0 & 0 & 0 \\
 0 & 0 & 0 & 1 & 0 & q & 0 & 0 & 0 & 0 & 0 & 0 & \frac{\left(q^2-1\right)^2 \overline{\gamma}}{2 \left(q^4-q^2+1\right)} \\
 0 & 0 & 0 & 0 & 0 & 0 & 0 & 0 & -\frac{\gamma  q \left(q^2-1\right)}{2 \left(q^4-q^2+1\right)} & \frac{\left(q^2-1\right) \lambda}{2 q} & 0 & 0 & \frac{\gamma  \left(q^2-1\right)}{2 \left(q^4-q^2+1\right)} \\
 0 & 0 & 0 & 0 & 0 & 0 & 0 & 0 & -\frac{q \left(q^2-1\right) \overline{\gamma}}{2 \left(q^4-q^2+1\right)} & \frac{q \left(q^2-1\right) \overline{\gamma}}{2 \left(q^4-q^2+1\right)} & 0 & 0 & -\frac{q^2 \left(q^2-1\right) \overline{\gamma}}{2 \left(q^4-q^2+1\right)} \\
\end{array}
\right]\]
\[E:= \left[
\begin{array}{ccccc}
 0 & \frac{q^5 \left(q^2-1\right)^2}{\left(q^2+1\right)^2 \left(q^4+1\right)} & 0 & q \left(\frac{1}{q^2+1}-\frac{1}{q^4+1}\right) & 0 \\
 0 & -\frac{q^4 \left(q^2-1\right)^2}{\left(q^2+1\right)^2 \left(q^4+1\right)} & 0 & \frac{1}{q^4+1}-\frac{1}{q^2+1} & 0 \\
 0 & -\frac{q^4 \left(q^2-1\right)^2}{\left(q^2+1\right)^2 \left(q^4+1\right)} & 0 & \frac{q^4 \left(q^2-1\right)^2}{\left(q^2+1\right)^2 \left(q^4+1\right)} & 0 \\
 0 & \frac{q^5 \left(q^2-1\right)}{q^6+q^4+q^2+1} & 0 & -\frac{q^3 \left(q^2-1\right)^2}{\left(q^2+1\right)^2 \left(q^4+1\right)} & 0 \\
 0 & \frac{q^3 \left(q^2-1\right)^2}{\left(q^2+1\right)^2 \left(q^4+1\right)} & 0 & -\frac{q^3 \left(q^2-1\right)^2}{\left(q^2+1\right)^2 \left(q^4+1\right)} & 0 \\
 0 & \frac{1}{q^2+1}+\frac{q^2}{q^4+1}-1 & 0 & \frac{q^2 \left(q^2-1\right)^2}{\left(q^2+1\right)^2 \left(q^4+1\right)} & 0 \\
 0 & 0 & 0 & 0 & 0 \\
 \frac{\left(q^2-1\right)^2 \left(2 q^2-1\right) \overline{\gamma}}{2 \left(q^6-q^4+q^2\right)} & -\frac{q^2 \left(q^2-1\right)^2 \overline{\gamma}}{2 \left(q^4-q^2+1\right)} & \frac{\left(q^2-1\right)^2 \overline{\gamma}}{2 \left(q^4-q^2+1\right)} & \frac{\left(q^2-1\right)^3 \overline{\gamma}}{2 \left(q^6-q^4+q^2\right)} & \frac{\left(q^2-1\right)^3 \overline{\gamma}}{2 \left(q^6-q^4+q^2\right)} \\
 \frac{1}{2} \left(q^2-1\right) \lambda & -\frac{\gamma  q^2 \left(q^2-1\right)^2}{2 \left(q^4-q^2+1\right)} & 0 & -\frac{\gamma  q^2 \left(q^2-1\right)}{2 \left(q^4-q^2+1\right)} & -\frac{\left(q^2-1\right) \lambda}{2 q^2} \\
 \frac{q^2 \left(q^2-1\right) \overline{\gamma}}{2 \left(q^4-q^2+1\right)} & 0 & -\frac{q^2 \left(q^2-1\right)^2 \overline{\gamma}}{2 \left(q^4-q^2+1\right)} & \frac{q^4 \left(q^2-1\right) \overline{\gamma}}{2 \left(q^4-q^2+1\right)} & \frac{q^2 \left(q^2-1\right) \overline{\gamma}}{2 \left(q^4-q^2+1\right)} \\
 0 & 0 & 0 & 0 & 0 \\
 \frac{\left(q^2-1\right)^3 \overline{\gamma}}{2 \left(q^5-q^3+q\right)} & \frac{q \left(q^2-1\right)^2 \overline{\gamma}}{2 \left(q^4-q^2+1\right)} & \frac{q \left(q^2-1\right)^2 \overline{\gamma}}{2 \left(q^4-q^2+1\right)} & -\frac{\left(q^2-1\right)^2 \overline{\gamma}}{2 \left(q^5-q^3+q\right)} & \frac{q \left(q^2-1\right)^2 \overline{\gamma}}{2 \left(q^4-q^2+1\right)} \\
 \frac{\left(q^2-1\right)^2 \lambda}{2 q^3} & \frac{\gamma  q^3 \left(q^2-1\right)}{2 \left(q^4-q^2+1\right)} & \frac{\left(q^2-1\right) \lambda}{2 q} & \frac{\gamma  \left(q^2-1\right)}{2 \left(q^5-q^3+q\right)} & -\frac{\left(q^2-1\right) \lambda}{2 q^3} \\
 0 & \frac{q^3 \left(q^2-1\right) \overline{\gamma}}{2 \left(q^4-q^2+1\right)} & \frac{q \left(q^2-1\right) \overline{\gamma}}{2 \left(q^4-q^2+1\right)} & -\frac{q \left(q^2-1\right) \overline{\gamma}}{2 \left(q^4-q^2+1\right)} & \frac{q \left(q^2-1\right) \overline{\gamma}}{2 \left(q^4-q^2+1\right)} \\
 -\frac{\left(q^2-1\right) \lambda}{2 q} & \frac{\gamma  q \left(q^2-1\right)^2}{2 \left(q^4-q^2+1\right)} & 0 & \frac{\gamma  q \left(q^2-1\right)}{2 \left(q^4-q^2+1\right)} & \frac{\left(q^2-1\right) \lambda}{2 q^3} \\
 -\frac{q \left(q^2-1\right) \overline{\gamma}}{2 \left(q^4-q^2+1\right)} & 0 & \frac{q \left(q^2-1\right)^2 \overline{\gamma}}{2 \left(q^4-q^2+1\right)} & -\frac{q^3 \left(q^2-1\right) \overline{\gamma}}{2 \left(q^4-q^2+1\right)} & -\frac{q \left(q^2-1\right) \overline{\gamma}}{2 \left(q^4-q^2+1\right)} \\
 \frac{\left(q^2-1\right) \lambda}{2 q} & \frac{\gamma  q^3 \left(q^2-1\right)}{2 \left(q^4-q^2+1\right)} & -\frac{1}{2} q \left(q^2-1\right) \lambda & -\frac{\gamma  q \left(q^2-1\right)^2}{2 \left(q^4-q^2+1\right)} & 0 \\
 -\frac{q^3 \left(q^2-1\right) \overline{\gamma}}{2 \left(q^4-q^2+1\right)} & \frac{q^3 \left(q^2-1\right) \overline{\gamma}}{2 \left(q^4-q^2+1\right)} & -\frac{q^3 \left(q^2-1\right) \overline{\gamma}}{2 \left(q^4-q^2+1\right)} & 0 & -\frac{q \left(q^2-1\right)^2 \overline{\gamma}}{2 \left(q^4-q^2+1\right)} \\
 -\frac{\left(q^2-1\right)^2 \lambda}{2 q^2} & -\frac{\gamma  q^4 \left(q^2-1\right)}{2 \left(q^4-q^2+1\right)} & -\frac{1}{2} \left(q^2-1\right) \lambda & \frac{\gamma -\gamma  q^2}{2 q^4-2 q^2+2} & \frac{\left(q^2-1\right) \lambda}{2 q^2} \\
 0 & -\frac{q^4 \left(q^2-1\right) \overline{\gamma}}{2 \left(q^4-q^2+1\right)} & -\frac{q^2 \left(q^2-1\right) \overline{\gamma}}{2 \left(q^4-q^2+1\right)} & \frac{q^2 \left(q^2-1\right) \overline{\gamma}}{2 \left(q^4-q^2+1\right)} & -\frac{q^2 \left(q^2-1\right) \overline{\gamma}}{2 \left(q^4-q^2+1\right)} \\
 0 & 0 & 0 & 0 & 0 \\
 \frac{\left(q^2-1\right)^2 \overline{\gamma}}{2 \left(q^4-q^2+1\right)} & 0 & 0 & \frac{\left(q^2-1\right)^2 \overline{\gamma}}{2 \left(q^4-q^2+1\right)} & \frac{\left(q^2-1\right)^2 \overline{\gamma}}{2 \left(q^4-q^2+1\right)} \\
 -\frac{\left(q^2-1\right) \lambda}{2 q^2} & -\frac{\gamma  q^2 \left(q^2-1\right)}{2 \left(q^4-q^2+1\right)} & \frac{1}{2} \left(q^2-1\right) \lambda & \frac{\gamma  \left(q^2-1\right)^2}{2 \left(q^4-q^2+1\right)} & 0 \\
 \frac{q^2 \left(q^2-1\right) \overline{\gamma}}{2 \left(q^4-q^2+1\right)} & -\frac{q^2 \left(q^2-1\right) \overline{\gamma}}{2 \left(q^4-q^2+1\right)} & \frac{q^2 \left(q^2-1\right) \overline{\gamma}}{2 \left(q^4-q^2+1\right)} & 0 & \frac{\left(q^2-1\right)^2 \overline{\gamma}}{2 \left(q^4-q^2+1\right)} \\
\end{array}
\right]\]
\[F:= \left[
\begin{array}{cccccc}
 -\frac{q^2 \left(q^2-1\right)^2}{\left(q^2+1\right)^2 \left(q^4+1\right)} & 0 & q^2+\frac{2}{q^2+1}-2 & 0 & \frac{q^4 \left(q^2-1\right)}{q^6+q^4+q^2+1} & 0 \\
 q \left(\frac{2}{\left(q^2+1\right)^2}-\frac{1}{q^4+1}\right) & 0 & q \left(\frac{2}{q^2+1}-1\right) & 0 & q \left(\frac{1}{q^4+1}-\frac{1}{q^2+1}\right) & 0 \\
 q \left(\frac{1}{q^2+1}-\frac{1}{q^4+1}\right) & 0 & q \left(\frac{2}{q^2+1}-1\right) & 0 & \frac{q^5 \left(q^2-1\right)^2}{\left(q^2+1\right)^2 \left(q^4+1\right)} & 0 \\
 \frac{q^4 \left(q^2-1\right)^2}{\left(q^2+1\right)^2 \left(q^4+1\right)} & 0 & \frac{1}{\frac{1}{q^2}+1} & 0 & -\frac{q^4 \left(q^2-1\right)^2}{\left(q^2+1\right)^2 \left(q^4+1\right)} & 0 \\
 \frac{1}{q^4+1}-\frac{1}{q^2+1} & 0 & 1-\frac{2}{q^2+1} & 0 & -\frac{q^4 \left(q^2-1\right)^2}{\left(q^2+1\right)^2 \left(q^4+1\right)} & 0 \\
 -\frac{q^3 \left(q^2-1\right)^2}{\left(q^2+1\right)^2 \left(q^4+1\right)} & 0 & -\frac{q}{q^2+1} & 0 & \frac{q^3 \left(q^2-1\right)^2}{\left(q^2+1\right)^2 \left(q^4+1\right)} & 0 \\
 0 & 0 & 0 & 0 & 0 & 0 \\
 -\frac{\left(q^2-1\right)^2 \overline{\gamma}}{2 q^3} & \frac{\left(q^2-1\right)^2 \left(2 q^2-1\right) \overline{\gamma}}{2 \left(q^7-q^5+q^3\right)} & 0 & 0 & \frac{\left(q^2-1\right)^3 \overline{\gamma}}{2 \left(q^5-q^3+q\right)} & \frac{\left(q^2-1\right)^3 \overline{\gamma}}{2 \left(q^5-q^3+q\right)} \\
 \frac{\gamma  q^3 \left(q^2-1\right)}{2 \left(q^4-q^2+1\right)} & \frac{\left(q^2-1\right) \lambda}{2 q} & 0 & 0 & -\frac{\gamma  q^3 \left(q^2-1\right)}{2 \left(q^4-q^2+1\right)} & -\frac{\left(q^2-1\right) \lambda}{2 q} \\
 \frac{q^3 \left(q^2-1\right) \overline{\gamma}}{2 \left(q^4-q^2+1\right)} & \frac{q \left(q^2-1\right) \overline{\gamma}}{2 \left(q^4-q^2+1\right)} & 0 & 0 & \frac{q^5 \left(q^2-1\right) \overline{\gamma}}{2 \left(q^4-q^2+1\right)} & \frac{q^3 \left(q^2-1\right) \overline{\gamma}}{2 \left(q^4-q^2+1\right)} \\
 0 & 0 & 0 & 0 & 0 & 0 \\
 \frac{\left(q^2-1\right)^3 \overline{\gamma}}{2 \left(q^6-q^4+q^2\right)} & \frac{\left(q^2-1\right)^3 \overline{\gamma}}{2 \left(q^6-q^4+q^2\right)} & 0 & 0 & -\frac{\left(q^2-1\right)^2 \overline{\gamma}}{2 \left(q^4-q^2+1\right)} & \frac{q^2 \left(q^2-1\right)^2 \overline{\gamma}}{2 \left(q^4-q^2+1\right)} \\
 \frac{\gamma  \left(q^2-1\right)^3}{2 \left(q^6-q^4+q^2\right)} & \frac{\left(q^2-1\right)^2 \lambda}{2 q^4} & 0 & 0 & \frac{\gamma  \left(q^2-1\right)}{2 \left(q^4-q^2+1\right)} & -\frac{\left(q^2-1\right) \lambda}{2 q^2} \\
 \frac{\left(q^2-1\right)^2 \overline{\gamma}}{2 \left(q^4-q^2+1\right)} & 0 & 0 & 0 & -\frac{q^2 \left(q^2-1\right) \overline{\gamma}}{2 \left(q^4-q^2+1\right)} & \frac{q^2 \left(q^2-1\right) \overline{\gamma}}{2 \left(q^4-q^2+1\right)} \\
 -\frac{\gamma  q^2 \left(q^2-1\right)}{2 \left(q^4-q^2+1\right)} & -\frac{\left(q^2-1\right) \lambda}{2 q^2} & 0 & 0 & \frac{\gamma  q^2 \left(q^2-1\right)}{2 \left(q^4-q^2+1\right)} & \frac{\left(q^2-1\right) \lambda}{2 q^2} \\
 -\frac{q^2 \left(q^2-1\right) \overline{\gamma}}{2 \left(q^4-q^2+1\right)} & -\frac{\left(q^2-1\right) \overline{\gamma}}{2 \left(q^4-q^2+1\right)} & 0 & 0 & -\frac{q^4 \left(q^2-1\right) \overline{\gamma}}{2 \left(q^4-q^2+1\right)} & -\frac{q^2 \left(q^2-1\right) \overline{\gamma}}{2 \left(q^4-q^2+1\right)} \\
 \frac{\gamma -\gamma  q^2}{2 q^4-2 q^2+2} & \frac{\left(q^2-1\right) \lambda}{2 q^2} & 0 & 0 & -\frac{\gamma  q^2 \left(q^2-1\right)^2}{2 \left(q^4-q^2+1\right)} & 0 \\
 \frac{q^2 \left(q^2-1\right) \overline{\gamma}}{2 \left(q^4-q^2+1\right)} & -\frac{q^2 \left(q^2-1\right) \overline{\gamma}}{2 \left(q^4-q^2+1\right)} & 0 & 0 & 0 & -\frac{q^2 \left(q^2-1\right)^2 \overline{\gamma}}{2 \left(q^4-q^2+1\right)} \\
 -\frac{\gamma  \left(q^2-1\right)^3}{2 \left(q^5-q^3+q\right)} & -\frac{\left(q^2-1\right)^2 \lambda}{2 q^3} & 0 & 0 & -\frac{\gamma  q \left(q^2-1\right)}{2 \left(q^4-q^2+1\right)} & \frac{\left(q^2-1\right) \lambda}{2 q} \\
 -\frac{q \left(q^2-1\right)^2 \overline{\gamma}}{2 \left(q^4-q^2+1\right)} & 0 & 0 & 0 & \frac{q^3 \left(q^2-1\right) \overline{\gamma}}{2 \left(q^4-q^2+1\right)} & -\frac{q^3 \left(q^2-1\right) \overline{\gamma}}{2 \left(q^4-q^2+1\right)} \\
 0 & 0 & 0 & 0 & 0 & 0 \\
 \frac{\left(q^2-1\right)^2 \overline{\gamma}}{2 \left(q^5-q^3+q\right)} & \frac{\left(q^2-1\right)^2 \overline{\gamma}}{2 \left(q^5-q^3+q\right)} & 0 & 0 & \frac{q \left(q^2-1\right)^2 \overline{\gamma}}{2 \left(q^4-q^2+1\right)} & \frac{q \left(q^2-1\right)^2 \overline{\gamma}}{2 \left(q^4-q^2+1\right)} \\
 \frac{\gamma  \left(q^2-1\right)}{2 \left(q^5-q^3+q\right)} & -\frac{\left(q^2-1\right) \lambda}{2 q^3} & 0 & 0 & \frac{\gamma  q \left(q^2-1\right)^2}{2 \left(q^4-q^2+1\right)} & 0 \\
 -\frac{q \left(q^2-1\right) \overline{\gamma}}{2 \left(q^4-q^2+1\right)} & \frac{q \left(q^2-1\right) \overline{\gamma}}{2 \left(q^4-q^2+1\right)} & 0 & 0 & 0 & \frac{q \left(q^2-1\right)^2 \overline{\gamma}}{2 \left(q^4-q^2+1\right)} \\
\end{array}
\right]\]

\end{document}